\documentclass[a4paper, 11pt, reqno]{article}

\usepackage[english]{babel}
\usepackage{fullpage}
\usepackage{enumerate}
\usepackage{listings}
\usepackage{graphicx}
\usepackage{amsmath, dsfont, amssymb, amsthm, mathrsfs}
\usepackage[dvipsnames]{xcolor}

\usepackage{hyperref}
\hypersetup{
	colorlinks=true,
	linkcolor= RoyalBlue,
	citecolor = Maroon,
	linktocpage = True,
}

\title{\textbf{
Harnack inequality and one-endedness of UST on reversible random graphs
}}
\author{Nathana\"el Berestycki \and Diederik van Engelenburg}

  \date{University of Vienna}

\newcommand{\DE}[1]{{\color{Magenta} #1}}

\usepackage{comment}
\includecomment{comment}
\excludecomment{wrong-old}

\def\P{\mathbb{P}}

\def\E{\mathbb{E}}
\def\C{\mathbb{C}}
\def\R{\mathbb{R}}

\newcommand{\RR}{\mathbb{R}}
\newcommand{\NN}{\mathbb{N}}

\newcommand{\ZZ}{\mathbb{Z}}
\newcommand{\PP}{\mathbb{P}}
\newcommand{\cF}{\mathcal{F}}
\renewcommand{\bf}[1]{\mathbf{#1}}
\newcommand{\io}{\text{\;i.o.}}
\newcommand{\id}{\mathds{1}}
\newcommand{\cRW}{\widehat{X}}
\newcommand{\wh}[1]{\widehat{#1}}
\renewcommand{\c}[1]{\mathcal{#1}}

\DeclareMathOperator{\hm}{\mathrm{hm}}
\DeclareMathOperator{\cE}{\mathcal{E}}
\DeclareMathOperator{\cG}{\mathcal{G}}
\DeclareMathOperator{\cA}{\mathcal{A}}
\DeclareMathOperator{\cT}{\mathcal{T}}

\DeclareMathOperator{\LE}{\mathrm{LE}}
\DeclareMathOperator{\Gr}{\mathbf{G}}
\newcommand{\Beff}{\mathcal{B}_{\mathrm{eff}}}
\newcommand{\Beuc}{B_{\mathrm{euc}}}

\newcommand{\Reff}{\mathcal{R}_{\mathrm{eff}}}

\newtheorem{theorem}{Theorem}[section]
\newtheorem*{claim}{Claim}

\newtheorem{conjecture}[theorem]{Conjecture}

\newtheorem{corollary}[theorem]{Corollary}
\newtheorem{lemma}[theorem]{Lemma}
\newtheorem{proposition}[theorem]{Proposition}
\newtheorem{definition}[theorem]{Definition}
\newtheorem*{obs}{Observation}

\theoremstyle{remark}
\newtheorem{remark}[theorem]{Remark}

\renewenvironment{abstract}
{\small
	\begin{center}
	{\bfseries \abstractname\vspace{-.5em}\vspace{0pt}}
	\end{center}
	\list{}{%
		\setlength{\leftmargin}{20mm}
		\setlength{\rightmargin}{\leftmargin}%
	}%
	\item\relax}
{\endlist}
\begin{document}
	
	\maketitle
	
	\begin{abstract}
We prove that for recurrent, reversible graphs, the following conditions are equivalent: (a) existence and uniqueness of the potential kernel, (b) existence and uniqueness of harmonic measure from infinity, (c) a new anchored Harnack inequality, and (d) one-endedness of the wired Uniform Spanning Tree. In particular this gives a proof of the anchored (and in fact also elliptic) Harnack inequality on the UIPT.
This also complements and strengthens some results of Benjamini, Lyons, Peres and Schramm \cite{BLPS}. Furthermore, we make progress towards a conjecture of Aldous and Lyons by proving that these conditions are fulfilled for strictly subdiffusive recurrent unimodular graphs. Finally, we discuss the behaviour of the random walk conditioned to never return to the origin, which is well defined as a consequence of our results.

	\end{abstract}
	\tableofcontents
	\section{Introduction} \label{section: Introduction}

\subsection{Background and main result}

Let $(G,o)$ be a random unimodular rooted graph, which is almost surely recurrent (with $\E(\deg(o)) < \infty$).
 The \textbf{wired Uniform Spanning Tree} (UST for short) on $G$ is defined to be the unique weak limit of the uniform spanning tree on any finite exhaustion of the graph, with wired boundary conditions. The existence of this limit is well known, see e.g. \cite{LyonsPeresProbNetworks}. (In fact, since the graph is assumed to be recurrent, the wired or free boundary conditions give the same weak limit). The UST is a priori a spanning forest of the graph $G$, but since $G$ is recurrent this spanning forest consists in fact a.s. of a single spanning tree which we denote by $\c{T}$ (see e.g. \cite{Pemantle}). We say that $\c{T}$ is \textbf{one-ended} if the removal of any finite set of vertices $A$ does not disconnect $\c{T}$ into at least two infinite connected components. Intuitively, a one-ended tree consists of a unique semi-infinite path (the spine) to which finite bushes are attached.

The question of the one-endedness of the UST (or the components of the UST, when the graph is not assumed to be recurrent) has been the focus of intense research ever since the seminal work of Benjamini, Lyons, Peres and Schramm \cite{BLPS}. Among many other results, these authors proved (in Theorem 10.1) that on every vertex-transitive graph, and more generally on a network with a transitive unimodular automorphism group, that every component is a.s. one-ended unless the graph is itself roughly isometric to $\mathbb{Z}$ (in which case it and the UST are both two-ended). (This was extended by Lyons, Morris and Schramm \cite{LyonsMorrisSchramm2008} to graphs that are neither transitive nor unimodular but satisfy a certain isoperimetric condition slightly stronger than uniform transience). More generally, a conjecture attributed to Aldous and Lyons is that every unimodular one-ended graph is such that every component of the UST is a.s. one-ended. This has been proved in the planar case in the remarkable paper of Angel, Hutchcroft, Nachmias and Ray \cite{AHNR} (Theorem 5.16) and in the transient case by results of Hutchcroft \cite{Hutchcroft1, Hutchcroft2}. The conjecture therefore remains open in the recurrent case, which is the focus of this article.

\medskip Let us motivate further the question of the one-enededness of the UST. It can in some sense be seen as the analogue\footnote{We thank Tom Hutchcroft for this wonderful analogy.} of the question of percolation at the critical value. To see this, note that when the UST is one-ended, every edge can be oriented towards the unique end, so that following the edges forward from any given vertex $w$, we have a unique semi-infinite path starting from $w$ obtained by following the edges forward successively. Observe that this forward path necessarily eventually arrives at the spine and moves to infinity along it. Given a vertex $v$, we may define the past $\textbf{Past}(v)$ of $v$ to be the set of vertices $w$ for which the forward path from $w$ contains $v$; it is natural to view $\textbf{Past}(v)$ as the analogue of a connected component in percolation. From this point of view, the a.s. one-endedness of the tree is equivalent to the finiteness of the past (i.e., connected component in this analogy) of every vertex, as anticipated. We further note that on a unimodular graph, the expected value of the size of the past is however always infinite, as shown by a simple application of the mass transport principle. This confirms the view that the past displays properties expected from a critical percolation model. In fact, Hutchcroft proved in \cite{Hutchcroft3} that the two models have same critical exponents in sufficiently high dimension.

\medskip In this paper we give necessary and sufficient conditions for the one-endedness of the UST on a recurrent, unimodular graph. These are, respectively: (a) existence of the potential kernel, (b) existence of the harmonic measure from infinity, and finally (c) an anchored Harnack inequality. Although we do not solve the remaining case of the Aldous--Lyons conjecture, we illustrate our results by showing that they give straightforward proofs of the aforementioned result of Benjamini, Lyons, Peres and Schramm \cite{BLPS} in the recurrent case (which is one of the most difficult aspects of the proof of the whole theorem, and is in fact stated as Theorem 10.6). We also apply our results to some unimodular random graphs of interest such as the Uniform Infinite Planar Triangulation (UIPT) and related models of infinite planar maps, for which we deduce the Harnack inequality.

\medskip To state these results, we first recall the following definitions. Our results can be stated for reversible environments or \textbf{reversible random graphs}, i.e., random rooted graphs such that if $X_0$ is the root and $X_1$ the first step of the random walk conditionally given $G$ and $X_0$ then $(G, X_0, X_1)$ and $(G,X_1, X_0)$ have the same law. As noted by Benjamini and Curien in \cite{BenjaminiCurien2012}, any unimodular graph $(G,o)$ with $\E( \deg (o) ) < \infty$ satisfies this reversibility condition after biasing by the degree of $o$.  Conversely, any reversible random graph gives rise to a unimodular rooted random graph after unbiasing by the degree of the root. This biasing/unbiasing does not affect any of the results below since they are almost sure properties of the graph. Note also that again by results in \cite{BenjaminiCurien2012}, a rooted random graph whose law is stationary for random walk is in fact necessarily reversible. See also Hutchcroft and Peres \cite{HutchcroftPeres} for a nice discussion and Aldous and Lyons \cite{AldousLyonsUnimod2007} for a systematic treatment.

For a nonempty set $A \subset v(G)$ we define the \textbf{Green function} by setting for $x, y \in v(G)$:
	\begin{equation} \label{eqdef: green function}
		\Gr_A(x, y) = \E_x \left[\sum_{n = 0}^{T_A - 1} \id_{\{X_n = y\} }\right],
	\end{equation}
where $T_A$ denote the hitting time of $A$, and let
	\[
		g_A(x, y) := \frac{\Gr_A(x, y)}{\deg(y)}
	\]
denote the normalised Green function. (Note that due to reversibility, $g_A(x,y) = g_A(y,x)$.)

Let $A_n$ be any (sequence) of finite sets of vertices such that $d(A_n, o) \to \infty$ as $n \to \infty$. Here, by $d(A_n, o)$ we just mean the minimal distance of any vertex in $A_n$ to $o$. It is natural to construct the \textbf{potential kernel} of the infinite graph $G$ by an approximation procedure; we set
	\begin{equation} \label{eqdef: PKn}
		a_{A_n}(x, y) :=  g_{A_n}(y,y) - g_{A_n}(x,y).
	\end{equation}
In this manner, the potential kernel compares the number of visits to $y$, starting from $x$ versus $y$, until hitting the far away set $A_n$. We are interested in existence and uniqueness of limits for $a_{A_n}$ as $n \to \infty$. In this case we call the unique limit the potential kernel of the graph $G$. We will see that the existence and uniqueness of this potential kernel turns out to be equivalent to a number of very different looking properties of the graph.

We move on to harmonic measure from infinity.
Let $A$ be a fixed finite, nonempty set of vertices. Let $\mu_n (\cdot)$ denote the harmonic measure on $A$, started from $A_n$ if we wire all the vertices in $A_n$. The \textbf{harmonic measure from infinity}, if it exists, is the limit of $\mu_n$ (necessarily a probability measure on $A$).

Now let us turn to Harnack inequality. We say that $(G,o)$ satisfies an \textbf{(anchored) Harnack inequality} (AHI) if there exists an exhaustion $(V_R)_{R \ge 1}$ of the graph (i.e. $V_R$ is a finite subset of vertices and $\cup_{R\ge 1} V_R = v(G)$), and there exists a nonrandom constant $C>0$, such that the following holds. For every function $h: v(G) \to \R_+$ which is harmonic except possibly at 0, and such that $h(0) = 0$, then:
\begin{equation}\label{eq:AHI}
\max_{x \in \partial V_R} h(x) \le C \min_{x \in \partial V_R} h(x) .
\end{equation}
The word \emph{anchored} in this definition refers to the fact that the exhaustion is allowed to depend on the choice of root $o$, and the functions are not required to be harmonic there. (As we show in Remark \ref{R:AHIEHI}, a consequence of our results is that an anchored Harnack inequality automatically implies the Elliptic Harnack inequality (EHI) on a suitably defined sequence of growing sets.)

We now state the main theorem.

\begin{theorem}
  \label{T:main_equiv}
 Suppose $(G,o)$ is a recurrent reversible graph (or equivalently after unbiasing by the degree of the root, $(G, o) $ is recurrent and unimodular with $\E(\deg(o))< \infty$). The following properties are equivalent.
 \begin{itemize}
   \item[(a)] The pointwise limit of the truncated potential kernel $a_{A_n}(x,y)$ exists and does not depend on the choice of $A_n$.

   \item[(b)] The weak limit of the harmonic measure $\mu_n$ from $A_n$ exists and does not depend on $A_n$.

   \item[(c)] $(G,o)$ satisfies an anchored Harnack inequality.

   \item[(d)] The uniform spanning tree $\c{T}$ is a.s. one-ended.
 \end{itemize}
 Furthermore, if any of these conditions hold, a suitable exhaustion for the anchored inequality is provided by the sublevel sets of the potential kernel, see Sections \ref{section: Level-sets of the PK} and \ref{section: Harnack Inequaltiy}.
\end{theorem}

\subsection{Some applications}

\paragraph{Strengthening of \cite{BLPS}.} Before showing some applications of this result, let us point out that Theorem \ref{T:main_equiv} complements and strengthens some of the results of Benjamini, Lyons, Peres and Schramm \cite{BLPS}. In that paper, the (easy) implication (d) implies (b) was noted. We therefore in particular obtain a converse.
One can furthermore easily see using their results that on any recurrent planar graph with bounded face degrees (e.g., any recurrent triangulation) (d) holds, i.e., the uniform spanning tree is a.s. one-ended: indeed, for such a graph, there is a rough embedding from the planar dual to the primal, which is assumed to be recurrent, and therefore the planar dual must be recurrent too by Theorem 2.17 in \cite{LyonsPeresProbNetworks}. By Theorem 12.4 in \cite{BLPS} this implies that the uniform spanning tree (on the primal) is a.s. one-ended, and so (d) holds. (In fact, Theorem 5.16 in \cite{AHNR} shows that the bounded face degree assumption is not needed). 

\paragraph{Applications to planar maps.} Therefore, in combination with \cite{BLPS}, Theorem \ref{T:main_equiv} above applies in particular to unimodular, recurrent triangulations such as the UIPT, or similar maps such as the UIPQ. This therefore implies that these maps have a well-defined potential kernel, harmonic measure from infinity, and satisfy the anchored Harnack inequality. As shown in Remark \ref{R:AHIEHI}, this also implies the \textbf{elliptic Harnack inequality} (for sublevel sets of the potential kernel, see Theorem \ref{theorem: harnack inequality} for a precise statement). We point out that the elliptic Harnack inequality should not be expected to hold on usual metric balls, but can only be expected on growing sequences of sets which take into account the ``natural conformal embedding'' of these maps. This is exactly what the potential kernel and its sublevel sets allows us to do.

\paragraph{More general implications.} We already mention that the equivalence between (a) and (b) is valid more generally, for instance for any locally finite, recurrent graph. The implication (a) $\implies$ (c) to the Harnack inequality (c) is then valid under the additional assumption that the potential kernel grows to infinity (something which we can prove assuming unimodularity). We recall that (d) implies (b) is also true for deterministic graphs, as proved in \cite{BLPS}.

\begin{remark}
  Many of the arguments in this article are true for deterministic graphs. The unimodularity (or reversibility) of the graph with respect to random walk is only used in Lemma \ref{lemma: reversible graphs are delta-good}, whose main use is to show that the potential kernel, if it exists, diverges to infinity along any sequence going to infinity (see Lemma \ref{lemma: delta-good implies level-sets}). This property is used for instance in both directions of the relations between (c) and (d), since both go via (a). The unimodularity (or stationarity) is also used to prove that the walks conditioned not to return to the origin satisfy the infinite intersection property, a key aspect of the proof one-endedness. Finally this is also proved to show that if there is a bi-infinite path in the UST then it must essentially be almost space-filling, which is the other main argument of the proof of one-endedness. 
\end{remark}
\paragraph{Deterministic case of the Aldous--Lyons conjecture.}
As previously mentioned, Theorem \ref{T:main_equiv} can be applied to give a direct proof of the one-endedness of the UST for recurrent vertex-transitive graphs not roughly equivalent to $\mathbb{Z}$, which is Theorem 10.6 in \cite{BLPS}.

\begin{corollary}
  \label{Cor} Suppose $G$ is a recurrent, vertex-transitive graph. If $G$ is one-ended then the UST is also a.s. one-ended. Otherwise $G$ is roughly isometric to $\mathbb{Z}$. 
\end{corollary}

\begin{proof}
First note that the volume growth of the graph is at most polynomial (as otherwise the walk cannot be recurrent). By results of Trofimov \cite{Trofimov}, the graph is therefore roughly isometric to a Cayley graph $\Gamma$. Since it is recurrent (as recurrence is preserved under rough isometries, see Theorem 2.17 and Proposition 2.18 of \cite{LyonsPeresProbNetworks}), we deduce by a classical theorem of Varopoulos (see e.g. Theorem 1 and its corollary in \cite{Varopoulos})
that $\Gamma$ is a finite extension of $\ZZ$ or $\ZZ^2$ and is therefore (as is relatively easily checked) roughly isometric to either of these lattices. Since either of these lattices enjoy the Parabolic Harnack Inequality (PHI), which is, by a consequence of a result proved by Grigoryan \cite{Grigoryan} and Saloff-Coste \cite{SaloffCoste} independently, preserved under rough isometries (see also \cite{CoulhonSaloffCoste}), we see that $G$ itself satisfies PHI and therefore also the Elliptic Harnack Inequality (EHI): for any $R>1$, if $h$ is harmonic in the metric ball $B(2R)$ of radius $2R$ around the origin, then $\sup_{B(R)} h(x) \le C \inf_{B(R)} h(x)$. (In fact, by a deep recent result of Barlow and Murugan, EHI is now known directly to be stable under rough isometries \cite{BarlowMurugan}, but here we can appeal to the much simpler stability of PHI. {We recommend the following textbooks for related expository material: \cite{KumagaiSaintFlour}, \cite{barlow_2017} and \cite{VaropoulosSaloffCosteCoulhon}.})

Suppose that $G$ is not roughly isometric to $\ZZ$, therefore it is roughly isometric to $\ZZ^2$. Let us show that $G$ satisfies the anchored Harnack inequality \eqref{eq:AHI}, with the exhaustion sequence simply obtained by considering metric balls $V_R = B(R)$. Let $h$ be nonnegative harmonic on $G$ except at 0. Since $G$ is rough isometric to $\ZZ^2$, we can cover $\partial V_R$ with a fixed number (say $K$) of balls of radius $R/10$, such that the union of these balls is connected (here we used two-dimensionality). Let $x, y \in \partial V_R$, we can find $x = x_0, \ldots, x_K = y$ with $d (x_i, x_{i+1}) \le R/10$, and $d(x_i, o) > 2R/10$. Exploiting the EHI in each of the $K$ balls $B(x_i, 2R/10)$ inductively (since $h$ is harmonic in each of these balls), we find that $h(x) \le C^K  h(y)$. Since $x, y$ are arbitrary in $\partial V_R$, this proves the anchored Harnack inequality \eqref{eq:AHI}.
\end{proof}

\paragraph{Subdiffusivity implies one-endedness.} As an application of our results we also show that the one-endedness of the UST holds for unimodular recurrent graphs if we in addition assume that they are strictly subdiffusive; that is, we settle the Aldous--Lyons conjecture in that case.
 (This encompasses many models of random planar maps, but can of course hold on more general graphs, see in particular \cite{Lee2017}, recalled also in Remark \ref{R:lee}, for sufficient conditions guaranteeing this).

\begin{theorem} \label{T:main_existence}
  Suppose $(G,o)$ is unimodular, almost surely recurrent and strictly subdiffusive (i.e., satisfies \eqref{eqdef: subdiffusive} below) and satisfies $\bf{E}[\deg(o)] < \infty$. Then $(G,o)$ satisfies (a)--(d).
\end{theorem}

This applies e.g. for high-dimensional incipient infinite percolation cluster, as explained after Remark \ref{R:lee}. The proof of Theorem \ref{T:main_existence} takes as an input the results of Benjamini, Duminil--Copin, Kozma and Yadin \cite{benjamini2015} which shows that for strictly subdiffusive unimodular graphs there are no nonconstant harmonic functions of linear growth, and the trivial observation that the effective resistance between points is at most linear in the distance between these points. We believe it should be possible to use the same idea to prove the result assuming only diffusivity: to do this, it would suffice to prove that the effective resistance grows strictly sublinearly, except on graphs roughly isometric to $\mathbb{Z}$.

\paragraph{Random walk conditioned to avoid the origin.} The existence of the potential kernel allows us to define (by $h$-transform) a random walk conditioned to never touch a given point (even though this is of course a degenerate conditioning on recurrent graphs). We study some properties of the conditioned walk and show among other things that two independent conditioned walks must intersect infinitely often, a fact which plays an important role in the proof of Theorem \ref{T:main_equiv} for the equivalence between (a) and (d). We conclude the article with a finer study of this conditioned walk on CRT-mated random planar maps. In this case we are able to show that the hitting probability of a point far away from the origin by the conditioned walk remains bounded away from 1 in the limit as the point diverges to infinity (and is bounded away from 0 for ``almost all'' such points). See Theorem \ref{theorem: CRT hm bounds} for a precise statement. We also discuss a conjecture (see \eqref{Conj:hmbounds}) which, if true, would show a significant difference of behaviours with respect to the more standard case of $\ZZ^2$ (where these hitting probabilities converge to 1/2, as surprisingly shown in \cite{PopovRW}).

\paragraph{Acknowledgements.} The authors are indebted to Tom Hutchcroft for inspiring conversations, numerous comments on a draft of this paper, and additional references. We are also grateful to Nicolas Curien for some useful comments on a draft.

The work of N.B. is supported by: University of
Vienna start-up grant, and FWF grant P33083 on “Scaling limits in random conformal geometry”.

\begin{wrong-old}
	\DE{SMALL READING GUIDE.
		
		Section \ref{section: Preliminaries} just contains some definitions and small results.
		
		Sections \ref{section: well-definition of the PK's} and \ref{section: Gluing, Capacity and Harmonic measure} show that the potential kernel is well defined (that is, exists and is unique) and links the potential kernel to the harmonic measure from infinity. They also provide some `elementary' properties of the potential kernel. One can skip reading Section \ref{section: capacity of SRW}, the remainder of the text does not rely on it.
		
		Section \ref{section: Level-sets of the PK} finishes the `basic' picture of the potential kernel.
		
		Section \ref{section: Harnack Inequaltiy} provides the Harnack inequality and similar results.
	
		Section \ref{section: The CRW} provides the definition of the Random Walk conditioned to not hit the root and shows some first properties of this walk. It also links the harmonic measure from infinity and hitting probabilities of this conditioned walk.
	
		Section \ref{section: examples} provides some basic examples and concludes with an elementary proof of the fact that, on transitive locally finite, one-ended, recurrent networks the potential kernel is well defined.
	}
\end{wrong-old}
	
	%
	\section{Background and notation} \label{section: Preliminaries}
	
Before we begin with the proofs of our theorems, we need to introduce the main notations that we will use throughout this text.
	
A \textbf{graph} $G$ consists of a countable collection of vertices $v(G)$ and edges $e(G) \subset \{\{x, y\}: x, y \in v(G)\}$ and we will always assume that the vertex degrees are finite. We will work with undirected graphs, but will sometimes take the directed edges $\vec{e}(G) = \{(x, y): \{x, y\} \in e(G)\}$.
	
	
	
	The graph $G$ comes with a natural metric $d(x, y)$, which is the \textbf{graph distance}, i.e. the minimal length of a path between two vertices $x$ and $y$. For $n \in \NN$, we will denote by
	\[
		B(y, n) = \{x \in v(G): d(x, y) \leq n\},
	\]
	the \textbf{metric ball} of radius $n$. For a set $A \subset v(G)$, we will write $\partial A$ for its outer boundary in $v(G)$, that is
	\[
		\partial A = \{x \in v(G) \setminus A: \text{ there exists a } y \in A \text{ with } x \sim y\}.
	\]
We will make extensive use of the graph \textbf{Laplacian} which we normalise as follows: 
	\begin{equation}\label{D:Laplacian}
		\Delta f(x) = \sum_{y \sim x} c(x, y)(f(y) - f(x)),
	\end{equation}
	for functions $f:v(G) \to \RR$ (here $c(x,y)$ is the conductance of the edge $(x,y)$, which is typically equal to one in this paper, except in Section \ref{section: The CRW} where we consider random walk conditioned to avoid the origin forever). 
A function $h:v(G) \to \RR$ is called \textbf{harmonic} at $x$ if $(\Delta h)(x) = 0$. 

Let $X = (X_n)_{n \geq 0}$ denote the simple random walk on $G$, with its law written as $\PP$ and $\PP_x$ to mean $\PP(\cdot \mid X_0 = x)$. For a set $A \subset v(G)$, we define the \textbf{hitting time} $T_A = \inf\{n \geq 0: X_n \in A\}$ and $T_x := T_{\{x\}}$ whenever $A=\{x\}$ consists of just one element. 
We will write $T_{A}^+$ for the \textbf{first return time} to a set $A$.
	Suppose that $G$ is a connected graph. The \textbf{effective resistance} is defined through
	\[
		\Reff(x \leftrightarrow y) := \frac{\Gr_x(y, y)}{\deg(y)}
	\]
	(recall our normalisation of the Green function $\Gr$ in \eqref{eqdef: green function}). Recall the useful identity
		\begin{equation}\label{eq:efres through flow}
			\Reff(x \leftrightarrow y) = \frac{1}{\deg(y) \PP_{y}(T_x < T_{y}^+)}
		\end{equation}
		The proof is obvious from the definition of effective resistance and our normalisation of the Green function when we use the obvious identity
		\[
			\Gr_x(y, y) = \frac{1}{\PP_y(T_x < T_y^{+})},
		\]
		which can be seen by considering the number of excursions from $y$ to $y$, which is a geometric random variable by the Markov property.

	For infinite graphs $G$, we will say that a sequence of subgraphs $(G_n)_{n \geq 1}$ of $G$ is an \textbf{exhaustion} of $G$ whenever $G_n$ is finite for each $n$ and $v(G_n) \to G$ as $n \to \infty$. Fix some exhaustion $(G_n)_{n \geq 1}$ of an infinite graph $G$ and define the graph $G_n^*$ as $G_n$, together with the identification of $G_n^c$, where we have deleted all self-loops created in the process. For two vertices $x, y \in v(G)$ we recall that
	\[
		\Reff(x \leftrightarrow y) = \lim_{n \to \infty} \Reff(x \leftrightarrow y; G^*_n),
	\]
	see for instance \cite[Section 9.1]{LyonsPeresProbNetworks}. 
As is well known, the effective resistance defines a metric (see for instance exercise 2.67 in \cite{LyonsPeresProbNetworks}).

	Later, we will often work with the metric $\Reff(\cdot \leftrightarrow \cdot)$ on $v(G)$, instead of the standard graph distance. We introduce the notation
	\begin{equation} \label{eqdef: Beff}
		\Beff(x, R) = \{y \in v(G): \Reff(x \leftrightarrow y) \leq R\}
	\end{equation}
	for the closed ball with respect to the effective resistance metric. Notice that, in general, this metric space is \emph{not} a length space - making it somewhat inconvenient.
	
	Another result that we will need to use a few times is the `last exit decomposition', or rather two versions thereof 
which can be proved similarly to \cite[Proposition 4.6.4]{LawlerLimicRandomWalks}. 
	
	\begin{lemma}[Last Exit Decomposition] \label{lemma: last exit decomposition}
		Let $G$ be a graph and $A \subset B \subset v(G)$ finite. Then for all $x \in A$ and $b \in \partial B$ we have
		\[
			\PP_x(X_{T_{B^c}} = b) = \sum_{z \in A} \Gr_{B^c}(x, z)\PP_z(T_{B^c} < T_A^+, X_{T_{B^c}} = b).
		\]
		Moreover, for $x \in B$ we have
		\[
			\PP_x(T_A < T_B) = \sum_{z \in A} \Gr_{B^c}(x, z)\PP_z(T_B < T_A^+)
		\]
	\end{lemma}
\begin{wrong-old}
	\begin{proof} We adapt the proof in \cite[Proposition 4.6.4]{LawlerLimicRandomWalks} to work in our setting, although this is rather straightforward.
		Fix $A \subset B \subset v(G)$ finite and $x \in A$, $b \in \partial B$. Denote $\sigma$ the last time the random walk started from $x$ and killed when hitting $B^c$ will visit $\partial^{in} A$. Then by the Markov Property for simple random walk, we have
		\begin{align*}
			\PP_x(X_{T_{B^c}} = b) &= \sum_{z \in A} \sum_{k = 0}^\infty \PP_x(\sigma = k; X_{\sigma} = z, X_{T_{B^c}} = b) \\
			&= \sum_{z \in A} \sum_{k=0}^{\infty} \PP_x(X_{k} = z, k < T_{B^c})\PP_z(T_A^+ < T_{B^c}; X_{T_{B^c}} = b),
		\end{align*}
		showing the first assertion. The second assertion is similar.
	\end{proof}

\end{wrong-old}

\begin{wrong-old}
	\subsection{Reversible environments}
	The third and last topic we introduce before starting our proofs, is that of reversible random graphs. A \textbf{rooted graph} $(G, o)$ is a graph $G$, together with a distinguished vertex $o \in v(G)$. Two rooted graphs $(G, o)$ and $(G', o')$ are isomorphic, written $(G, o) \simeq (G', o')$ whenever there exists a bijective map $\phi:v(G) \to v(G')$ such that $\phi(o) = o'$ and such that
	\[
		\{ \{\phi(x), \phi(y)\}: \{x, y\} \in e(G)\} = e(G').
	\]
	We can then define $\cG_{\bullet}$ to be the space of all equivalence classes of rooted, connected graphs. As usual we will simply write $(G, o)$ for a rooted graph, as well as for the equivalence class to which it belongs.
	
	Similarly, we can define the space of \textbf{doubly rooted graphs} $(G, o, \rho)$ denoted by $\cG_{\bullet \bullet}$. We equip the space of rooted graphs $\cG_{\bullet}$ with the local topology, which was first introduced in \cite{BenjaminiSchrammRecurrence} and which comes from the metric
	\[
		d_{\mathrm{loc}}((G, o), (G', o')) = \frac{1}{1 + r}
	\]
	where $r$ is the largest value for which $(B(o, r), o)$ and $(B(o', r), o')$ are isomorphic (here, the graph balls are taken on the respective graphs $G$ and $G'$). A random rooted graph is then a random variable taking values in the space $\cG_{\bullet}$ with the Borel sets as $\sigma$-algebra. We will write $\bf{P}$ and $\bf{E}$ for the law respectively expected value of random rooted graphs.
	
	A \textbf{mass transport} is a function $f: \cG_{\bullet \bullet} \to [0, \infty]$. We call a random rooted graph \textbf{unimodular}  whenever it satisfies the mass transport principle
	\[
		\bf{E} \left[\sum_{x \in v(G)} f(G, o, x)\right] = \bf{E} \left[ \sum_{x \in v(G)} f(G, x, o) \right],
	\]
	which says so much as ``expected mass in equals expected mass out''. This goes back again to \cite{BenjaminiSchrammRecurrence}, and was later developed in \cite{AldousLyonsUnimod2007}. The random graph $(G, o)$ is called \textbf{reversible}, or a reversible environment, whenever $(G, o, X_1)$ has the same law as $(G, X_1, o)$, where $(X_n)_{n \geq 0}$ is the simple random walk on $G$, started from the root $o$. The random rooted graph $(G, o)$ is called \textbf{stationary} whenever $(G, X_n)$ has the same law as $(G, o)$. Note that reversibility implies stationarity, but the converse is not true in general. In the recurrent case, however, the two notions are equivalent \cite[Theroem 4.3]{BenjaminiCurien2012}.
	
	Suppose for a second that $G$ is a finite graph, with $o$ simply being uniformly chosen over its vertices, then $(G, o)$ is unimodular. On the other hand, if we bias the law of $o$ by the degree, then the random walk is stationary. This idea generalizes to the infinite setting, with a minor regularity assumption, which is well known and called ``degree biasing''. If $(G, o)$ is a unimodular random graph with $\bf{E}[\deg(o)] < \infty$, then degree biasing it's law gives a reversible environment. Similarly, if $(G, o)$ is reversible, then biasing by $\deg(o)^{-1}$ gives a unimodular random graph.
	
\end{wrong-old}
	%
	\section{Equivalence between (a) and (b)} \label{section: well-definition of the PK's}

\subsection{Base case of equivalence}	
	We will say that a sequence of finite sets of vertices $(A_n)_{n \geq 1}$ `goes to infinty' whenever $d(A_n, o) \to \infty$ as $n \to \infty$. Here, by $d(A_n, o)$ we just mean the minimal distance of any vertex in $A_n$ to $o$. Recall the definition of $a_{A_n}$, which also satisfies
	\begin{equation} \label{eqdef: PK along An}
		a_{A_n}(x, y) = g_{A_n}(y, y) - g_{A_n}(x, y) = \frac{1}{\deg (y)} \frac{\PP_x (T_{A_n} < T_y)}{  \PP_y ( T_{A_n} < T_{y}^+) }.
	\end{equation}
	Clearly, both the numerator and the denominator tend to $0$ as $n$ tends to infinity by recurrence of the underlying graph $G$. When a sequence of subsets $A_n$ has be chosen we will write $a_n$ instead of $a_{A_n}$ with a small abuse of notations.
	
	The goal of this section is to prove the equivalence between (a) and (b) in Theorem \ref{T:main_equiv} (in the base case where the set $A$ consists of two points; this will be extended to arbitrary finite sets in Section \ref{section: Gluing, Capacity and Harmonic measure}). 
		First, we show that subsequential limits of $a_n$ always exist.
	
	\begin{lemma} \label{lem: PK exists}
		Let $(A_n)_{n \geq 1}$ be some sequence of finite sets of vertices going to infinity. There exists a subsequence $(n_k)_{k \geq 1}$ going to infinity such that for all $x, y \in v(G)$ the limit
		\[
			a(x,y): = \lim_{k \to \infty} a_{{n_k}}(x,  y)
		\]
		exists in $[0, \infty)$. Moreover, $a(x,y) > 0$ precisely when the removal of $y$ from $G$ does not disconnect $x$ from $A_{n_k}$ for all $k$ large enough.
	\end{lemma}
	\begin{proof}
		Fix $y \in v(G)$ and suppose first that for $u \sim y$ we have $\PP_u(T_{A_n} < T_y) > 0$ for all $n$ large enough (i.e., $y$ does not disconnect a portion of the graph from infinity).
		
		Let $x \in v(G)$ and fix $n$ so large that $A_n$ does not contain $y, x$ or any of the neighbors of $y$. For each $u \sim y$, we can force the random walk started from $x$ to go through $u$ before touching $A_n$ or $y$ to get
		\begin{equation} \label{subeq: forcing the RW through u}
			\PP_x(T_{A_n} < T_y) \geq \PP_x(T_u < T_{A_n} \wedge T_y)\PP_u(T_{A_n} < T_y).
		\end{equation}
		Upon taking $u \sim y$ such that it maximizes $\PP_u(T_{A_n} < T_y)$ and by recurrence of $G$ we get the existence of $c(x, y) > 0$ for which
		\[
			a_{n}(x, y) = \frac{\PP_x(T_{A_n} < T_y)}{\sum_{u \sim y} \PP_u(T_{A_n} < T_y)} \geq \frac{\PP_x(T_u < T_{A_n} \wedge T_y)}{\deg(y)} \geq c(x, y) >0.
		\]
		The same reasoning as in \eqref{subeq: forcing the RW through u} but in the other direction gives
		\[
			\PP_x(T_{A_n} < T_y) \leq \frac{\PP_u(T_{A_n} < T_y)}{\PP_u(T_x < T_{A_n} \wedge T_y)}.
		\]
		Hence, using again recurrence of $G$ we get that there is some $C(x, y) < \infty$ such that (upon taking the right $u$)
		\[	
			a_{n}(x, y) \leq \frac{\PP_u(T_{A_n} < T_y)}{\PP_u(T_x < T_{A_n} \wedge T_y)\sum_{u \sim y}\PP_u(T_{A_n} < T_y)} \leq C(x, y) < \infty.
		\]
		We deduce that for fixed $x, y$, subsequential limits of $a_{n}(x, y)$ exist and the existence of subsequential limits for all $x, y$ simultaneously follows from diagonal extraction.
		
		The existence of subsequential limits in the general case is the same as we can always lower bound $a_n(x, y)$ by $0$ and the upper bound does not change.
		
		Now, if $x \in v(G)$ is such that the removal of $y$ disconnects $x$ from $A_{n_k}$, then $a_{n_k}(x, y) = 0$. Suppose thus that $x$ is such that the removal of $y$ does not disconnect $x$ from $A_{n_k}$ for all $k$ large enough. In this case, we can restrict ourselves to just the component of $G$ with $y$ removed, in which both $A_{n_k}$ and $x$ are as the hitting probabilities are the same in this case. Hence, we are back in the situation above and $a_{n_k}(x, y) \geq c(x, y) > 0$.
	\end{proof}
	
	We next present a result, which shows that any subsequential limit appearing in Lemma \ref{lem: PK exists} must satisfy a certain number of properties.
	
	\begin{proposition}\label{prop: properties of the PKs}
		Let $a(x,y)$ be any subsequential limit as in Lemma \ref{lem: PK exists}. Then $a:v(G) \to \RR_+$ satisfies
		\begin{enumerate}[(i)]
			\item for each $y \in v(G)$
			\[
				\Delta a(\cdot, y) = \delta_y(\cdot) \qquad \text{ and } \qquad a(y, y) = 0,
			\]
			where we recall that $\Delta$ is defined in \eqref{D:Laplacian} and is normalised so that $\Delta f(x) = \sum_y (f(y) - f(x))$.
			\item for all $x, y \in v(G)$ we have
			\[
				a(x, y) = \lim_{k \to \infty}\PP_{A_{n_k}}(T_x < T_y)\Reff(x \leftrightarrow y),
			\]
			where $\PP_A$ refers to the law of a random walk starting from $A$, when all of the vertices in $A$ have been wired together.
		\end{enumerate}
	\end{proposition}

	The equivalence between (a) and (b) of Theorem \ref{T:main_equiv} (in the base case where the finite set $B$ on which we need to define harmonic measure consists of two points) is then obvious, and we collect it here:
	
	\begin{corollary} \label{cor: the harmonic measure for two points is well-defined}
		Let $G$ be a recurrent graph. Then
		\[
			\hm_{x, y}(x) := \lim_{n \to \infty} \PP_{A_n}(T_x < T_y)
		\]
		exists for all $x, y \in v(G)$ and is independent of the sequence $(A_n)_n$ if and only if the potential kernel is uniquely defined. Furthermore, in this case,
		\[
			a(x, y) = \hm_{x, y}(x)\Reff(x \leftrightarrow y).
		\]
	\end{corollary}
	
	\begin{proof}[Proof iof Proposition \ref{prop: properties of the PKs}]
		The proof of item (i) is rather elementary. Fix $y \in v(G)$ and $n \ge 1$. Since $x \mapsto \PP_x(T_{A_{n}} < T_y)$ is a harmonic function outside of $y$ and $A_{n}$ by the simple Markov property, we get that $x \mapsto a_{n}(x, y)$ is harmonic outside $y$ and $A_{n}$, see \eqref{eqdef: PK along An}. It follows that $x \mapsto a(x, y)$ is harmonic at least away from $y$.
		Furthermore, note that $a_{n}(y, y) = 0$ by definition and
		\[
			\sum_{u \sim y} a_{n}(u, y)  = \frac{\sum_{u \sim y} \PP_u(T_{A_{n}} < T_y)}{\sum_{u \sim y} \PP_u(T_{A_{n}} < T_y)} = 1
		\]
		so $\Delta a_n( \cdot, y )|_{\cdot =y } = 1$. This finishes the proof of (i).
		
		For part (ii), we notice first that by properties of the electrical resistance,
		\[
			\sum_{u \sim y} \PP_u(T_{A_{n}} < T_y) = \deg(y)\PP_y( T_{A_{n}} < T_y^+) = \frac{1}{\Reff(y \leftrightarrow A_{n})},
		\]
		which allows us to write
		\begin{equation} \label{eq: a_n = Reff()*Prob}
			a_{n}(x, y) = \Reff(y \leftrightarrow A_{n}) \PP_x(T_{A_{n}} < T_y).
		\end{equation}
		Identify the vertices in $A_{n}$ and delete possible self-loops created in the process. The resulting graph $G'_{n}$ is then still recurrent. 
		Let $\Gr_y(\cdot, \cdot)$ denote the Green function on this graph when the walk is killed at $y$. We can also express the effective resistance in terms of the normalised Green function: that is,
		$$
			\Reff (y \leftrightarrow A_n) = \frac{\Gr_y (A_n, A_n)}{\deg(A_n)}
		$$
		Using the Markov property and since $G'_n$ is reversible,
		\begin{align}
			a_{n}(x, y) &= \PP_x(T_{A_n} < T_y) \frac{\Gr_y(A_{n}, A_{n})}{\deg(A_{n})} \nonumber \\
			& = \frac{\Gr_{y}(x, A_{n})}{\deg(A_{n})}\nonumber \\
			& = \frac{\Gr_y (A_n, x)}{\deg(x)}\label{E:pkGreen}\\
			& =		\PP_{A_{n}}(T_x < T_y)\Reff(x \leftrightarrow y; G'_n)\nonumber
		\end{align}
		by using the same argument in the other direction, and	where the effective resistance in the last line is calculated in $G'_n$.
		
		Since the graph $G$ is recurrent, it follows that $\Reff(x \leftrightarrow y; G'_n)$ converges to $\Reff(x \leftrightarrow y; G)$ as $n \to \infty$ (as the free and wired effective resistances agree). We deduce that
		\[
			a(x, y) = \lim_{k \to \infty} \PP_{A_{n_k}}(T_x < T_y)\Reff(x \leftrightarrow y),
		\]
		which finishes part (ii).
	\end{proof}


	\subsection{Triangle inequality for the potential kernel}

Before we start of the proof of the remaining implications, we need some preliminary estimates on the potential kernel, showing that it satisfies a form of triangle inequality. This plays a crucial role throughout the rest of this paper.
We also need a decomposition of the potential kernel in order to prove that for reversible graphs, the potential kernel (if it is well defined) satisfies the growth condition.

	We start with a simple and well known application of the optional stopping theorem:

	\begin{lemma} \label{lem: green function finite set and PK}
		Let $A$ be some finite set and suppose that $x, y \in A$. Then
		\[
			\frac{\Gr_{A^c}(x, y)}{\deg(y)} = \E_x[a(X_{T_{A^c}}, y)] - a(x, y).
		\]
	\end{lemma}
	
	\begin{proof}
	This is Proposition 4.6.2 in \cite{LawlerLimicRandomWalks}, but we include for completeness since its proof if simple.
		Let $x, y \in A$ and notice that
		\[
		M_n := a(X_n, y) - \sum_{j=0}^{n - 1} \frac{\delta_y(X_j)}{\deg(y)}
		\]
		is a martingale. Applying the optional stopping theorem at $T_{A^c} \wedge n$, we obtain
		\[
		a(x, y) = \E_{x}[M_0] = \E_x[a(X_{n \wedge T_{A^c}}, y)] - \frac{1}{\deg(y)}\E_x \left[\sum_{j=0}^{(n \wedge T_{A^c}) - 1} \delta_y({X_j })\right],
		\]
		Taking $n \to \infty$, since $A$ is finite, we deduce from dominated (resp. monotone) convergence that
		\[
		\E_x[a(X_{n \wedge T_{A^c}}, y)] \to \E_x[a(X_{T_{A^c}}, y)], \frac{1}{\deg(y)}\E_x \left[\sum_{j=0}^{(n \wedge T_{A^c}) - 1}  \delta_y(X_j)\right] \to \frac{\Gr_{A^c}(x, y)}{\deg(y)},
		\]
		showing the result.
	\end{proof}

	\begin{proposition} \label{prop: green function and PK}
		Let $x, y, z \in v(G)$ be three vertices. We have the identity
		\[
			\frac{\Gr_z(x, y)}{\deg(y)} = a(x, z) - a(x, y) + a(z, y).
		\]
	\end{proposition}

	\begin{proof}
		Fix $x, y, z \in v(G)$ and let $(A_n)_{n \geq 1}$ be some sequence of finite sets of vertices going to infinity\footnote{Although the simplest proof of this fact is by picking a `nice' sequence $(A_n)_{n \geq 1}$ (since we already know that the potential kernel does not depend on this choice), we present a proof that works for all sequences. See also Remark \ref{remark: prop Green and PK for subsequential limits}}. Glue together $A_n$ on the one hand, and the vertices of $B(o, m)^c$ on the other hand.  Delete all self-loops created in the process and write $\partial_m$ for the vertex corresponding to $B(o, m)^c$. Let $\tilde X_k$ be the simple random walk on the graph obtained from gluing $A_n$ and $\partial_m$. We define for $w, w' \in B(o, m) \cup \{\partial_m\}$ the function
		\[
			a_{n, m}(w, w') = \Reff(\{\partial_m, w'\} \leftrightarrow A_n)\PP_{w}(T_{A_n} < T_{w'} \wedge T_{\partial_m}).
		\]
		By recurrence and \eqref{eq: a_n = Reff()*Prob}, we have that $a_{n, m}(w, w') \to a_n(w, w')$ as $m \to \infty$, for all $w, w'$.
		
		Fix $n$ so large that $x, y$ and $z$ are not in $A_n$. Let $m$ be so large that $x, y, z$ and $A_n$ are in $B(o, m)$. Define $E_{n, m} = \{A_n, z, \partial_m\}$. Then, as in Lemma \ref{lem: green function finite set and PK},
		\begin{equation} \label{subeq: green function and PK 1}
			a_{m, n}(x, y)  
= \E_x[a_{m, n}(\tilde X_{T_{E_{m, n}}}, y)] - \frac{\Gr_{E_{m,n}}(x,y)}{\deg(y)}
		\end{equation}
On the other hand, by definition of $E_{m,n}$ we have
		\begin{align*}
			\E_{x}[a_{m, n}(X_{T_{E_{m, n}}}, y)] &= \PP_x(T_z < T_{A_n} \wedge T_{\partial_m})a_{m, n}(z, y) \\
			&\quad + \PP_x(T_{A_n} < T_z \wedge T_{\partial_m})a_{m, n}(A_n, y) \\
			&\quad + \PP_x(T_{\partial_m} < T_{A_n} \wedge T_z)a_{m, n}(\partial_m, y),
		\end{align*}
		where a priori the hitting probabilities are calculated on the graph where $A_n$ and $\partial_m$ are glued. However, as we are only interested in the first hitting time of either of these sets, it does not matter and we can calculate the probabilities also for the random walk on the graph $G$. Notice that, by definition, $a_{m, n}(\partial_m, y) = 0$. Plugging this back into \eqref{subeq: green function and PK 1} we obtain
		\[
			a_{m, n}(x, y) = \PP_x(T_z < T_{A_n} \wedge T_{\partial_m})a_{m, n}(z, y) + \PP_x(T_{A_n} < T_z \wedge T_{\partial_m})a_{m, n}(A_n, y) - \frac{\Gr_{E_{m, n}}(x, y)}{\deg(y)}.
		\]
		We have already observed that $a_{m, n}(w, y) \to a_n(w, y)$ for each $w$ as $m \to \infty$. Then, by recurrence of $G$ and monotone convergence, we get
		\begin{equation} \label{subeq: green function and PK 2}
			a_n(x, y) = \PP_x(T_z < T_{A_n})a_{n}(z, y) + \PP_x(T_{A_n} < T_z)a_{n}(A_n, y) - \frac{\Gr_{\{A_n, z\}}(x, y)}{\deg(y)}.
		\end{equation}
		Next, we wish to take $n \to \infty$. The left-hand side converges to $a(x, y)$ as $n \to \infty$, by definition of the potential kernel. The first term on the right-hand side converges to $a(z, y)$ by the same argument and recurrence of the graph $G$. Using once more monotone convergence, we find
		\begin{equation} \label{subeq: green function and PK 3}
			\frac{\Gr_{\{A_n, z\}}(x, y)}{\deg(y)} \to \frac{\Gr_z(x, y)}{\deg(y)}
		\end{equation}
		as $n$ goes to infinity. We are left to deal with the term $\PP_x(T_{A_n} < T_z)a_{n}(A_n, y)$, which we claim converges to $a(x, z)$.
		
		From the definition of $a_n$, together with the representation in \eqref{eq: a_n = Reff()*Prob}, we find
		\[
			a_{n}(A_n, y) = \Reff(y \leftrightarrow A_n)\PP_{A_n}(T_{A_n} < T_y) = \Reff(y \leftrightarrow A_n).
		\]
		Thus, using again the same representation of $a_n(x, z)$, we see that
		\begin{align*}
			a_n(A_n, y) \PP_x(T_{A_n} < T_z) &= \Reff(y \leftrightarrow A_n)\PP_x(T_{A_n} < T_z) \\
			&= a_{n}(x, z)\frac{\Reff(y \leftrightarrow A_n)}{\Reff(z \leftrightarrow A_n)}.
		\end{align*}
		Using recurrence of $G$, we notice that
		\[
			\frac{\Reff(y \leftrightarrow A_n)}{\Reff(z \leftrightarrow A_n)} \to 1
		\]
		as $n \to \infty$. In particular, we deduce that
		\[
			a_n(A_n, y)\PP_x(T_{A_n} < T_z) \to a(x, z)
		\]
		as $n \to \infty$. Plugging this, together with \eqref{subeq: green function and PK 3} back into \eqref{subeq: green function and PK 2} we conclude:
		\[
			a(x, y) = a(x, z) + a(z, y) - \frac{\Gr_z(x, y)}{\deg(y)}
		\]
as desired.		
	\end{proof}

	\begin{remark}
		Proposition \ref{prop: green function and PK} is an extensions of results known for the lattice $\ZZ^2$, see Proposition 4.6.3 in \cite{LawlerLimicRandomWalks} and the discussion thereafter. As far as we know, these proofs are based on precise asymptotic behavior of the potential kernel, a tool we do not seem to have.
	\end{remark}

	\begin{remark} \label{remark: prop Green and PK for subsequential limits}
		The statement of Proposition $\ref{prop: green function and PK}$ is also valid for an arbitrary subsequential limit $a(\cdot, \cdot)$ of $a_n (\cdot, \cdot)$, even when a proper limit is not known to exist. In particular, it shows that given such a subsequential limit $a( \cdot, y)$ there is a unique way to coherently define $a(\cdot, z)$. For this reason, if $\lim_{n \to \infty} a_n(x, y)$ is shown to exist for a fixed $y$ and \emph{all} $x \in v(G)$, it follows that this limit exists for \emph{all} $x, y \in v(G)$ simultaneously. This will be used in Theorem \ref{theorem: gluing expression}.
	\end{remark}

	\begin{corollary} \label{cor: a(x, z) - a(y, z) goes to 0 as z to infinity}
		For each $x, z \in v(G)$ and all $\epsilon > 0$ there exists an $N = N(\epsilon, x, z)$ such that for all $y$ with $d(x, y) \geq N$ we have
		\[
			|a(x, y) - a(z, y)| \leq \epsilon
		\]
		and in particular $\lim_{n \to \infty} a(x, y_n) - a(z, y_n) = 0$ for any sequence $(y_n)_{n \geq 1}$ going to infinity.
	\end{corollary}
	Notice that Corollary \ref{cor: a(x, z) - a(y, z) goes to 0 as z to infinity} does \emph{not} say that $a(y_n, x) - a(y_n, z) \to 0$ as $z \to \infty$ in general! Indeed, a similar argument shows that $a(y_n, x) - a(y_n, z) \to a(z, x) - a(x, z)$, which is nonzero in general.
	\begin{proof}
		Fix $x, z \in v(G)$ and suppose by contradiction that there is some $\epsilon > 0$, such that for infinitely many $n \geq 1$ (but in fact we can with a small abuse of notation assume for all $n\ge 1$ after taking a subsequence), there is some $y_n$ with $d(x, y_n) \geq n$ for which
		\[
			|a(z, y_n) - a(x, y_n)| > \epsilon.
		\]
		By Proposition \ref{prop: green function and PK} and $\deg(\cdot)$-reversibility of the Simple Random Walk we have
$$
a(x, y_n) - a(z,y_n) = a(x, z) - \frac{\Gr_z (x, y_n)}{\deg (y_n)} =  a(x,z) - \frac{\Gr_z (y_n,x)}{\deg(x)}.
$$
		Take $A_n = \{y_n\}$ and recall (see e.g. \eqref{E:pkGreen}) that
		\[
			a_n(x, z) = \frac{\Gr_z(y_n, x)}{\deg(x)}.
		\]
Therefore
$$
a(x, y_n) - a(z, y_n) = a(x,z) - a_n(x, z).
$$
Since this converges to zero as $n \to \infty$, we get the desired contradiction.
	\end{proof}

We immediately deduce that the harmonic measures from infinity of $\{x, y\}$ and $\{z, y\}$ are very similar if $y$ is far away from $x$ and $z$.

	\begin{corollary} \label{cor: harmonic measure are local, not global}
		Fix $x, z \in v(G)$. For every $\epsilon > 0$, there exists an $N = N(x, z, \epsilon)$ such that for all $y$ with $d(x, y) \geq N$ we have
		\[
			|\hm_{x, y}(x) - \hm_{z,y}(z)| \leq \frac{\epsilon + \Reff(z \leftrightarrow x)}{\Reff(x \leftrightarrow y)}.
		\]
	\end{corollary}
	\begin{proof}
		This follows from Corollary \ref{cor: a(x, z) - a(y, z) goes to 0 as z to infinity} and using the expression
		\[
			a(x, z) = \hm_{x, z}(z)\Reff(x \leftrightarrow z)
		\]
		of Corollary \ref{cor: the harmonic measure for two points is well-defined}.
	\end{proof}

	\subsection{Gluing and harmonic measure} \label{section: Gluing, Capacity and Harmonic measure}
	We suppose throughout this section that the potential kernel is well defined in the sense that the subsequential limits appearing in Lemma \ref{lem: PK exists} are all equal. By Corollary \ref{cor: the harmonic measure for two points is well-defined}, this implies that the harmonic measure from infinity is well defined for two-point sets.
	
	Let $B \subset v(G)$ be a set. Glue together all vertices in $B$ and delete all self-loops that were created in the process. We denote the graph induced by the gluing $G_B$. Note that $G_B$ need not be a simple graph, even when $G$ was.
	
	We will prove in this section that, if the potential kernel is well defined on $G$, it is also well defined on $G_B$, whenever $B$ is a finite set. Furthermore, we will prove an explicit expression of the potential kernel on the graph $G_B$ in the case where $B$ is a finite set. These results are an extension of results on the lattice $\ZZ^2$, see for instance \cite[Chapter 6]{LawlerLimicRandomWalks}, but we will use different arguments, following from the expression for the potential kernel in terms of harmonic measure from infinity as in Corollary \ref{cor: the harmonic measure for two points is well-defined}.
	
	\begin{theorem}[Gluing Theorem] \label{theorem: gluing expression}
		Suppose $a(x,y) = \lim_{n \to \infty} a_n(x,y)$ exists for all $x,y \in v(G)$ and does not depend on the choice of the sequence of sets $A_n $ going to infinity. Let $B \subset v(G)$ be a finite set, whose removal does not disconnect $G$, and suppose $x \in B$. Then
\begin{equation}\label{D:qB}
  q_B(w) := \lim_{n\to \infty} \Reff(B\leftrightarrow A_n)\PP_w( T_{A_n} < T_B)
\end{equation}
exists and is given by
		\begin{equation}\label{qB}
			q_B(w) = a(w, x) - \E_w[a(X_{T_B}, x)]; \ \ w \in v(G_B) \setminus \{B\}; q_B(B) =0.
		\end{equation}
 		Extending $q_B$ to $v(G)$ in the natural way (i.e., using \eqref{qB} with $w \in v(G)$), we have
		\begin{equation}\label{qBext}
			(\Delta q_B)(w) = \hm_B(w) := \lim_{z \to \infty} \PP_z(X_{T_B} = w); \ \ w \in B
		\end{equation}
where the Laplacian $\Delta$ is calculated on $G$ via \eqref{D:Laplacian}.
	\end{theorem}
	
		Note in particular, that in the expression \eqref{qB} for $q_B$, any choice of $x \in B$ gives the same value and so is irrelevant. 	We will prove this theorem in the two subsequent subsections, proving first \eqref{D:qB} and \eqref{qB} in Section \ref{S:qB}, and then \eqref{qBext} in Section \ref{S:qBext}.

Before we give the proof, we first state some corollaries. The first one is that the harmonic measure from infinity is well defined for the arbitrary finite set $B$ (subject to the assumption that the removal of $B$ does not disconnect $G$).

\begin{corollary}
  \label{Cor:hminfinity}
  Fix a finite set $B \subset v(G)$ as in Theorem \ref{theorem: gluing expression}. Let $A_n$ be a set of vertices tending to infinity. Then for any $x \in B$,
  \begin{equation}\label{hB}
\hm_B (x) =  \lim_{n \to \infty} \PP_{A_n} (X_{T_B}  = x)
  \end{equation}
  exists and is positive for all $x \in B$ such that the removal of $B \setminus \{x\}$ does not disconnect $x$ from infinity.
\end{corollary}

\begin{proof}
Fix $w \notin B$, then arguing as in \eqref{eq: a_n = Reff()*Prob} and \eqref{E:pkGreen} we get
$$
\P_{A_n} (T_w < T_B ) = \frac{ \Reff (A_n \leftrightarrow B) \P_w ( T_{A_n} < T_B)}{ \Reff (w \leftrightarrow B; G_{A_n})} \to \frac{q_B(w) }{\Reff (w \leftrightarrow B)}
 $$
 as $n \to \infty$. This limit is by definition the desired value of $\hm_{B \cup\{w\}} (w)$. Note furthermore that $q_B(w)$ is strictly positive by Lemma \ref{lem: PK exists}.

 Applying the same reasoning but with $B$ changed into $B' = B \setminus\{x\}$ (with $x \in B$) and $w  =x$, shows that the limit in \eqref{hB} exists. Furthermore, if the removal of $B'$ does not disconnect $x$ from $\infty$, we see that $q_{B'}(w) >0$ again, and so $\hm_B(x) >0$.
\end{proof}


Next, we show  that the potential kernel can only be well defined if the graph $G$ is one-ended.
	
	\begin{corollary} \label{corollary: not one-ended inplies that potential kernels is not well-defined}
		If the potential kernel is well defined, $G$ is one-ended.
	\end{corollary}
	\begin{proof}
	Intuitively, on multiple-ended graphs there isn't a single harmonic measure from infinity since there are several ways of converging to infinity.	Suppose $G$ has more than one end. Let $x_1, x_2, \ldots, x_M$ be some finite number of vertices, such that removing then from $v(G)$ and looking at the induced graph, we have (at least) two infinite components. Write $B_n = B(o, n)$ and choose $n$ large enough that $x_1, \ldots, x_M \in B_n$. Consider the graph $G_{B_n}$ resulting from gluing $B_n$ together as in the theorem. Clearly, the removal of $B_n$ creates at least two infinite components. Pick a vertex $z$ of $B_n^c$ and suppose it is in one infinite component. Let $(\{w_{i}\})_{i \geq 1}$ be any sequence of vertices going to infinity in an infinite component that does not contain $z$. Then $\PP_{w_i}(T_{z} < T_{B_n}) = 0$ (for each $i$), yet this converges by Corollary \ref{Cor:hminfinity} to $\hm_{B_n \cup{z}} (z)>0$ since the removal of $B_n$ does not disconnect $z$ from infinity. This is the desired contradiction.
	\end{proof}

	Theorem \ref{theorem: gluing expression} a priori only shows that the potential kernel with `pole' $B$ is well defined when $B$ does not disconnect $G$. We can, however, extend it to arbitrary finite sets $B$ and to an arbitrary second variable $y$.

	\begin{corollary} \label{cor: pk with all poles after gluing}
		Let $B \subset v(G)$ be \emph{any} finite set. The potential kernel $a_B:v(G_B)^2 \to \RR_+$ is well defined in the sense that the limit
		\[
			a^{G_B}(w, y) = \lim_{n\to \infty} \PP_w(T_{A_n} < T_y; G_B)\Reff(w \leftrightarrow y; G_B),
		\]
		exist for all $w, y \in v(G_B)$ and does not depend on the choice of sequence of sets $A_n$. Here, the probability and effective resistance are calculated on the graph $G_B$.
	\end{corollary}
	\begin{proof}
		We start with taking $\bar{B}$ as the hull (in the sense of complex analysis, meaning we ``fill it in'' with respect to the point at infinity) of $B$, defined by adding to $B$ all the points in $v(G_B)$ that belong to \emph{finite} connected components of $v(G_B) \setminus B$. Since $G$ is one-ended by Corollary \ref{corollary: not one-ended inplies that potential kernels is not well-defined}, $\bar{B}$ does not disconnect $G$. By Theorem \ref{theorem: gluing expression}, we have that for any sequence of sets $(A_n)_{n \geq 1}$ going to infinity, the limit
		\[
			a^{G_{\bar{B}}}(w, \bar{B}) := q_{\bar{B}}(w) = \lim_{n \to \infty} \PP_{w}(T_{A_n} < T_{\bar{B}})\Reff( \bar{B} \leftrightarrow A_n)
		\]
		exists for each $w \in v(G_{\bar{B}})$ and does not depend on the choice of sequence of vertices $A_n$ going to infinity. Moreover, this limit also trivially exists (and is zero) if $w$ is in one of the finite components of $v(G) \setminus B$.
		
Hence we deduce that actually for all $w \in v(G_B)$ we have that the limit
		\[
			a^{G_B}(w, B) = \lim_{n\to \infty} \PP_w(T_{A_n} < T_B)\Reff(A_n \leftrightarrow B)
		\]
		exists and does not depend on the choice of sequence of sets going to infinity $A_n$. Now, by Proposition \ref{prop: green function and PK} (see Remark \ref{remark: prop Green and PK for subsequential limits}) we get that for \emph{all} $w, y \in v(G_B)$ the limit
		\[
			a^{G_B}(w, y) = \lim_{n\to \infty} \PP_w(T_{A_n} < T_y; G_B)\Reff(A_n \leftrightarrow y; G_B)
		\]
		exists and does not depend on the choice of the sequence $A_n$. This is the desired result.
	\end{proof}
	
	\subsubsection{Proof of \eqref{D:qB} and \eqref{qB} }
\label{S:qB}

	\begin{proof}
		Fix $(A_n)_{n \geq 1}$ a sequence of finite sets of vertices going to infinity. For a finite set $B \subset v(G)$ and $x \in B$, we will define the function $q_B:v(G_B) \to \RR_+$ through
		\[
			q_B(w) = a(w, x) - \E_w[a(X_{T_B}, x)],
		\]
		and $q_B(B) = 0$, whenever the potential kernel on $G$ is well defined. We will prove \eqref{D:qB} using induction on the number of vertices $m$ in $B$. To be more precise, we will show that for any recurrent graph $G$ for which the potential kernel is well defined (in other words, $\lim_{n\to \infty} a_{A_n}(x, y) = a(x, y)$ and does not depend on the sequence $A_n$) for any set $B \subset v(G)$ with $|B| = m$ and $v(G) \setminus B$ connected, we have that \eqref{D:qB} holds. The base case $m = 1$ holds trivially.
		
		Let $m \in \NN$ and suppose that for any recurrent graph $G$ on which the potential kernel is well defined and for any subset $B \subset v(G)$ with $|B| = m$ and $v(G) \setminus B$ connected we have that \eqref{D:qB} and \eqref{qB} are satisfied for each $x \in B$.
		
		In this case,
		\[
			q_B(w) = \lim_{n\to \infty} \PP_{w}(T_{A_n} < T_B)\Reff(B \leftrightarrow A_n)
		\]
		by assumption exists and does not depend on the sequence $(A_n)_{n \geq 1}$, so we also have that $a^{G_B}(\cdot, B) = q_B(\cdot)$ by \eqref{eq: a_n = Reff()*Prob}. Remark \ref{remark: prop Green and PK for subsequential limits} then shows us that $a^{G_B}(\cdot, y)$ is well defined for \emph{any} $y \in v(G_B)$ and hence we know that the potential kernel is well defined on $G_B$ too.
		
		\noindent
		\textbf{Induction.} Let $G$ be a recurrent graph for which the potential kernel is well defined and let $B \subset v(G)$ be a finite set such that $|B| = m + 1$ and $v(G) \setminus B$ is connected. Fix $x \in B$. We split into two cases, depending on $x$:
		\begin{enumerate}[(i)]
			\item the removal of $x$ from $G$ disconnects $B \setminus \{x\}$ from infinity in $G$ or
			\item it does not.
		\end{enumerate}
		We begin with the easy case. Suppose we are in situation (i). We have that for all $w \notin B$ (for $n$ large enough)
		\[
			\PP_{A_n}(T_w < T_x) = \PP_{A_n}(T_w < T_{B}) \quad \text{and} \quad \Reff(w \leftrightarrow B) = \Reff(w \leftrightarrow x).
		\]
		The limit on the left-hand side exists as the potential kernel is well defined, see Corollary \ref{cor: the harmonic measure for two points is well-defined}, and hence $\lim_{n \to \infty} \PP_{A_n}(T_w < T_{B}) \Reff(w \leftrightarrow B)$ exists and equals $a(w, x)$. Moreover, we also have
		\[
			q_B(w) = a(w, x) - \E_{w}[a(X_{T_B}, x)] = a(w, x) - a(x, x) = a(w, x)
		\]
		which proves the result for this choice of $x$.
		
		We move on to the more interesting case (ii). Since we are not in case (i), we can find a set $B' \subset B$ with $|B'| = m$ and $v(G) \setminus B'$ connected (indeed, since we are not in case (i), there is at least a path going from some vertex in $B$ to infinity, without touching $x$, and removing from $B$ the last vertex in $B$ visited by this path provides such a set $B'$). Take $y$ to be the vertex such that $\{y\} = B \setminus B'$.
		
		Since $|B'| = m$, we have by the induction hypothesis that the potential kernel $a^{G_{B'}}(\cdot, \cdot)$ is well defined. Pick $w \in v(G)$ such that $w \notin B$, which we can view also as a vertex in $G_B$ and $G_{B'}$. Fix $n$ so large that both $B$ and $w$ are not in $A_n$. Using \eqref{eq: a_n = Reff()*Prob} we have that
		\[
			a_{A_n}^{G_{B'}}(w, B') = \Reff(B' \leftrightarrow A_n)\PP_w(T_{A_n} < T_{B'}).
		\]
		We focus on the probability appearing on the right-hand side. By the law of total probability and the strong Markov property of the simple random walk, we have
		\begin{align*}
			\PP_w(T_{A_n} < T_{B'}) &= \PP_w(T_{y} < T_{A_n} < T_{B'}) + \PP_w(T_{A_n} < T_{y} \wedge T_{B'}) \\
			&= \PP_w(T_{y} < T_{A_n} \wedge T_{B'})\PP_{y}(T_{A_n} < T_{B'}) + \PP_w(T_{A_n} < T_{B'} \wedge T_{y}).
		\end{align*}
		Since $G$ (and hence $G_{B'}$) is recurrent, we have that $\Reff(B' \leftrightarrow A_n) \sim \Reff(x \leftrightarrow A_n) \sim \Reff(B \leftrightarrow A_n)$ where $a_n \sim b_n$ means $a_n / b_n \to 1$ as $n \to \infty$. Taking $n \to \infty$ in the above identity after multiplying by $\Reff(B' \leftrightarrow A_n)$ and using once more recurrence, we deduce that
		\[
			a^{G_{B'}}(w, B') = \PP_w(T_{y} < T_{B'})a^{G_{B'}}(y, B') + \lim_{n\to \infty} \PP_w(T_{A_n} < T_{B}) \Reff(A_n \leftrightarrow B),
		\]
		because the potential kernel on $G_{B'}$ is well defined by assumption. This implies in particular that
		\[
			\lim_{n\to \infty} \PP_w(T_{A_n} < T_B) \Reff(B \leftrightarrow A_n) = a^{G_{B'}}(w, B') - \PP_w(T_y < T_{B'})a^{G_{B'}}(y, B')
		\]
		exists and does not depend on the sequence $A_n$ and, thus, we deduce that $a^{G_{B}}(w, B)$ is well defined and satisfies
		\begin{equation} \label{eq: aGB in aGB'}
			a^{G_B}(w, B) = a^{G_{B'}}(w, B') - \PP_w(X_{T_B}  = y)a^{G_{B'}}(y, B').
		\end{equation}
		We are left to prove that $q_B(w) = a^{G_B}(w, B)$. By the induction hypothesis (because $x \in B'$) we know that
		\[
			a^{G_{B'}}(w, B') = q_{B'}(w) = a(w, x) - \E_w[a(X_{T_{B'}}, x)].
		\]
		Using this in \eqref{eq: aGB in aGB'} we get
		\begin{align*}
			a^{G_B}(w, B) &= a(w, x) - \E_w[a(X_{T_{B'}}, x)] - \PP_w(X_{T_B} = y) \Big( a(y, x) - \E_{y}[a(X_{T_{B'}}, x)]\Big) \\
			&= a(w, x) - \PP_{w}(X_{T_B} = y)a(y, x) - \sum_{z \in B' } \PP_w(X_{T_{B'}} = z)a(z, x) \\
			&\quad + \sum_{z \in B'}\PP_w(X_{T_B} = y) \PP_{y}(X_{T_{B'}} = z) a(z, x) \\
			&= a(w, x) - \sum_{z \in B} \PP_w(X_{T_B} = z)a(z, x),
		\end{align*}
		where in the last line we used for $z \in B'$ the equality
		\[
			\PP_w(X_{T_{B}} = z) = \PP_w(X_{T_{B'}} = z) - \PP_w(X_{T_{B}} = y)\PP_{y}(X_{T_{B'}} = z),
		\]
		which holds due to the strong Markov property for the random walk. But of course, this is the same as
		\[
			a^{G_B}(w, B) = a(w, x) - \E_w[a(X_{T_B}, x)],
		\]
		so indeed we have that $a^{G_B}(w, B) = q_B(w)$, which finishes the induction argument.
	\end{proof}

	\begin{wrong-old}
	\begin{proof}
		Fix $(A_n)_{n \geq 1}$ some sequence of (finite) sets of vertices going to infinity. We will prove that for all $B \subset v(G)$ finite, for which the removal of $B$ does not disconnect the graph, we have
		\begin{equation} \label{eq: gluing induc hypothesis}
			a(w, x) - \E_w[a(X_{T_B}, x)] = q_B(w) = \lim_{n\to \infty} \PP_w(T_{A_n} < T_B) \Reff(A_n \leftrightarrow B),
		\end{equation}
		which proves \eqref{D:qB} and \eqref{qB} as the left-hand side is independent of the sequence $(A_n)_{n \geq 1}$ going to infinity. We will do an induction on the number of vertices $m$ in $B$, showing that for every recurrent graph $G$ such that the potential kernel $a(x,y) = \lim_{n \to \infty} a_n(x,y) = \lim_{n \to \infty} a_{A_n}(x,y)$ (defined in \eqref{eqdef: PK along An}) is well defined on $G$ and is independent of the sequence $A_n$, then \eqref{D:qB} and \eqref{qB} are satisfied on $G$ for any set $B \subset v(G)$ with $|B|= m$ and $v(G) \setminus B$ connected. The main idea will be to view $q_B(\cdot) $ as the two-point function $a(\cdot, B)$ on the graph $G_B$ where $B$ is glued to form a single vertex.
		The case $m =1$ of this induction is actually trivial, but for this argument we will also need to treat the case $m=2$ first. 	

		\noindent
		\textbf{Base case $m =2$.} Let $G$ be a graph satisfying the above assumptions. Assume that $B\subset v(G)$ is of the form $\{x_1, x_2\}$, with $x_1 \not = x_2$. Pick $w \in v(G)$ such that $w \notin B$, and which we can also view as a vertex of $G_B$. Fix $n$ so large that $x_1, x_2, w \not \in A_n$. We have
		\[
			a_n(w, x_1) = \Reff(x_1 \leftrightarrow A_n) \PP_w(T_{A_n} < T_{x_1})
		\]
		as in \eqref{eq: a_n = Reff()*Prob}. We focus on the probability appearing on the right-hand side. By law of total probability and the strong Markov property of the simple random walk, we get
		\begin{align*}
			\PP_w(T_{A_n} < T_{x_1}) &= \PP_w(T_{x_2} < T_{A_n} < T_{x_1}) + \PP_w(T_{A_n} < T_{x_1} \wedge T_{x_2}) \\
			&= \PP_w(T_{x_2} < T_{A_n} \wedge T_{x_1})\PP_{x_2}(T_{A_n} < T_{x_1}) + \PP_w(T_{A_n} < T_{x_1} \wedge T_{x_2}).
		\end{align*}
		By recurrence of $G$, we have that $\Reff(x_1 \leftrightarrow A_n) \sim \Reff(x_2 \leftrightarrow A_n) \sim \Reff(\{x_1, x_2\} \leftrightarrow A_n)$ where by $a_n \sim b_n$ we mean $a_n / b_n \to 1$ as $n \to \infty$. Taking $n \to \infty$ in the above identity after multiplying by $\Reff (x_1 \leftrightarrow A_n)$, and using once more recurrence, we deduce that
		\[
			a(w, x_1) = \PP_w(T_{x_2} < T_{x_1})a(x_2, x_1) + \lim_{n \to \infty}\PP_w(T_{A_n} < T_{B})\Reff(A_n \leftrightarrow B),
		\]
		since the potential kernel is well defined on $G$ by assumption. This implies in particular that
		\begin{equation}\label{twopointGB}
			 a(w, x_1) - \PP_{w}(T_{x_2} < T_{x_1})a(x_2, x_1) = \lim_{n \to \infty} \PP_w(T_{A_n} < T_B)\Reff(A_n \leftrightarrow B)
		\end{equation}
		exists and does not depend on the choice of sequence of vertices $A_n$ going to infinity. We can then define $q_B(w)$ to be either the left hand side or the right hand side, and this proves \eqref{eq: gluing induc hypothesis} for the case $|B| = 2$.

Note also in particular that with this definition, the two-point function on $G_B$, $a^{G_B} (\cdot,B)$, is well defined and is equal to $q_B(\cdot)$ by \eqref{twopointGB}. By Remark \ref{remark: prop Green and PK for subsequential limits}, it follows that this two-point function then exists everywhere on $G_B$, a fact we will use in the inductive part of the argument. 		\\
		
		\noindent
		\textbf{Induction.} Fix $m \in \NN$ and assume that for all $B$ with $|B| \leq m$, such that removing $B$ does not disconnect $G$, we have that \eqref{eq: gluing induc hypothesis} holds.
		
		Let $B \subset v(G)$ be a set with $|B| = m + 1$ such that removing $B$ does not disconnect $G$. Fix $x \in B$ and write $B = \{x_1, \ldots, x_{m + 1}\}$. Fix $x \in B$. We split into two cases:
		\begin{enumerate}[(i)]
			\item the removal of $x$ from $G$ disconnects $B \setminus \{x\}$ from infinity in $G$ or
			\item the removal of $x$ from $G$ does not disconnect $B \setminus \{x\}$ from infinity.
		\end{enumerate}
		We start with case (i). In this case, we have that for all $w \notin B$ (for $n$ large enough)
		\[
			\PP_{A_n}(T_w < T_x) = \PP_{A_n}(T_w < T_{B}) \quad \text{and} \quad \Reff(w \leftrightarrow B) = \Reff(w \leftrightarrow x).
		\]
		The limit on the left-hand side exists as the potential kernel is well defined, see Corollary \ref{cor: the harmonic measure for two points is well-defined}, and hence $q_B(w) = \lim_{n \to \infty} \PP_{A_n}(T_w < T_{B}) \Reff(w \leftrightarrow B)$ exists and equals $a(w, x)$. Moreover, we also have
		\[
			a(w, x) - \E_{w}[a(X_{T_B}, x)] = a(w, x) - a(x, x) = q_B(w),
		\]
		which proves the result \eqref{eq: gluing induc hypothesis} for $x$.
		
		Case (ii) is a bit more elaborate. Since we are not in case (i), we can find a set $B' \subset B$ with $|B'| = m, x \in B' $ and the removal of $B'$ does not disconnect $G$. (Indeed, since we are not in case (i), there is at least a path going from some vertex in $B$ to infinity without touching $x$, and removing from $B$ the last vertex in $B$ visited by this path provides the desired subset $B'$).

Note that by the case $m =2$, applied iteratively, we know that the two-point function $a^{G_{B'}} (\cdot, \cdot)$ is well defined on $G_{B'}$, and $G_{B'}$ is of course also recurrent. We can therefore apply the induction hypothesis to $G_{B'}$.
By the induction hypothesis that \eqref{eq: gluing induc hypothesis} holds for $B_m$. In particular, we can apply the \emph{base case} $m = 2$ to the graph $G_{B'}$, where we glue together the vertex $B' $ and $y$, where $B = B' \cup \{y\}$. Therefore let $\tilde q_{B}$ denote the potential kernel, defined on $G_B$, defined in \eqref{twopointGB}. Then
		\begin{equation} \label{eq: potential kernel after ungluing two vertices}
			\tilde{q}_{B}(w) := a^{G_{B'}} (w, B')  - \PP_w(X_{T_B} = y)a^{G_{B'}}(y, B') = \lim_{n \to \infty} \PP_w(T_{A_n} < T_B) \Reff(A_n \leftrightarrow B).
		\end{equation}
It remains to prove that $\tilde{q}_B$ satisfies \eqref{eq: gluing induc hypothesis}. Notice however that, by the induction hypothesis, we have
		\[
			a^{G_B'} (w, B') = q_{B' }(w) = a(w, x) - \E_w[a(X_{T_{B'}}, x)].
		\]
		Plugging this in \eqref{eq: potential kernel after ungluing two vertices} we get
		\begin{align*}
			\tilde{q}_{B}(w) &= a(w, x) - \E_w[a(X_{T_{B'}}, x)] - \PP_w(X_{T_B} = y) \Big( a(y, x) - \E_{y}[a(X_{T_{B'}}, x)]\Big) \\
			&= a(w, x) - \PP_{w}(X_{T_B} = y)a(y, x) - \sum_{z \in B' } \PP_w(X_{T_{B'}} = z)a(z, x) \\
			&\quad + \sum_{z \in B'}\PP_w(X_{T_B} = y) \PP_{y}(X_{T_{B'}} = z) a(z, x) \\
			&= a(w, x) - \sum_{z \in B} \PP_w(X_{T_B} = z)a(z, x),
		\end{align*}
		where in the last line we used for $z \in B_m$ the equality
		\[
			\PP_w(X_{T_{B}} = z) = \PP_w(X_{T_{B'}} = z) - \PP_w(X_{T_{B}} = y)\PP_{y}(X_{T_{B'}} = z),
		\]
		which holds due to the strong Markov property for the random walk. We deduce that
		\[
			\tilde{q}_{B}(w) = a(w, x) - \E_w[a(X_{T_{B}}, x)] = q_B(w),
		\]
		which finishes the induction argument.
	\end{proof}
	\end{wrong-old}

	\subsubsection{Proof of \eqref{qBext}}
\label{S:qBext}
	Let $B \subset v(G)$ be a finite set, such that its removal does not disconnect $G$. So far, we have shown that the potential kernel is well defined on the graph $G_B$ and hence that the harmonic measure from infinity is well defined, see Corollary \ref{Cor:hminfinity}. In this section, we will prove \eqref{qBext}; the third statement of Theorem \ref{theorem: gluing expression}. First, let us introduce some notation that will only be used here. If $G$ is a graph and $B \subset v(G)$ a (finite) set, then we will write $\Delta$ for the Laplacian on $G$ and $\Delta^{G_B}$ for the Laplacian on $G_B$.
	
	\begin{proof}[Proof of \eqref{qBext}]
		Let $G$ be a recurrent graph on which the potential kernel is well defined, and suppose that $B \subset v(G)$ is a finite set such that $v(G) \setminus B$ is connected. Fix $x \in B$. We split into two cases:
		\begin{enumerate}[(i)]
			\item the removal of $x$ disconnects $B \setminus \{x\}$ from infinity in $G$ or
			\item is does not.
		\end{enumerate}
		In the first case, we have that $\hm_{B}(x) = 1$ and also that $q_B(w) = a(w, x)$ for all $w \in B$ (indeed, for $w \notin B$ this follows immediately from \eqref{qB} and for $w \in B \setminus \{x\}$ we have that $q_B(w) = 0 = a(w, x)$ in this case). Hence, we deduce
		\[
			\delta_{x}(\cdot) = \Delta(a(\cdot, x)) = \Delta(q_B(\cdot)),
		\]
		which shows the result in case (i).
		
		In case (ii), take $B' = B \setminus \{x\}$. We will show that
		\begin{equation} \label{eq: DeltaGB' = hmB}
			\Delta_u^{G_{B'}}\big(a^{G_{B'}}(w, B') - \PP_w(T_{x} < T_{B'})a^{G_{B'}}(x, B') \big) \mid_{w = x} = \hm_{B}(x),
		\end{equation}
		where $\Delta^{G_{B'}}_u$ is the Laplacian acting on the function with variable $u$. Let us first explain how this shows the final result. As in \eqref{eq: aGB in aGB'} and \eqref{qB} we know that (when $q_B$ is viewed as a function on $v(G_{B'})$)
		\[	
			q_B(w) = a^{G_{B'}}(w, B') - \PP_w(T_{x} < T_{B'})a^{G_{B'}}(x, B').
		\]
		Moreover, when $w \in B$, we have
		\[
			q_B(w) = a(w, x) - \E_w[a(X_{T_B}, x)] = a(w, x) - a(w, x) = 0.
		\]
		Hence, actually,
		\[
			(\Delta^{G_{B'}} q_{B})(x) = \sum_{\substack{w \sim x \\ w \in v(G_{B'})}} q_B(w) = \sum_{\substack{w \sim x \\ w \in v(G)}} q_B(w) = (\Delta q_B)(x),
		\]
		so that \eqref{eq: DeltaGB' = hmB} implies the final result.
		
		To prove \eqref{eq: DeltaGB' = hmB}, recall from \eqref{eq:efres through flow} that\footnote{Of course, to be precise we would need to calculate the probabilities and effective resistances on the graph $G_{B'}$, but since this makes no difference in the current setting, we skip the extra notation.}
		\[
			\sum_{\substack{u \sim x \\ u \in v(G_{B'})}} \PP_u(T_x < T_{B'}) = \frac{1}{\Reff(x \leftrightarrow B')},
		\]
		and that $\Delta^{G_{B'}} (a^{G_{B'}}(\cdot, B')) = \delta_{B'}(\cdot)$ by Proposition \ref{prop: properties of the PKs}. Using these two facts, we get
		\begin{align*}
			\Delta^{G_{B'}}_w(a^{G_{B'}}(w, x_2) - \PP_w(T_{x} < T_{B'})a^{G_{B'}}(x, B'))\Big|_{w = x_1} &= -a^{G_{B'}}(x, B')\sum_{\substack{u \sim x \\ u \in v(G_{B'})}} (\PP_u(T_{x} < T_{B'}) - 1) \\
			& = a^{G_{B'}}(x, B') \sum_{\substack{u \sim x \\ u \in v(G_{B'})}} \PP_u(T_{B'} < T_{x}) \\
			&= \frac{a^{G_{B'}}(x, B')}{\Reff(x \leftrightarrow B')} = \hm_{B', x}(x).
		\end{align*}
		The last equality follows from Corollary \ref{cor: the harmonic measure for two points is well-defined}, which allows us to write
		\[
			a^{G_{B'}}(x, B') = \hm_{x, B'}(x)\Reff(x \leftrightarrow B').
		\]
		This shows  \eqref{qBext} and therefore concludes the proof of Theorem \ref{theorem: gluing expression}. In turn this finishes the proof that (a) is equivalent to (b) in Theorem \ref{T:main_equiv} (see e.g. Corollary \ref{Cor:hminfinity}).
 	\end{proof}

	\begin{wrong-old}
	\begin{claim} \label{claim: the harmonic measure for two-point sets}
		Let $B$ be of the form $B = \{x_1, x_2\}$. For all $w \in B$ we have
		\begin{equation} \label{eq: the harmonic measure for two-point sets}
			\Delta_u (a(u, x_2) - \PP_{u}(T_{x_1} < T_{x_2})a(x_1, x_2)) \mid_{u = w} = \hm_{\{x_1, x_2\}}(w),
		\end{equation}
		where $\Delta_u$ is the graph Laplacian (recall \eqref{D:Laplacian}) on $G$, acting on the function with $u$ as a variable.
	\end{claim}
	\begin{proof}
		Fix $x_1, x_2 \in v(G)$. Recall that
		\[
			\sum_{u \sim x_1} \PP_u (T_{x_1} < T_{x_2}) = \frac{1}{\Reff(x_1 \leftrightarrow x_2)},
		\]
		for instance by \eqref{eq:efres through flow}. Also recall that $\Delta a(\cdot, x_2) = \delta_{x_2}(\cdot)$ by Proposition \ref{prop: properties of the PKs}. Using these two facts, we find
		\begin{align*}
			\Delta_u(a(u, x_2) - \PP_u(T_{x_1} < T_{x_2})a(x_1, x_2)\Big|_{u = x_1} &= -a(x_1, x_2)\sum_{u \sim x_1} (\PP_u(T_{x_1} < T_{x_2}) - 1) \\
			& = a(x_1, x_2) \sum_{u \sim x_1} \PP_u(T_{x_2} < T_{x_1}) \\
			&= \frac{a(x_1, x_2)}{\Reff(x_1 \leftrightarrow {x_2})} = \hm_{x_1, x_2}(x_1).
		\end{align*}
		The last equality follows from Corollary \ref{cor: the harmonic measure for two points is well-defined}, which allows us to write
		\[
			a(x_1, x_2) = \hm_{x_1, x_2}(x_1)\Reff(x_1 \leftrightarrow x_2).
		\]
		Using the same arguments, we also get
		\begin{align*}
			\Delta_u(a(u, x_2) - \PP_u(T_{x_1} < T_{x_2})a(x_1, {x_2})\Big|_{u = x_2} &= 1 - a(x_1, x_2)\sum_{u \sim x_2} (\PP_u(T_{x_1} < T_{x_2})- 0) \\
			&= 1 - \frac{a(x_1, x_2)}{\Reff(x_1 \leftrightarrow x_2)} \\
			&= \hm_{x_1, x_2}(x_2),
		\end{align*}
		showing the desired result.
	\end{proof}

	We recall \eqref{qBext} in the next lemma; which finishes the proof of Theorem \ref{theorem: gluing expression}. Recall also that, for a finite set $B \subset v(G)$ and $x \in B$, for $w \in v(G)$, we can write
	\[
		q_B(w) = a(w, x) - \E_w[a(X_{T_B}, x)].
	\]
	
	\begin{lemma} \label{lem: harmonic measure of finite sets}
		Let $B$ be a finite subset of $v(G)$ such that its removal does not disconnect $G$. For all $w \in v(G)$ we have
		\[
			(\Delta q_B)(w) = \hm_{B}(w),
		\]
		where $\Delta$ is the graph Laplacian acting on $G$.
	\end{lemma}
	
	\begin{proof}
		By definition of $q_B(w)$, we have that $\Delta q_B(w) = 0$ for all $w \not \in B$, as $\Delta a(\cdot, x) = \delta_x(\cdot)$ and the map $w \mapsto \E_w[a(X_{T_B}, x)]$ is clearly harmonic for $w \not \in B$.
		
		Assume henceforth that $w \in B$. Define $B_w = B \setminus \{w\}$ as the removal of $w$ from $B$. Note that $q_{B_w}(\cdot)$ is then the potential kernel on $G_{B_w}$ with pole $B_w$. Using Claim \ref{claim: the harmonic measure for two-point sets} for the two-point sets, we find
		\begin{equation} \label{subeq: gluing harmonic measure}
			\Delta_u \big[q_{B_w}(u)- \PP_{u}(T_w < T_{B_w})q_{B_w}(w) \big]_{u = w} = \hm_{B}(w).
		\end{equation}
		Now, by definition of $q_B$ we also have that
		\[
			q_B(w) = a(w, x) - \E_{w}[a(X_{T_B}, x)] = 0
		\]
		whenever $w \in B$ (and similarly, $q_{B_w}(u) = 0$ when $u \in B_w$). By \eqref{qB} we have that for $u \not \in B$,
		\[
			q_B(u) = q_{B_{w}}(u) - \E_w[q_{B_w}(X_{T_B})] = q_{B_w}(u) - \PP_u(T_w < T_{B_w})q_{B_w}(w),
		\]
		Hence, we actually have that \eqref{subeq: gluing harmonic measure} is equivalent to $(\Delta q_B)(w) = \hm_{B}(w)$, showing the desired result.
	\end{proof}
	\end{wrong-old}

	\section{Proof of Theorem \ref{T:main_existence} }
	Before proceeding with the remaining equivalences we give a proof that (a) holds under the assumption of Theorem \ref{T:main_existence}. Recall that a random graph $(G, o)$ is \textbf{strictly subdiffusive} whenever there exits a $\beta > 2$ such that
	\begin{equation} \label{eqdef: subdiffusive} \tag{SD}
		\bf{E}[d(o, X_n)^\beta] \leq Cn.
	\end{equation}
	
	%
	
	We collect the following theorem of \cite{benjamini2015}. The main theorem from that paper shows that, assuming subdiffusivity, strictly sublinear harmonic functions must be constant. In fact, as already mentioned in that paper (see Example 2.10), the arguments in that paper also show that assuming \emph{strict} subdiffusivity, even harmonic functions of at most linear growth must be constant. It is this extension which we use here, and which we quote below.
	
	\begin{theorem}[Theorem 3 in \cite{benjamini2015}] \label{theorem: Benjamini et al - no linear harmonics}
		Let $(G, o, X_1)$ be a strictly subdiffusive \eqref{eqdef: subdiffusive}, recurrent, stationary environment. A.s., every harmonic function on $G$ that is of at most linear growth is constant.
	\end{theorem}

	We now give the proof of Theorem \ref{T:main_existence} using this result.

	\begin{proof}[Proof of Theorem \ref{T:main_existence}.]
	Let $(G, o)$ be a unimodular graph that is almost surely strictly subdiffusive \eqref{eqdef: subdiffusive} and recurrent, satisfying $\E[\deg(o)] < \infty$. Then degree biasing $(G, o)$ gives a reversible environment and hence, almost surely, all harmonic functions on $(G, o, X_1)$ that are at most linear are constant due to Theorem \ref{theorem: Benjamini et al - no linear harmonics}. After degree unbiasing, the same statement is true for $(G, o)$.
	
	Let $a_1, a_2:v(G)^2 \to \RR_+$ be two potential kernels arising as subsequential limits in the sense of Lemma \ref{lem: PK exists}. Fix $y \in v(G)$. By Proposition \ref{prop: properties of the PKs} we have that $a_i(\cdot, y)$ is of the form
	\[
		a_i(x, y) = \Reff(x \leftrightarrow y)H_i(x),
	\]
	with $0 \leq H_i(x) \leq 1$ for each $x$ and $i = 1, 2$. Define next the map $h:v(G) \to \RR$ through
	\[
		h(x) = a_1(x, y) - a_2(x, y).
	\]
	Clearly, $h$ is harmonic everywhere outside $y$ by choice of the $a_i$'s and linearity of the Laplacian. Since $\Delta a_1(\cdot, y) = \Delta a_2(\cdot, y)$ by Proposition \ref{prop: properties of the PKs}, we also get that $\Delta h(y) = 0$ and we deduce that $h$ is harmonic everywhere.
	
	Next, we notice
	\[
		|h(x)| \leq |H_1(x) - H_2(x)|\Reff(y \leftrightarrow x) \leq 2 d(y, x),
	\]
	implying that $h$ is (at most) linear. Thus $h$ must be constant. Since $h(y) = a_1(y, y) - a_2(y, y) = 0$, it follows that $h(x) = 0$ and hence we finally obtain $a_1(x, y) = a_2(x, y)$ for all $x \in v(G)$. Since $y \in v(G)$ was arbitrary, we deduce the desired result.
	\end{proof}

	\begin{remark}\label{R:lee}
Strict subdiffusivity on the UIPT was obtained by Benajmini and Curien in the beautiful paper \cite{BenjCurienSubdiffusivity}. A result of Lee \cite[Theorem 1.10]{Lee2017} gives a more general condition which guarantees strict subdiffusivity (essentially, the graph needs to be planar with at least cubic volume growth). 
	\end{remark}

As a prominent example of application of Theorem \ref{T:main_existence} consider the Incipient Infinite percolation Cluster (IIC) of $\ZZ^d$ for sufficiently large $d$. By a combination of Theorem 1.2 in \cite{KozmaNachmias} and Theorem 1.1 in \cite{Lee_scaling}, one can check that the strict subdiffusivity \eqref{eqdef: subdiffusive} is satisfied in all sufficiently high dimensions.  The recurrence is easier to check. (Note that a weaker form of subdiffusivity can be deduced by combining \cite{KozmaNachmias} with \cite{BJKS}). In fact, it was already checked earlier that in high dimensions the backbone of the IIC is one-ended (\cite{vdHJarai}), implying also the UST is one-ended in this case.

We point out that the result should apply in dimension two (even for non-nearest neighbour walk), or for the IIC of spread-out percolation, although we do not know if strict subdiffusivity has been checked in that case.

\begin{wrong-old}	
	This gives the following list of examples to which Theorem \ref{T:main_equiv} applies and hence for which (a) -- (d) in Theorem \ref{T:main_equiv} hold.
	\begin{description}
		\item[UIPT / UIPQ] It is known that both graphs are recurrent (\cite{GurelNachmias2013}) and have volume growth of dimension $d = 4$ (\cite{Angel2003}), hence Theorem \ref{theorem: lee2017} holds. (Although in both cases, better bounds on the diffusivity exponent are known, see \cite{GwynneMiller2020} and \cite{GwynneHutchcroft2020} for the UIPT.)
			
		\item[IIC in two-dimensions] This graph is clearly recurrent as a subgraph of $\ZZ^2$, but sub-diffusivity is non-obvious and does not follow from Theorem \ref{theorem: lee2017}. However, this was proved in \cite{GangulyLee2020}.
			
		\item[CRT-mated map] Although from Theorem \ref{theorem: lee2017} it would only follow from all $\gamma \geq \gamma_0$, where $\gamma_0$ is large enough so that the volume growth assumption is satisfies. However, it is known that the CRT-mated-map is strictly sub diffusive for all $\gamma \in (0, 2)$, so that Theorem \ref{T:main_existence} also applies. We note that there is a more elementary proof of this fact, using the Harnack inequalities obtained in \cite{NathanaelEwainCRT} and applying directly Theorem \ref{T:main_equiv}.
	\end{description}
\end{wrong-old}

\begin{wrong-old}
	\section{(d) implies (b)}
	In this short section, we recall the result (d) implies (b)\footnote{The careful reader might notice that the result in \cite{BLPS} concerns the limit $\PP_{u}(T_x < T_y)$ as $u \to \infty$ along any sequence. However, it is not hard to see that this also implies the result for arbitrary sequences of sets $(A_n)$ going to infinity.} due to \cite[Theorem ...]{BLPS} and that for planar graphs with uniformly bounded face-degrees, the uniform spanning tree is one-ended almost surely:
	
	\begin{theorem} \label{T: G recurrent bdd faces}
		Let $G$ be a recurrent, planar graph such that its faces have uniform bounded degree, then the UST is one-ended almost surely.
	\end{theorem}

	We stress that this result is not new - see for example \cite{LyonsPeresProbNetworks}. Denote by $\c{T}$ the uniform spanning tree on $G$.
	
	\begin{proof}[Sketch of proof of Theorem \ref{T: G recurrent bdd faces}]
	Let $G$ be a planar graph from now on and let $G^{\dagger}$ be its planar dual. If $G$ is a recurrent planar graph and $G^{\dagger}$ is locally finite and recurrent, then the uniform spanning tree on $G$ is almost surely one-ended, see \cite{BLPS} or \cite[Theorem 10.36]{LyonsPeresProbNetworks}. Thus, we need to prove that $G^{\dagger}$ is locally finite and recurrent. This holds whenever $G$ is recurrent and $G^{\dagger}$ has (uniform) bounded degree - i.e. when the degrees of the faces of $G$ are uniformly bounded, as recurrence ``survives'' rough embeddings, see \cite{Kanai1986}, or also \cite[Theorem 2.17]{LyonsPeresProbNetworks}. This proves theorem \ref{T: G recurrent bdd faces}.
	\end{proof}

\end{wrong-old}

	
	\section{The sublevel set of the potential kernel} \label{section: Level-sets of the PK}
	Let $(G, o)$ be some recurrent, rooted, graph for which the potential kernel is well defined in the sense that $a_n(x, y)$ obtains a limit and this does not depend on the choice of the sequence $(A_n)_{n \geq 1}$ of finite sets of vertices going to infinity.
	
	Fix $z \in v(G)$ and $R \in \RR_+$. Recall the notation in \eqref{eqdef: Beff} for the ball with respect to the effective resistance metric:
	\begin{equation*}
		\Beff(z, R) = \{x \in v(G): \Reff(z \leftrightarrow x) \leq R\}
	\end{equation*}
	We also introduce the notation for the sublevel set of $a(\cdot, z)$ through
	\[
		\Lambda_a(z, R) = \{x \in v(G): a(x, z) \leq R\}.
	\]
	In case $z = o$, we will drop the notation for $z$ and write $\Beff(R)$, $\Lambda_a(R)$ for $\Beff(o, R)$, $\Lambda_a(o, R)$ respectively. Although $a(\cdot, \cdot)$ fails to be a distance as it lacks to be symmetric, it \emph{is} what we call a quasi-distance as it does satisfy the triangle inequality due to Proposition \ref{prop: green function and PK}. On $2$-connected graphs (where the removal of any single vertex does not disconnect the graph), we have that $a(x, y) = 0$ precisely when $x = y$. In particular, this is true for triangulations.

	Let us first explain \emph{why} we care about the sublevel sets of the potential kernel and why we will prefer it over the effective resistance balls. We will call a set $A \subset v(G)$ \textbf{simply connected} whenever it is connected (that is, for any two vertices $x, y$ in $A$, there exists a path connecting $x$ and $y$, using only vertices inside $A$) and when removing $A$ from the graph does not disconnect a part of the graph from infinity. We make the following observation, which holds because $x \mapsto a(x, o)$ is harmonic outside of $o$.
	
	\begin{obs}
		The set $\Lambda_a(R)$ is simply connected.
	\end{obs}
	
	\noindent
	This is not true, in general, for $\Beff(R)$. Introduce the \textbf{hull} $\overline{\Beff(z, R)}$ of $\Beff(z, R)$ as the set $\Beff(z, R)$ together with the \emph{finite} components of $v(G) \setminus \Beff(z, R)$. Even though $\overline{\Beff(z, R)}$ does not have any more ``holes'', we notice that still, it is not evident (or true in general) that $\overline{\Beff(z, R)}$ is \emph{connected}.
	
	We do notice that $\overline{\Beff(z, R)} \subset \Lambda_a(z, R)$ as
	\[
		a(x, z) = \hm_{x, z}(x)\Reff(x \leftrightarrow z) \leq \Reff(x \leftrightarrow z),
	\]
	by Corollary \ref{cor: the harmonic measure for two points is well-defined}. See also Figure \ref{fig: level-sets} for a schematic picture.
	
	\begin{figure}[h]
		\begin{center}
			\includegraphics[width=0.5\linewidth]{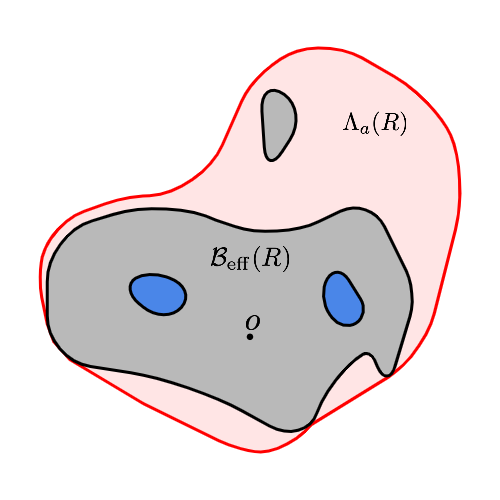}
			\caption{A schematic drawing. In dark gray, we see the set $\Beff(R)$. The blue parts are $\overline{\Beff(R)} \setminus \Beff(R)$. The red area (and everything inside) is then the sublevel set $\Lambda_a(R)$. }
			\label{fig: level-sets}
		\end{center}
	\end{figure}
	
	We thus get that the sets $\Lambda_a(R)$ are more regular than the sets $\Beff(R)$ and if $G$ is planar, they correspond to euclidean simply connected sets.
	
	In this section, we are interested in some properties of $\Lambda_a(R)$, that we will need to prove our Harnack inequalities. We now state the main result, which shows that $\lim_{z \to \infty} a(z, x) = \infty$, under the additional assumption that the underlying rooted graph is random and (stationary) reversible.
	
	\begin{proposition} \label{prop: PK level-sets are exhaustion}
		Suppose $(G, o)$ is a reversible random graph, that is a.s. recurrent and for which the potential kernel is a.s. well defined. Almost surely, the sets $\Lambda_a(z, R)$ are finite for each $R \geq 1$ and all $z \in v(G)$, and hence $(\Lambda_a(z, R))_{R \geq 1}$ defines an exhaustion of $G$.
	\end{proposition}
	
	Although we expect this proposition to hold for all graphs where the potential kernel is well defined, we do not manage to prove the general case. In addition, the proof actually yields something slightly stronger which may not necessarily hold in full generality.
	
	Note also that for all $R \geq 0$ we have $v(G) \setminus \Lambda_a(R) $ is non-empty because $x \mapsto a(x, o)$ is unbounded (to see this, assume it is bounded and use recurrence and the optional stopping theorem to deduce that $a(x,o)$ would be identically zero, which is not possible since the Laplacian is nonzero at $o$). We introduce the following definition, that we will use throughout the remaining document.
	
	\begin{definition} \label{def: delta-good}
		Let $\delta \in [0, 1]$ and $x \in v(G)$.
		\begin{itemize}
			\item We call $x$ $(\delta, o)$-good if $\hm_{o, x}(x) \geq \delta$. We will omit the notation for the root if it is clear from the context.
			\item We call the rooted graph $(G, o)$ $\delta$-good if for all $\epsilon > 0$, there exist infinitely many $(\delta - \epsilon, o)$-good points.
			\item We call the rooted graph $(G, o)$ uniformly $\delta$-good if all vertices are $(\delta, o)$-good.
		\end{itemize}
	\end{definition}
	Note that if the graph $(G, o)$ is uniformly $\delta$-good for some $\delta > 0$, then actually $\Lambda_a(\delta R) \subset \Beff(R)$, so that the sets $\Lambda_a(\delta R)$ are finite for each $R$. It turns out that the graph $(G, o)$ being $\delta$-good is also enough, which is the content of Lemma \ref{lemma: delta-good implies level-sets} below.
	
	Although the definition of $\delta$-goodness is given in terms of rooted graphs $(G, o)$, the next (deterministic) lemma shows that the definition is actually invariant under the choice of the root, and hence we can can omit the root and say ``$G$ is $\delta$-good'' instead.
	
	\begin{lemma} \label{lemma: delta-good is indep of root}
		Suppose $\delta > 0$ is such that $(G, o)$ is $\delta$-good, then also $(G, z)$ is $\delta$-good for each $z \in v(G)$.
	\end{lemma}
	\begin{proof}
		Fix $z \in v(G)$ and let $\delta > 0$ be such that $(G, o)$ is $\delta$-good. Fix $0 < \epsilon < \delta$ and denote by $G_{\alpha, o}$ the set of $(\alpha, o)$-good vertices. Take $\epsilon_1, \epsilon_2 > 0$ such that $\epsilon_1 + \epsilon_2 = \epsilon$. Then $G_{\delta - \epsilon_1, o}$ has infinitely many points by assumption.
		
		By Corollary \ref{cor: harmonic measure are local, not global}, we can take $R_0 := R_0(z, o, \epsilon_2)$ so large that for all $x \notin B(o, R_0)$ we have
		\[
			|\hm_{x, z}(x) - \hm_{x, o}(x)| < \epsilon_2.
		\]
		This implies that any vertex $x \in G_{\delta - \epsilon_1, o} \cap B(o, R_0)^c$ must in fact be $(\delta - \epsilon, z)$-good since $\epsilon = \epsilon_1 + \epsilon_2$. This shows the desired result as $\epsilon$ was arbitrary.
	\end{proof}

	The next lemma shows the somewhat interesting result that reversible environments are always $\delta$-good, with $\delta$ arbitrary close to $\frac{1}{2}$.
	
	\begin{lemma} \label{lemma: reversible graphs are delta-good}
		Suppose that $(G, o, X_1)$ is a reversible environment on which the potential kernel is a.s. well defined. Then for each $\delta < \frac{1}{2}$, a.s., $(G, o)$ is $\delta$-good.
	\end{lemma}

	\begin{proof}
		In this proof we will write $\mathbf{P}, \mathbf{E}$ to denote probability respectively expectation with respect to the law of the random rooted graph $(G, o)$. In compliance with the rest of the document, we will write $\PP, \E$ to denote the probability respectively expectation w.r.t. the law of the simple random walk, conditional on $(G, o)$.
		
		By Lemma \ref{lemma: delta-good is indep of root}, we note that $(G, o)$ being $\delta$-good is independent of the root and hence for each $\delta > 0$, the event
		\[
			\mathcal{A}_{\delta} = \{(G, o) \text{ is } \delta\text{-good}\}
		\]
		is invariant under re-rooting, that is
		\[
			(G, o) \in \mathcal{A}_{\delta} \iff (G, x) \in \mathcal{A}_{\delta} \text{ for all } x \in v(G).
		\]
		A natural approach to go forward, would be to use that any unimodular law is a mixture of ergodic laws \cite[Theorem 4.7]{AldousLyonsUnimod2007}. We will not use this, as there is an even simpler argument in this case.
		
		Nevertheless, we will use the invariance under re-rooting to prove that $\mathcal{A}_{\delta}$ has probability one. Suppose, to the contrary, that the event $\mathcal{A}_{\delta}$ does not occur with probability one, so that $\bf{P}(\mathcal{A}_{\delta}) \in [0, 1)$. Then we can condition the law $\bf{P}$ on $\mathcal{A}_{\delta}^c$ to obtain again a reversible law $\bf{P}(\cdot \mid \mathcal{A}_{\delta}^c)$ (it is here that we use the invariance under re-rooting of $\mathcal{A}_{\delta}$), under which $\mathcal{A}_{\delta}$ has probability zero. However, we will show that $\bf{P}(\mathcal{A}_\delta) > 0$ always holds when $\delta < \frac{1}{2}$, independent of what the exact underlying reversible law $\bf{P}$ is - as long as the potential kernel is a.s. well defined and the graph is a.s. recurrent. Now, this implies that we actually need to have $\bf{P}(\mathcal{A}_{\delta}) = 1$, which is the desired result.
		
		Fix $\delta < \frac{1}{2}$. We thus still need to prove that $\bf{P}(\mathcal{A}_\delta) > 0$, which we do by contradiction. Assume henceforth that $\bf{P}(\mathcal{A}_{\delta}) = 0$. By reversibility, we get for each $n \in \NN$ the equality
		\[
			\mathbf{E}[\hm_{o, X_n}(X_n)] = \mathbf{E}[\hm_{X_n, o}(o)] = \frac{1}{2},
		\]
		due to the fact that $(G, o, X_n)$ has the same law as $(G, X_n, o)$, which is reversibility (here, the expectation is both with respect to the environment and the walk).
		
		As $\bf{P}(\mathcal{A}_{\delta}) = 0$, we can assume that a.s. there exists a (random) $N = N(G, o) \in \NN$, such that for all $x \not \in B(o, N)$ we have
		\[
			\hm_{x, o}(x) \leq \delta
		\]
		Also, as the environment is a.s. null-recurrent, we have that
		\[
			\PP(X_n \text{ in } B(o, N)) \to 0,
		\]
		whenever $n \to \infty$.  Moreover, notice that for each $n$ we have
		\[
			\E[\hm_{o, X_n}(X_n)] \leq \PP(X_n \text{ not in } B(o, N))\delta + \PP(X_n \text{ in } B(o, N)).
		\]
		Since $\hm_{o, X_n}(X_n) \in [0, 1]$, we can apply Fatou's lemma (applied to \emph{just} the expectation with respect to the law of $(G, o)$, so that we can use the just found inequality) from which we deduce that
		\[
			\frac{1}{2} = \limsup_{n \to \infty} \bf{E}[\hm_{o, X_n}(X_n)] \leq \delta,
		\]
		which is a contradiction as $\delta < \frac{1}{2}$.
	\end{proof}
		
	We next show that for any $\delta$-good (rooted) graph, the set $\Lambda_a(R)$ is finite for each $R \geq 1$. Combined with Lemma \ref{lemma: reversible graphs are delta-good}, this implies Proposition \ref{prop: PK level-sets are exhaustion} in case of reversible environments. However, Lemma \ref{lemma: reversible graphs are delta-good} shows more than just this fact. Indeed, $\Lambda_a(R)$ being finite need not imply that $(G, o)$ is $\delta$-good for some $\delta > 0$.
	
	\begin{lemma} \label{lemma: delta-good implies level-sets}
		If $(G, o)$ is $\delta$-good for some $\delta > 0$, then $\Lambda_a(o, R)$ is finite for each $R \geq 1$.
	\end{lemma}
		
	\begin{proof}
		Let $\delta > 0$ and suppose that $G$ is $\delta$-good. We will show that for each $R \geq 1$, there exists an $M \geq 1$ such that for all $x \notin B(o, M)$ we have
		\[
			a(x, o) \geq \frac{(\delta)^2 R}{8}.
		\]
		This implies the final result.
		
		By assumption on $\delta$-goodness, for each $R \geq 1$ there exists a vertex $x_R \notin \Beff(o, R)$ such that
		\[
			\hm_{x_R, o}(x_R) \geq \frac{\delta}{2}
		\]
		This implies by Corollary \ref{cor: the harmonic measure for two points is well-defined} that $a(x_{R}, o) \geq \frac{\delta}{2} \Reff(o \leftrightarrow x_R) \geq \frac{\delta R}{2}$.
		
		Fix $R \geq 1$ and define the set $B_R = \{o, x_{R}\}$. By Theorem \ref{theorem: gluing expression}, we get for all $x$ the decomposition
		\[
			a(x, o) = q_{B_R}(x) + \E_x[a(X_{T_{B_R}}, o)],
		\]
		where $q_{B_R}(\cdot)$ is the potential kernel on the graph $G_{B_R}$, which we recall is the graph $G$, with $B_R$ glued together. Since potential kernels are non-negative, we can focus our attention to the right-most term.
		
		Take $M = M(o, x_R, \delta)$ so large that for all $x \notin B(o, M)$
		\[
			|\hm_{B_R}(x_R) - \PP_x(X_{T_{B_R}} = x_R)| \leq \frac{\delta}{4},
		\]
		which is possible as the potential kernel is well defined, see Proposition \ref{prop: properties of the PKs} and Corollary \ref{cor: the harmonic measure for two points is well-defined}. We deduce that for all $x \notin B(o, M)$
		\[
			a(x, o) \geq \E_x[a(X_{T_{B_R}}, o)] \geq \frac{(\delta)^2 R}{8},
		\]
		as desired.
	\end{proof}

	\section{Two Harnack inequalities} \label{section: Harnack Inequaltiy}
	We are now ready to prove the equivalence between (c) and (a). The first part of this section deals with a classical Harnack inequality, whereas the second part of this section provides a variation thereof, where the functions might have a single pole. The first Harnack inequality (Theorem \ref{theorem: harnack inequality} below) does not involve Theorem \ref{T:main_equiv}.
	
	Recall that $\Lambda_a(z, R)$ is the sublevel set $\{x \in v(G): a(x, z) \leq R\}$ (for $R$ not necessarily integer valued) and that $a(z, x)$ defines a quasi distance on $G$. Also recall the notation $\Beff(z, R) = \{x: \Reff(z \leftrightarrow x) \leq R\}$, for the (closed) ball with respect to the effective resistance distance.
	
	\subsection{The standing assumptions}
	Throughout this section we will work with deterministic graphs $G$, which satisfy a certain number of assumptions. To be precise, we will say that $G$ satisfies the \textbf{standing assumptions} whenever it is recurrent, the potential kernel is well defined and the level sets $(\Lambda_a(z, R))_{R \geq 1}$ are finite for some (hence all) $z \in v(G)$.
	
	\begin{remark}
		Proposition \ref{prop: PK level-sets are exhaustion} implies that any unimodular random graph with $\bf{E}[\deg(o)] < \infty$ that is a.s. recurrent and for which the potential kernel is a.s. well defined, the level sets $\Lambda_a(z, R)$ are finite for all $R$ and $z \in v(G)$ (and so satisfies the standing assumptions). Note for instance that the UIPT therefore satisfies the standing assumptions.
	\end{remark}
	
	\subsection{Elliptic Harnack Inequality}
	
We first show that under the standing assumptions, a version of
the elliptic Harnack inequality holds, where the constants are uniform over all graphs that satisfy the standing assumptions. Recall the definition of the ``hull'' $\overline{\Beff(z, R)}$ introduced in Section \ref{section: Level-sets of the PK}.
	
	\begin{theorem}[Harnack Inequality] \label{theorem: harnack inequality}
		There exist $M, C > 1$ such that the following holds. Let $G$ be a graph satisfying the standing assumptions. For all $z \in v(G)$, all $R \geq 1$ and all $h: \Lambda_a(z, MR) \cup \partial \Lambda_a(z, MR) \to \RR_+$ that are harmonic on $\Lambda_a(z, MR)$ we have
		\begin{equation} \label{eq: harnack inequality} \tag{H}
			\max_{x \in \overline{\Beff(z, R)}} h(x) \leq C \min_{x \in \overline{\Beff(z, R)}} h(x)
		\end{equation}
	\end{theorem}

	\begin{remark}
		In case the rooted graph $(G, o)$ is in addition uniformly $\delta$-good for some $\delta$ (that is, $\hm_{x, o}(x) \geq \delta$ for each $x$, see Definition \ref{def: delta-good}), then we have that
		\[
			\Lambda_a(\delta R) \subset \Beff(R) \subset \Lambda_a(R),
		\]
		and hence the Harnack inequality above becomes a standard ``elliptic Harnack inequality'' for the graph equipped with the effective resistance distance. (As will be discussed below, we conjecture that many infinite models of random planar maps, including the UIPT, satisfy the property of being $\delta$-good for some nonrandom $\delta>0$.)
	\end{remark}
	
	\subsubsection*{The harmonic exit measure.}
	In the proof, we fix the root $o \in v(G)$, but it plays no special role. Define for $k \in \NN$, $x \in \Lambda_a(k)$ and $b \in \partial \Lambda_a(k)$ the ``harmonic exit measure''
	\[
		\mu_k(x, b) = \PP_x(X_{T_{k}} = b),
	\]
	where $T_k$ is the first hitting time of $\partial \Lambda_a(k)$. We will write
	\begin{equation} 
		\Gr_k(x, y) := \Gr_{\Lambda_a(k)^c}(x, y)
	\end{equation}
	where we recall the definition of the Green function in \eqref{eqdef: green function}.
	The following proposition shows that changing the starting points $x, y \in \Beff(R)$, does not significantly change the exit measure $\mu_k(\cdot, b)$. The Harnack inequality will follow easily from this proposition (in fact, it is equivalent).
	
	\begin{proposition} \label{prop: hitting measure of SRW converges}
		There exist constants $\tilde{C}, M > 1$ such that for all $G$ satisfying the standing assumptions, all $R \geq 1$ and all $x, y \in \partial \overline{\Beff(R)}$ we have
		\[
			\frac{1}{\tilde{C}}\mu_{MR}(y, b) \leq \mu_{MR}(x, b) \leq \tilde{C}\mu_{MR}(y, b)
		\]
		for each $b \in \partial \Lambda_a(MR)$.
	\end{proposition}
	
	We first prove the following lemma, giving an estimate on the number of times the simple random walk started from $x$ visits $y$, before exiting the set $\Lambda_a(MR)$.
	
	\begin{lemma} \label{sublemma harnack: green function SRW}
		For all $M_0 > 1$, there exists an $M > M_0$ and $C = C(M, M_0) > 1$ such that for all $G$ satisfying the standing assumptions and for all $R \geq 1$ we have
		\[
			\frac{R}{C} \leq \frac{\Gr_{MR}(x, y)}{\deg(y)} \leq CR
		\]
		for all $x \in \partial \Lambda_a(M_0R)$ and $y \in \partial \overline{\Beff(R)}$.
	\end{lemma}
	
	\begin{proof}
		Fix $M_0 > 1$ let $M > M_0 + 2$ for now. Let $G$ be any graph satisfying the standing assumptions. Let $R \geq 1$, take $x \in \partial \Lambda_a(M_0R)$ and $y \in \partial \overline{\Beff(R)}$.
		Notice that, by Lemma \ref{lem: green function finite set and PK}, we can write
		\begin{equation} \label{eq: expression of stopped green function in PK}
			\frac{\Gr_{MR}(x, y)}{\deg(y)} = \E_x[a(X_{T_{MR}}, y)] - a(x, y).
		\end{equation}
		Let $z \in \Lambda_a(MR)$. Recalling that $a(\cdot, \cdot)$ is a quasi metric that satisfies the triangle inequality due to Proposition \ref{prop: green function and PK}, we have, by assumption on $x$ and $y$ and the expression for the potential kernel in terms of harmonic measure and effective resistance (Corollary \ref{cor: the harmonic measure for two points is well-defined}), that
		\begin{equation} \label{eq: upper bound a(y, z)}
			a(z, y) \leq a(z, o) + a(o, y) \leq MR + \Reff(o \leftrightarrow y) = (M + 1)R.
		\end{equation}
		Going back to \eqref{eq: expression of stopped green function in PK} and upper-bounding $-a(x, y) \leq 0$, we find the desired upper bound:
		\[
			\frac{\Gr_{MR}(x, y)}{\deg(y)} \leq (M + 1)R.
		\]
		
		For the lower bound, fix again $z \in \partial \Lambda_a(MR)$. From Theorem \ref{theorem: gluing expression} (and the fact that $\Beff(R) \subset \Lambda_a(R)$) we obtain the equality
		\[
			a(z, y) - \E_z[a(X_{T_{R}}, y)] = a(z, o) - \E_z[a(X_{T_{R}}, o)].
		\]
		It follows that
		\begin{equation} \label{eq: lower bound a(z, y)}
			a(z, y) \geq a(z, o) - \E_z[a(X_{T_R}, o)] = (M - 1)R.
		\end{equation}
		On the other hand, invoking the triangle inequality (as in \eqref{eq: upper bound a(y, z)}), we have
		\[
			a(x, y) \leq a(x, o) + a(o, y) \leq (M_0 + 1)R.
		\]
		The lower-bound now follows from \eqref{eq: expression of stopped green function in PK} and \eqref{eq: lower bound a(z, y)} as
		\[
			\frac{\Gr_{MR}(x, y)}{\deg(y)} \geq (M - 1)R - (M_0 + 1)R = (M - M_0 - 2)R.
		\]
		Since $M > M_0 + 2$, we can take $C = C(M, M_0)$ such that we get the result.
	\end{proof}
	
	\begin{proof}[Proof of Proposition \ref{prop: hitting measure of SRW converges}]
		Take $M_0 > 1, M = M(M_0)$ and $C > 1$ as in Lemma \ref{sublemma harnack: green function SRW}. Let $G$ be a graph satisfying the standing assumptions. Fix $R \geq 1$ and let $x, y \in \partial \overline{\Beff(R)}$. For $b \in \Lambda_a(MR)$ we use the last-exit decomposition (Lemma \ref{lemma: last exit decomposition}) to see
		\begin{align*}
			\mu_{MR}(x, b) = \sum_{z \in \partial \Lambda_a(M_0R)} \frac{\Gr_{MR}(x, z)}{\deg(z)} \deg(z)\PP_z(X_{T_{MR}} = b; T_{M_0R}^+ < T_{MR}).
		\end{align*}
		By Lemma \ref{sublemma harnack: green function SRW}, we have for each $z \in \partial \Lambda_a(M_0R)$
		\[
			\frac{\Gr_{MR}(z, x)}{\deg(x)} \leq CR \leq C^2 \frac{\Gr_{MR}(z, y)}{\deg(y)}.
		\]
		We thus get, defining $\tilde{C} = C^2$, and using $\deg(\cdot)$-reversibility of the simple random walk that
		\begin{align*}
			\mu_{MR}(x, b) &\leq \tilde{C} \sum_{z \in \partial \Lambda_a(M_0R)} \frac{\Gr_{MR}(y, z)}{\deg(z)} \deg(z)\PP_z(X_{T_{MR}} = b; T_{M_0R}^+ < T_{MR})\\
			&= \tilde{C}\mu_{MR}(y, b),
		\end{align*}
		showing the final result.
	\end{proof}
	
	\begin{proof}[Proof of Theorem \ref{theorem: harnack inequality}]
	The proof of Theorem \ref{theorem: harnack inequality} is easy now. Indeed, let $C, M > 1$ large enough, as in Proposition \ref{prop: hitting measure of SRW converges} and take any graph $G$ satisfying the standing assumptions and $R \geq 1$. Take $h : \Lambda_a(MR) \cup \partial \Lambda_a(MR) \to \RR_+$ a function harmonic on $\Lambda_a(MR)$. Using the maximum principle for harmonic functions, we deduce that it is enough to prove
	\[
		\max_{x \in \partial \overline{\Beff(R)}} h(x) \leq C \min_{x \in \partial \overline{\Beff(R)}} h(x).
	\]
	Take $x, y \in \partial \overline{\Beff(R)}$. By optional stopping and Proposition \ref{prop: hitting measure of SRW converges} we have
	\begin{align*}
		h(x)  = \E_x[h(X_{T_{MR}})] &= \sum_{b \in \partial \Lambda_a(MR)} h(b) \mu_{MR}(x, b) \\
		&\leq \tilde{C} \sum_{b \in \partial \Lambda_a(MR)} h(b) \mu_{MR}(y, b) = \tilde{C}h(y),
	\end{align*}
	showing the result.
	\end{proof}
	
	\subsection{(a) implies (c): anchored Harnack inequality} \label{sec: HI implies PK}
	Sometimes, one wants to apply a version of the Harnack inequality to functions that are harmonic on a big ball, but not in some vertex inside this ball (the pole). Clearly, we can only hope to compare the value of harmonic function in points that are ``far away'' from the pole, say on the boundary of a ball centered at the pole.
	
	This ``anchored'' inequality does not always follow from the Harnack inequality as stated in Theorem \ref{theorem: harnack inequality}. As an example, think of the graph $\ZZ$ with nearest neighbor connections. Pick any two positive real numbers $\alpha, \beta$ satisfying $\alpha + \beta = 1$. Then the function $h$ that maps $x$ to $\alpha(-x)$ when $x$ is negative and to $\beta x$ when $x$ is positive, is harmonic everywhere outside of $0$, with $\Delta h(0) = 1$. This implies that no form of ``anchored Harnack inequality'' can hold.
	
	We next present a reformulation of (a) implies (c) in Theorem \ref{T:main_equiv}. We will use it to prove results for the ``conditioned random walk'' as introduced in Section \ref{section: The CRW}.
	
	\begin{theorem}[Anchored Harnack Inequality] \label{theorem: harnack with a pole}
		There exists a $C < \infty$ such that the following holds. Let $G$ be a graph satisfying the standing assumptions. For $z \in v(G)$, $R \geq 1$ and all $h: v(G) \to \RR_+$ that are harmonic outside of $z$ and satisfy $h(z) = 0$, we have
		\begin{equation} \label{eqdef: harnack inequality with pole. } \tag{aH}
			\max_{x \in \partial \Lambda_a(z, R)} h(x) \leq C \min_{x \in \partial \Lambda_a(z, R)} h(x).
		\end{equation}
	\end{theorem}

	\begin{remark}
		Actually, we will prove that for each $z \in v(G)$ and $R \geq 1$, there exists $\Psi_z(R) \geq R$ such that for all harmonic functions $h:\Lambda_a(z, \Psi_z(R)) \cup \partial \Lambda_a(z, \Psi_z(R)) \to \RR_+$ that are harmonic on $\Lambda_a(z, \Psi_z(R)) \setminus \{z\}$ and $h(z) = 0$, we have
		\[
			\max_{x \in \partial \Lambda_a(z, R)}h(x) \leq  C \min_{x \in \partial \Lambda_a(z, R)} h(x).
		\]
		As before, if the graph is uniformly $\delta$-good for some $\delta > 0$, we can actually take $\Psi_z(R) = MR$ for some $M = M(\delta)$ depending \emph{only} on $\delta$.
	\end{remark}

	\subsubsection*{Proof of Theorem \ref{theorem: harnack with a pole}}
	The proof will be somewhat similar to the proof of Theorem \ref{theorem: harnack inequality}. Again, we will prove it for the vertex $o$ to simplify our writing, but it will not matter which vertex we choose. For $k \in \NN$, we will write again $T_k = T_{\Lambda_a(k)^c}$ for the first time the random walk exists the sublevel-set $\Lambda_a(k)$. Fix $k \in \NN$ and $x \in \Lambda_a(k)$. Define the exit measure
	\[
		\nu_k(x, b) = \PP_x(X_{T_k} = b, T_k < T_o),
	\]
	for $b \in \partial \Lambda_a(k)$. We begin by showing that, taking $x, y$ in $\Lambda_a(R)$, the exit measures $\nu_k(x, \cdot)$ and $\nu_k(y, \cdot)$ are similar up to division by $a(x, o), a(y, o)$ respectively, when $k$ is large enough. Although it might seem at first slightly counterintuitive that that we need to divide by $a(x, o)$, this actually means that the \emph{conditional} exit measures $\PP_w(X_{T_k} = b \mid T_{k} < T_o)$ for $w = x, y$ are comparable.
	
	\begin{proposition} \label{prop: conditioned exit measure mixes}
		There exists a $C < \infty$ such that for each $R \geq 1$, there exists a constant $\Psi(R) \geq R$ such that for all $x, y \in \Lambda_a(R) \setminus \Lambda_a(1)$ and all $b \in \partial \Lambda_a(\Psi(R))$ we have
		\[
			\frac{\nu_{\Psi(R)}(x, b)}{a(x, o)} \leq C \frac{\nu_{\Psi(R)}(y, b)}{a(y, o)}.
		\]
	\end{proposition}
	
	In order to prove this proposition, we will first prove a few preliminary lemma's. The next result offers bounds on the probability that the random walk goes ``far away'' before hitting $o$ in terms of the potential kernel.
	\begin{lemma} \label{lemma: TM < To bounds}
		For each $z \in v(G) \setminus \Lambda_a(1)$ and all $M > a(z, o)$, we have
		\[
			\frac{a(z, o)}{M + 1} \leq \PP_z(T_M < T_o) \leq \frac{a(z, o)}{M}.
		\]
	\end{lemma}
	\begin{proof}
		This is a straightforward consequence of the optional stopping theorem. Indeed,
		\[
			a(z, o) = \E_z[a(T_{M} \wedge T_o, o)]
		\]
		and since $M \leq a(w, o) \leq M + 1$ for each $w \in \partial \Lambda_a(M)$ and $a(o, o) = 0$, we find
		\[
			\frac{a(z, o)}{M + 1} \leq \PP_z(T_M < T_o) \leq \frac{a(z, o)}{M},
		\]
		which are the desired bounds.
	\end{proof}

	\begin{lemma}
		For each $R \geq 1$, there exist $M, M_0 > R$ such that for all $x \in \Lambda_a(R)$ and $z \in \Lambda_a(M_0)$,
		\[
			\frac{1}{10} \leq \frac{\Gr_{B_M}(z, x)}{\deg(x)a(x, o)} \leq 2,
		\]
		where $B_M = \{o\} \cup \Lambda_a(M)^c$.
	\end{lemma}
	\begin{proof}
		Fix $R \geq 1$ and $x, y \in \Lambda_a(R)$. Take $M_0 = M_0(R)$ at least so large that for \emph{all} $w \notin \Lambda_a(M_0)$ we have
		\begin{equation} \label{eq: bound P-hit / hm}
			\frac{1}{2} \leq \frac{\PP_w(T_x < T_o)}{\hm_{x, o}(x)} \leq 2,
		\end{equation}
		which is possible due to Corollary \ref{cor: a(x, z) - a(y, z) goes to 0 as z to infinity}. Fix then $M = 5M_0$ and $B_M = \{o\} \cup \Lambda_a(M)^c$.
		
		Take $z \in \Lambda_a(M_0)$. By choice of $M$ and Lemma \ref{lemma: TM < To bounds}, we have
		\begin{equation} \label{subeq: Pz(TM < To) bound}
			\PP_z(T_M < T_o) \leq \frac{M_0}{M} \leq \frac{1}{5}.
		\end{equation}
		Using the strong Markov property of the walk we get
		\begin{equation} \label{eq: Go_m bound 1}
			\begin{aligned}
			&\Gr_{B_M}(z, x) =  \Gr_o(z, x) - \PP_z(T_M < T_o)\sum_{b \in \partial \Lambda_a(M)} \PP_z(X_{T_{M}} = b \mid T_{M} < T_o)\Gr_o(b, x).
			\end{aligned}
		\end{equation}
		The definition of the Green function and Corollary \ref{cor: the harmonic measure for two points is well-defined} allow us to write
		\[
			\frac{\Gr_o(z, x)}{\deg(x)} = \PP_z(T_x < T_o) \Reff(x \leftrightarrow o) \quad \text{and} \quad a(x, o) = \hm_{x, o}(x)\Reff(x \leftrightarrow o),
		\]
		which implies that
		\[
			\frac{\Gr_{o}(z, x)}{\deg(x)a(x, o)} = \frac{\PP_z(T_x < T_o)}{\hm_{x, o}(x)} \quad \text{and} \quad \frac{\Gr_o(b, x)}{\deg(x)a(x, o)} = \frac{\PP_b(T_x < T_o)}{\hm_{x, o}(x)},
		\]
		for each $b \in \Lambda_a(M)$. Thus \eqref{eq: Go_m bound 1} is equivalent to
		\begin{equation} \label{eq: Go_m bound 2}
			\begin{aligned}
			\frac{\Gr_{B_M}(z, x)}{\deg(x)a(x, o)} = \frac{\PP_z(T_x < T_o)}{\hm_{x, o}(x)} - \PP_z(T_M < T_o) \sum_{b \in \partial \Lambda_a(M)} \PP_z(X_{T_M} = b \mid T_M < T_o) \frac{\PP_b(T_x < T_o)}{\hm_{o, x}(x)}.
			\end{aligned}
		\end{equation}
		Hence, by \eqref{subeq: Pz(TM < To) bound} and using \eqref{eq: bound P-hit / hm} twice with $w = z$ and $w = b$ respectively in \eqref{eq: Go_m bound 2} we get
		\[
			\frac{1}{2} - \frac{2}{5} \leq \frac{G_{o, \partial_M}(z, x)}{\deg(x)a(x, o)} \leq 2,
		\]
		which is the desired result.
	\end{proof}
	\begin{proof}[Proof of Proposition \ref{prop: conditioned exit measure mixes}.]
		Just as in the proof of Proposition \ref{prop: hitting measure of SRW converges}, we use the last-exit decomposition to see
		\[
			\frac{\nu_M(x, b)}{a(x, o)} = \sum_{z \in \partial \Lambda_a(k)} \frac{\Gr_{o, \partial_M}(x, z)}{a(x, o)\deg(z)} \deg(z) \PP_z(X_{T_M} = b; T_M < T_k^+).
		\]
		This implies that
		\[
			\frac{\nu_M(x, b)}{a(x, o)} \leq 20 \frac{\nu_M(y, b)}{a(y, o)}.
		\]
		We are left to define $\Psi(R) = M$ and $C = 20$ to obtain the desired result.
	\end{proof}

	\begin{wrong-old}	
	\begin{lemma}
		Let $R \geq 1$. For each $x \in \Beff(R)$ we have that for all $z \notin \Lambda_a(2R)$
		\[
		|\PP_z(T_x < T_o) - \hm_{o, x}(x)| \leq 4\hm_{o, x}(x).
		\]
	\end{lemma}
	\begin{proof}
		We follow roughly the lines of \cite[Theorem 3.17]{PopovRW}, albeit somewhat simplified to our setting and yielding much weaker results. We begin by noticing that
		\[
		\PP_z(T_x < T_o) - \hm_{x, o}(x) = \sum_{w \in \{o, x\}} \PP_w(T^+_{\{o, x\}} = w)a(w, z) - a(x, z).
		\]
		
		Fix $R \geq 1$ and $x \in \Beff(R) \subset \Lambda_a(R)$, so that also $a(o, x) = \hm_{x, o}(x)\Reff(o \leftrightarrow x) \leq R$. Consider the set $V = \partial \Lambda_a(2R)$, from which we deduce that for all $v \in V$ we have
		\begin{equation} \label{eq: lowerbound pk diff}
		a(v, o) - a(x, o) \geq 2R - 1 - (R + 1) = R - 2,
		\end{equation}
		where we used the triangle inequality together with the fact that $a(z, w) \leq 1$ whenever $z \sim w$.
		
		Next, let $z \notin \Lambda_a(3R)$ say. Then, by the triangle inequality we obtain for each $v \in V$ and $w \in \{o, x\}$ that
		\[
		a(w, z) \leq a(w, o) + a(o, v) + a(v, z),
		\]
		and hence, by choice of $V$ and $x$,
		\begin{equation} \label{eq: upperbound pk diff}
		a(v, z) - a(w, z) \leq a(v, o) + a(o, w) \leq 3R + 2.
		\end{equation}
		
		Write $A = \{o, x\}$ and apply optional stopping (which holds as  $T_V < T_z$) to see that
		\begin{align*}
		a(w, z) &= \E_w[a(X_{T_{A^+} \wedge T_{V}}, z)] \\
		&=
		\end{align*}
		
	\end{proof}

	\begin{proof}
		Let $R \geq 1$ Fix $x, y \in \Lambda_a(R) \setminus \Lambda_a(1)$. For $M \in \RR_+$ we will write $\partial_M$ for the identification of $\partial \Lambda_a(M)$ in $G$. Using Corollaries \ref{cor: a(x, z) - a(y, z) goes to 0 as z to infinity} and \ref{cor: the harmonic measure for two points is well-defined}, we can pick $k = k(R)$ so large that
		\[
			\left|\frac{\Gr_o(z, w)}{\deg(w)} - a(w, o)\right| = |\PP_z(T_w < T_o) - \hm_{w, o}(w)|\Reff(o \leftrightarrow w) \leq \frac{1}{4},
		\]
		for all $z \in \partial \Lambda_a(k)$ and all $w \in \Lambda_a(R)$. Also notice that for each $M > k$, we can write, using the strong Markov property,
		\[
			\Gr_o(z, w) = \Gr_{o, \partial_M}(z,w) + \PP_z(T_{\partial_M} < T_o) \PP_{\partial_M}(T_w < T_o)\Gr_o(w, w)
		\]
		for all $w \in \Lambda_a(R)$ and $z \in \partial \Lambda_a(k)$. Thus, we are allowed to take $M = M(k)$ so large that
		\[
			\left|\frac{\Gr_o(z, w)}{\deg(w)} - \frac{\Gr_{o, \partial_M}(z,w)}{\deg(w)}\right| \leq \frac{1}{4}a(w, o),
		\]
		holds for all $z \in \Lambda_a(k)$ and $w \in \Lambda_a(R)$. In particular, using $\deg(\cdot)$-reversibility of the walk and the triangle inequality, we deduce
		\begin{equation} \label{eq: GMo to a(x, o) bound}
			\left|\frac{\Gr_{o, \partial_M}(x, z)}{\deg(z)} - a(x, o)\right| \leq \frac{1}{2}a(x, o).
		\end{equation}
		Let $b \in \partial \Lambda_a(M)$. By the last-exit decomposition (Lemma \ref{lemma: last exit decomposition}) we have
		\[
			\nu_M(x, b) = \sum_{z \in \partial \Lambda_a(k)} \frac{\Gr_{o, \partial_M}(x, z)}{\deg(z)} \deg(z) \PP_z(X_{T_M} = b; T_M < T_k^+).
		\]
		Combining this with \eqref{eq: GMo to a(x, o) bound} we obtain
		\begin{align*}
			\nu_M(x, b) &= \sum_{z \in \partial \Lambda_a(k)} \left[\frac{\Gr_{o, \partial_M}(x, z)}{\deg(z)} - a(x, o) + a(x, o)\right]\deg(z)\PP_z(X_{T_M} = b; T_M < T_k^+),
		\end{align*}
		from which we deduce
		\[
			\frac{1}{2}\sum_{z \in \partial \Lambda_A(k)}\deg(z)\PP_z(X_{T_M} = b; T_M < T_k^+) \leq \frac{\mu_{M}(x, b)}{a(x, o)} \leq 2\sum_{z \in \partial \Lambda_A(k)}\deg(z)\PP_z(X_{T_M} = b; T_M < T_k^+)
		\]
		Since the upper and the lower bound are independent of $x \in \Lambda_a(R)$, we can take $C = 4$ to find
		\[
			\frac{\nu_M(x, b)}{a(x, o)} \leq C \frac{\nu_M(y, b)}{a(y, o)}.
		\]
		Setting $\Psi(R) = M$, we get the result.
	\end{proof}
	\end{wrong-old}

	Finishing the proof of Theorem \ref{theorem: harnack with a pole} is now straightforward. Indeed, we fix $C > 1$ and $\Psi$ as in Proposition \ref{prop: conditioned exit measure mixes}. Let $R \geq 1$ and $h: \Lambda_a(\Psi(R)) \to \RR_+$ harmonic outside $o$; with $h(o) = 0$. Fix $x, y \in \Lambda_a(R)$. By optional stopping, which holds as $\Lambda_a(\Psi(R))$ is finite,
	\begin{align*}
		h(x) &= \int_{\partial \Lambda_a(\Psi(R))} h(b) \nu_{\Psi(R)}(x, b) \\
		&\leq C\frac{a(x, o)}{a(y, o)} \int_{\partial \Lambda_a(\Psi(R))} h(b) \nu_{\Psi(R)}(y, b) = C \frac{a(x, o)}{a(y, o)}h(y).
	\end{align*}
	This shows the desired result when $x, y \in \partial \Lambda_a(R)$. \qed

	\subsection{(c) implies (a)}
	Let $(V_R)_R$ be any sequence of connected subsets of $v(G)$ satisfying $o \in V_{R} \subset V_{R + 1}$, $|V_R| < \infty$ for all $R$ and $\cup_{R \geq 1} V_R = v(G)$.
	
	\begin{proposition} \label{P: (c) implies (a)}
		Suppose that the (rooted) graph $(G, o)$ satisfies the anchored Harnack inequality with respect to the sequence $(V_R)_{R \geq 1}$ and some (non-random) constant $C$: for all $h: v(G) \to \RR_+$ harmonic outside possibly $o$ and such that $h(o) = 0$,
		\[
			\max_{x \in \partial V_R} h(x) \leq C \min_{x \in \partial V_R} h(x).
		\]
		In this case, the potential kernel $a(x, o)$ is well defined.
	\end{proposition}
	
	We take some inspiration from \cite{HarnackUniquenessPotential2015}, although the strategy goes back in fact to a paper of Ancona \cite{Ancona01}. Pick some sequence $e = (e_{R})_{R \geq 1}$ on $v(G)$ satisfying $e_R \in \partial V_R$.
	
	\begin{lemma} \label{lemma: harnack comparing on the boundary}
		Let $R \geq 1$ and suppose that $h, g$ are two positive, harmonic functions on $\Lambda_{a}(\Psi(R)) \setminus \{o\}$ vanishing at $o$. We have
		\[
			\max_{x \in V_R \setminus \{o\}} \frac{h(x)}{g(x)} \leq C^2 \frac{h(e_R)}{g(e_R)}.
		\]
	\end{lemma}
	\begin{proof}
		Fix $R \geq 1$ and let $h, g$ be as above. Write $T_R = T_{\partial V_R}$. By optional stopping, $h(o) = 0$ and the Harnack inequality, we get
		\[
			h(x) = \PP_x(T_R < T_o)\E_{x}[h(X_{T_R}) \mid T_R < T_o] \leq C h(e_R) \PP_x(T_R < T_o)
		\]
		for all $x \in V_R \setminus \{o\}$. Similarly, we obtain
		\[
			g(x) \geq \frac{1}{C}g(e_R)\PP_x(T_R < T_o)
		\]
		for $x \in V_R \setminus \{o\}$. Combining this, we find
		\[
			\frac{1}{C}\frac{h(x)}{h(e_R)} \leq \PP_x(T_R < T_o) \leq C\frac{g(x)}{g(e_R)},
		\]
		showing the final result.
	\end{proof}	
	
	\begin{proof}[Proof of Proposition \ref{P: (c) implies (a)}]
		We follow closely Section 3.2 in \cite{HarnackUniquenessPotential2015}. We will show that whenever $h_1, h_2:v(G) \to [0, \infty)$ are harmonic functions on $v(G) \setminus \{o\}$, vanishing at $o$, such that $h_1(e_1) = h_2(e_1)$, we have $h_1 = h_2$. The result then follows as we can pick $h_1(\cdot)$ and $h_2(\cdot)$ to be two subsequential limits of $a_{A_n}(\cdot, o)$ (for possibly different sequences $(A_n)$ going to infinity), and rescaling so that they are equal at $e_1$.
		
		Consider $h_1, h_2: v(G) \to [0, \infty)$ harmonic functions on $v(G) \setminus \{o\}$, vanishing at $o$. Assume without loss of generality that $h_1(e_1) = h_2(e_1) = 1$. By Lemma \ref{lemma: harnack comparing on the boundary} we get that there is some appropriate (large) $M$ which does not depend on $h_1, h_2$, for which
		\begin{equation} \label{eq: two harmonics devided by their boundary value}
			\frac{1}{M}\frac{h_1(x)}{h_1(e_R)} \leq \frac{h_2(x)}{h_2(e_R)} \leq M \frac{h_1(x)}{h_1(e_R)},
		\end{equation}
		for all $x \in V_R$ and $R \geq 1$. It follows that (setting $x = e_1$)
		\[
			\frac{1}{M} h_1(e_R) \leq h_2(e_R) \leq M h_1(e_R).
		\]
		Using this in \eqref{eq: two harmonics devided by their boundary value} and letting $R \to \infty$, we obtain
		\begin{equation}\label{boundh1h2}
			\frac{1}{M^2} \leq \frac{h_2(x)}{h_1(x)} \leq M^2,
		\end{equation}
		for all $x \in v(G) \setminus \{o\}$. Define recursively, for $i \geq 3$,
		\begin{equation}\label{E:hi}
			h_i(x) = h_{i - 1}(x) + \frac{1}{M^2 - 1}(h_{i - 1}(x) - h_1(x)).
		\end{equation}
		It is straightforward to check that $h_i$ is non-negative (as follows from an iterated version of \eqref{boundh1h2}) and harmonic outside $o$. Since $M$ did not depend on $h_1, h_2$, and because $h_i(e_1) = 1$ also, we obtain that
		\begin{equation}\label{E:boundhi}
			\frac{1}{M^2} \leq \frac{h_i(x)}{h_1(x)} \leq M^2.
		\end{equation}
		On the other hand, it is straightforward to check that the recursion \eqref{E:hi} can be solved explicitly to get:
		\[
			h_i(x) = \left( \frac{M^2}{M^2 - 1} \right)^{i - 2}(h_2(x) - h_1(x)) + h_1(x).
		\]
		Unless $h_1(x) = h_2(x)$, this grows exponentially, which is incompatible with \eqref{E:boundhi}.
		Therefore $h_1(x) = h_2(x)$.
	\end{proof}
	\begin{remark}\label{R:anchoredPK}
	The proof above makes it clear that if the potential kernel is uniquely defined (i.e. if (a) holds), then any function $h:v(G) \to \RR_+$ satisfying $\Delta h(x) = 0$ for all $x \in v(G) \setminus \{o\}$ and for which $h(o) = 0$, is of the form $\alpha a(x, o)$ for some $\alpha \geq 0$.
	\end{remark}

\begin{remark}\label{R:AHIEHI}
  If $G$ is reversible, and satisfies the anchored Harnack inequality, then it satisfies (a) as a consequence of the above. It therefore satisfies the standing assumptions: in particular, by Theorem \ref{theorem: harnack inequality} holds so it also satisfies the Elliptic Harnack Inequality (EHI). We have therefore proved that anchored Harnack inequality (AHI) $\implies$ (EHI) at least for reversible random graphs, which is not a priori obvious.
\end{remark}

	%
	\section{Random Walk conditioned to not hit the root} \label{section: The CRW}
	Let $(G, o)$ satisfy the \emph{standing assumptions}, i.e.,
it is recurrent, the potential kernel is well defined and the potential kernel tends to infinity.
In this section, we will define what we call the conditioned random walk (CRW), which is the simple random walk on $G$, conditioned to never hit the root $o$ (or any other vertex). Of course, a priori this does not make sense as the event that the simple random walk $X$ will never hit $o$ has probability zero. However, we can take the Doob $a(\cdot, o)$-transform and use this to define the CRW. We make this precise below.
	
	We apply some of the results derived earlier to answer some basic questions about CRW. For example: is there a connection between the harmonic measure from infinity and the hitting probability of points (and sets)? What is the probability that the CRW will ever hit a given vertex? Do the traces of two independent random walks intersection infinitely often? Does the random walk satisfy a Harnack inequality? Does is satisfy the Liouville property? The answers will turn out to be yes for all of the above, and the majority of this section is devoted to proving such statements.

	In a series of papers studying the conditioned random walk (\cite{PopovEtal2015, PopovEtal2019, PopovEtal2020}, see also the lecture notes by Popov \cite{PopovRW}), the following remarkable observation about the CRW $(\hat X_t, t\ge 0)$ on $\mathbb{Z}^2$ was made. Let
	\[
		\wh q(y) = \P( \wh X_t = y \text{ for some $t \ge 0$}) = \P( \wh T_y < \infty),
	\]
	then $\lim_{y \to \infty} \wh q(y) = 1/2$, even though asymptotically the conditioned walk $\hat X$ looks very similar to the unconditioned walk.

	One may wonder if such a fact holds in the generality of stationary random graphs for which the potential kernel is well defined. This question was in fact an inspiration for the rest of the paper. Unfortunately, we are not able to answer this question in generality, but believe it should not be true in general. In fact, on most natural models of random planar maps, we expect
	\begin{equation}\label{E:conj}
		0 < \liminf_{y \to \infty} \wh q(y) < 1/2 < \limsup_{y \to \infty} \wh q(y) <1,
	\end{equation}
	with every possible value in the interval between $ \liminf_{y \to \infty} \wh q(y)$ and $\limsup_{y \to \infty} \wh q(y)$ a possible subsequential limit. We will prove the upper-bound of \eqref{E:conj} and a form of the lower bound on CRT-mated maps in Theorem \ref{theorem: CRT hm bounds}. The fact that every possibly value between $\liminf_{y \to \infty} \wh{q}(y)$ and $\limsup_{y \to \infty} \wh q(y)$ will have a subsequential limit converging to it, holds in general and will be proved in Proposition \ref{prop: q is interval}.

	
	\subsection{Definition and first estimates}
	Instead of the graph distance or effective resistance distance, we will work with the quasi distance $a(x, y)$. Recall the definition $\Lambda_a(y, R) := \{x \in v(G): a(x, y) \leq R\}$ and $\Lambda_a(R) = \Lambda_a(o, R)$. We will fix $y = o$, but we note that in the random setting, it is of no importance that we perform our actions on the root (in that setting, everything here is conditional on some realization $(G, o)$).
	
	We can thus define the \textbf{conditioned random walk} (CRW), denoted by $\wh{X}$, as the so called Doob $h$-transform of the simple random walk, with $h(x) = a(x, o)$. To avoid unnecessarily loaded notations, we will in fact denote $a(x) = a(x,o)$ in the rest of this section.

To be precise, let $p(x, y)$ denote the transition kernel of the simple random walk on $G$. Then the transition kernel of the CRW is defined as
	\[
		\wh{p}(x, y) = \begin{cases}
			\frac{a(y)}{a(x)} p(x, y), & x \not = 0\\
			0, &\text{else}
		\end{cases}.
	\]
	It is a standard exercise to show that $\wh{p}$ indeed defines a transition kernel. To include the root $o$ as a possible starting point for the CRW, we will let $\wh{X}_1$ have the law $\PP_o(\wh{X_1} = x) = a(x)$, and then take the law of the CRW afterwards. In this case, we can think of the CRW as the walk conditioned to never return to $o$.

	We now collect some preliminary results, starting with transience, and showing that the walk conditioned to hit a far away region before returning to the origin converges to the conditioned walk, as expected.

	We will write $\wh{T}_A$ for the first hitting time of a set $A \subset v(G)$ by the conditioned random walk, and $\wh{T}_x$ when $A = \{x\}$. We will also denote $\wh{T}_R = \wh{T}_{v(G) \setminus \Lambda_a(R)}$. We recall that $a(\cdot, \cdot)$ satisfies a triangle inequality (see Proposition \ref{prop: green function and PK}) and hence we have the growth condition
	\begin{equation} \label{eq: growth of a bound}
		a(x) \leq a(y) + 1
	\end{equation}
	for two neighboring sites $x, y$ since $a(x, y) \leq 1$ in this case.

	\begin{proposition} \label{prop: properties of the CRW}
		Let $x \in v(G)$ and $\wh{X}$ the CRW avoiding the root $o$. Then
		\begin{enumerate}[(i)]
			\item The walk $\wh{X}$ is transient.
			\item The process $n \mapsto 1/a(\wh{X}_{n \wedge \wh{T}_N})$ is a martingale, where $ N = \{y: y \sim o\}$
		\end{enumerate}
	\end{proposition}
	
	\begin{proof}
		The proof of (ii) is straightforward since $1/ a( \wh{X}_{n \wedge \wh{T}_N})$ is the Radon--Nikdoym derivative of the usual simple random walk with respect to the conditioned walk. (i) then follows from the fact that $a(y, o) \to \infty$ along at least a sequence of vertices. Indeed, fix $2 < r <R$ large and $y \in v(G) \setminus \Lambda_a(o, r)$. By optional stopping (since $1/ a(y)$ is bounded)
		\[	
			\frac{1}{a(y)} = \E_y\left[\frac{1}{a(\wh{X}_{\wh{T}_R \wedge \wh{T}_r})}\right]	\geq \frac{1}{r + 1}\PP_y(\wh{T}_r < \wh{T}_R) + \frac{1}{R + 1}\PP_y(\wh{T}_R \leq \wh{T}_r).
		\]
		Rearranging gives
		\begin{equation}\label{eq:hitsmall}
			\PP_y(\wh{T}_r < \wh{T}_R) \leq \frac{\frac{1}{a(y)} - \frac{1}{R + 1}}{\frac{1}{r + 1} - \frac{1}{R + 1}}.
		\end{equation}
		Taking $R \to \infty$, we see that $\PP_y(\wh{T}_r < \infty) \le (r+1) / (1+ a(y)) < 1$, showing that the chain is transient.
	\end{proof}

	We now check (as claimed earlier) that the conditioned walk $\hat X$ can be viewed as a limit of simple random walk conditioned on an appropriate event of positive (but vanishingly small) probability.
	\def\ph{\varphi}

	\begin{lemma} \label{lemma: cRW does what it should}
Uniformly over all paths $\ph = (\ph_0, \ldots, \ph_m) \subset \Lambda_a(R)$, as $R \to \infty$,
		\[
			\PP_{x}((X_0, \ldots, X_m) = (\ph_0, \ldots, \ph_m) \mid T_{R} < T_o) = \PP_x( (\wh X_0, \ldots, \wh X_m) = (\ph_0, \ldots, \ph_m) )(1 + o(1)).
		\]
	\end{lemma}

	\begin{proof}
		The proof is similar to \cite[Lemma 4.4]{PopovRW}. Assume here that $x \not = o$ for simplicity. The proof for $x = o$ follows after splitting into first taking one step and, comparing this, and then do the remainder. Let us first assume that the end point $\ph_m$ of $\ph$ lies in $\partial \Lambda_a(R)$. Then
		\[
			\PP_x((\cRW_0, \ldots, \cRW_m) = \varphi) = \frac{a(\varphi_m)}{a(\varphi_0)} \PP_{x}((X_0, \ldots, X_m) = \varphi).
		\]
		Since $\ph_m \in \partial \Lambda_a(R)$, we know that $a(\varphi_m) \in (R, R + 1]$ due to \eqref{eq: growth of a bound}. By optional stopping, we see
		\[
			a(x) = \PP_x(T_R < T_o) \E_x[a(X_{T_R}) \mid T_R < T_o], 		
		\]
		and also $a(X_{T_R}) \in (R, R+1]$. We thus find that
		\begin{equation} \label{eq: expression for escape prob in terms of a(x)}
			\PP_x(T_R < T_o) = \frac{a(x)}{R}(1 + o_R(1)).
		\end{equation}
		Combining this, we get
		\begin{align*}
			\PP_x((X_0, \ldots, X_m) = \varphi \mid T_R < T_o) &= 	\frac{\PP_x((X_0, \ldots, X_m) = \varphi)}{a(x,o)} R(1 + o(1)).  			
		\end{align*}
		Now let $\ph$ be an arbitrary path in $\Lambda_a(R)$ starting from $x$, then by the Markov property,
		\begin{align*}
			\P_x( (X_0, \ldots, X_m) = \ph \mid  T_R < T_o) & = \P_x ((X_0,\ldots,  X_m) = \ph) \P_{\ph_m} (T_R < T_o)/ \P_x( T_R< T_o) \\
			& =  \P_x ((X_0,\ldots, X_m) = \ph)  a(\ph_m)/ a(x) (1+ o(1))\\
			& = \P_x ( (\wh X_0, \ldots, \wh X_m) = \ph) (1+ o(1)),
		\end{align*}
		as desired.
	\end{proof}

	\subsubsection*{The Green Function}
	We can find an explicit expression for the Green function associated to $\wh{X}$. To that end, we define for $x, y \in v(G) \setminus \{o\}$
	\[
		\wh{\Gr}(x, y) = \E_x \left[\sum_{n=0}^\infty \id_{\wh{X}_n = y}\right],
	\]
	which is well defined as $\wh{X}$ is transient (also, the well-definition would follow from the proof below, which provides yet another way to see that the CRW is transient).
	
	\begin{proposition} \label{prop: CGreen in PK}
		Let $x, y \in v(G) \setminus \{o\}$. Then
		\[
			\frac{\wh{\Gr}(x, y)}{\deg(y)} = \frac{a(y, o)}{a(x, o)} \frac{\Gr_o(x, y)}{\deg(y)} = \frac{a(y, o)}{a(x, o)} \big(a(x, o) - a(x, y) + a(o, y)\big).
		\]
	\end{proposition}

	\begin{proof}
		Fix $x, y \in v(G) \setminus \{o\}$. For definiteness we take the exhaustion $\Lambda_a(R)$ of $G$ here, but we need not to, any exhaustion would work. Define for $R \geq 1$ the truncated Green function:
		\[
			\wh{\Gr}_R(x, y) := \E_x \left[ \sum_{n = 0}^{\wh{T}_R - 1} \id_{\cRW_n = y} \right].
		\]
		We denote $A_R = (\Lambda_a(R))^c \cup \{o\}$ and will show that
		\begin{equation} \label{eq: hatGR and G_AR}
			\wh{\Gr}_R(x, y) = \frac{a(y)}{a(x)} \Gr_{A_R}(x, y),
		\end{equation}
		from which the result follows when $R$ goes to infinity. Fix $R \geq 1$ and notice the following standard equality, which follows from the Markov property of the CRW:
		\begin{equation*}
			\wh{\Gr}_R(x, y) = \frac{\PP_x(\wh{T}_y < \wh{T}_R)}{\PP_y(\wh{T}_y^+ < \wh{T}_R)}
		\end{equation*}
		We first deal with the numerator. From Proposition the definition of the CRW we get
		\begin{equation} \label{eq: CRW hits y before R direct formula}
			\PP_x(\wh{T}_y < \wh{T}_R) = \frac{a(y)}{a(x)} \PP_x(T_y < T_R \wedge T_o).
		\end{equation}
		Indeed, just sum over all paths $\varphi$ taking $x$ to $y$, and which stay inside $\Lambda_a(R) \setminus \{o\}$. Then each path has as endpoint $y$, and the probability that the simple random walk will take any of these paths is nothing but $\PP_x(T_y < T_R \wedge T_o)$.
		
		We can deal with the denominator in a similar fashion, only this time we note that the beginning and end point are the same. Hence, the $a(y)$-terms cancel and we get
		\[
			\wh{\Gr}_R(x, y) = \frac{a(y)}{a(x)} \frac{\PP_x(T_y < T_{o} \wedge T_R)}{\PP_y(T_y^+ < T_o \wedge T_R)} = \frac{a(y)}{a(x)} \Gr_{A_R}(x, y).
		\]
		This shows the first equality appearing in Proposition \ref{prop: CGreen in PK} upon taking $R \to \infty$. The second statement follows from Proposition \ref{prop: green function and PK}.
	\end{proof}

	\begin{wrong-old}
	\begin{proof}
		Fix $x, y \in v(G) \setminus \{o\}$. Define for $R \geq 1$ the truncated Green function:
		\[
			\wh{\Gr}_R(x, y) := \E_x \left[ \sum_{n = 0}^{\wh{T}_R - 1} \id_{\cRW_n = y} \right].
		\]
		Recall from Proposition \ref{prop: green function and PK} that
		\[
			\frac{\Gr_o(x, y)}{\deg(y)} = a(x, o) - a(x, y) + a(y, o).
		\]
		Set $A_R = (v(G) \setminus \Lambda_R) \cup \{o\}$. It is thus enough to prove
		\begin{equation} \label{eq: relation green and cgreen}
			\wh{\Gr}_R(x, y) = \frac{a(y, o)}{a(x, o)} \Gr_{A_R}(x, y)(1 + o(1)),
		\end{equation}
		and this is what we will do.
		
		For a random walk $Y$ on $v(G)$ and a point $y$ we will use the notation
		\[
			L^Y_y(k) = \sum_{n=0}^{k - 1} \id_{Y_n = y}
		\]
		for the local time spent at $y$ before time $k$. Let $x, y \in \Lambda_R \setminus \{o\}$. Then for each $n \geq 1$ we have
		\begin{align*}
			&\PP_x(L^{X}_y(T_{A_R}) = n, T_R < T_o) \\
			&\qquad = \PP_x(L_y^X(T_{A_R}) \geq n) \PP_y(T_R < T_o, T_{A_R} < T_{y}^+) \\
			&\qquad = \PP_x(L_y^X(T_{A_R}) \geq n) \PP_y(T_R < T_o \mid T_{A_R} < T_y^+) \PP_y(T_{A_R} < T_{y}^+) \\
			&\qquad = \PP_x(L_y^X(T_{A_R}) \geq n) \PP_y(T_R < T_o) \PP_y(T_{A_R} < T_{y}^+) \\
			&\qquad = \PP_x(L_y^X(T_{A_R}) = n) \PP_y(T_R < T_o),
		\end{align*}
		which follows using standard arguments. Now, we use \eqref{eq: expression for escape prob in terms of a(x)} to obtain
		\begin{equation} \label{eq: local time probability 1}
			\PP_x(L^{X}_y(T_{A_R}) = n, T_R < T_o) = \PP_x(L_y^X(T_{A_R}) = n) \frac{a(y)}{R}(1 + o(1)).
		\end{equation}
		Treating the left-hand side of this equation and using again \eqref{eq: expression for escape prob in terms of a(x)} and then Lemma \ref{lemma: cRW does what it should} we find
		\begin{align*}
			\PP_x(L^{X}_y(T_{A_R}) = n, T_R < T_o) &= \PP_x(L^{X}_y(T_{A_R}) = n \mid T_R < T_o) \frac{a(x)}{R}(1 + o(1)) \\
			&= \PP_x(L^{\wh{X}}_y(T_{A_R}) = n)\frac{a(x)}{R}(1 + o(1)).
		\end{align*}
		Combining this with \eqref{eq: local time probability 1} we thus get
		\[
			a(x)\PP_x(L_y^X(T_{A_R}) = n) = a(y)\PP_x(L_y^{\wh{X}}(T_{A_R}) = n)(1 + o(1)).
		\]
		Summing over $n$, we obtain the desired identity.
	\end{proof}
	\end{wrong-old}

	\subsection{Intersection and hitting probabilities}
	Suppose $\wh{X}$ and $\wh{Y}$ are two independent CRW's. We will begin by describing hitting probabilities of points and sets and use this to prove that the traces of $\wh{X}$ and $\wh{Y}$ intersect infinitely often a.s.
	
	We begin giving a description of the hitting probability of a vertex $y$ by the CRW started from $x$. Although it is a rather straightforward consequence of the expression for the Green function of the CRW, it is still remarkably clean.
	\begin{lemma} \label{lem: expression for CRW hitting probability}
		Let $x, y \in v(G) \setminus \{o\}$, then
		\[
			\PP_{x}(\wh{T}_y < \infty) = \frac{\hm_{y, o}(y)\PP_y(T_x < T_o)}{\hm_{x, o}(x)}.
		\]
	\end{lemma}
	\begin{proof}
		Note that for $x \not = y$ we have
		\[
			\wh{\Gr}(x, y) = \PP_x(\wh{T}_y < \infty)\wh{\Gr}(y, y),
		\]
		so that by Proposition \ref{prop: CGreen in PK} and Corollary \ref{cor: the harmonic measure for two points is well-defined} we find
		\begin{align}
			\PP_x(\wh{T}_y < \infty) &= \frac{\wh{\Gr}(x, y)}{\wh{\Gr}(y, y)} = \frac{a(y,o)}{a(x, o)}\frac{\frac{\Gr_o(x, y)}{\deg(y)}}{a(y, o) + a(o, y)} \nonumber \\
			&= \frac{\hm_{y, o}(y) \Reff(o \leftrightarrow y) \frac{\Gr_o(x, y)}{\deg(y)}}{\hm_{x, o}(x) \Reff(o \leftrightarrow x)\Reff(o \leftrightarrow y)} \nonumber \\
			&= \frac{\hm_{y, o}(y) \PP_y(T_x < T_o)}{\hm_{x, o}(x)}, \nonumber
		\end{align}
		as desired.
	\end{proof}

	Since the potential kernel is assumed to be well defined, we also have that $\PP_y(T_x < T_o) \to \hm_{o, x}(x)$ as $y \to \infty$ due to Corollary \ref{cor: the harmonic measure for two points is well-defined}, and hence we deduce immediately the next result.
	
	\begin{corollary} \label{cor: hm and q are equivalent}
		We have that
		\[
			\liminf_{y \to \infty} \wh q(y) = \liminf_{y \to \infty} \hm_{o, y}(y)
		\]
		and the same with `limsup' instead of `liminf'.
	\end{corollary}

	In particular, it is true that on \emph{transitive} graphs that are recurrent and for which the potential kernel is well defined, by symmetry one always has $\wh q(y) \to \frac{1}{2}$. This gives another proof to a result of \cite{PopovRW} on the square lattice.
	
	We can now prove that the subsequential limits of the hitting probabilities $\wh{q}(y)$ define an interval, as promised before. Note that this proposition is fairly general: it does not require the underlying graphs to be unimodular, only for the graph to satisfy the standing assumption (recurrence, existence of potential kernel and convergence to infinity of the potential kernel).
	
	\begin{proposition} \label{prop: q is interval}
		For each $q \in [\liminf_{y \to \infty} \wh q(y), \limsup_{y \to \infty} \wh q(y)]$, there exists a sequence of vertices $(y_n)_{n \geq 1}$ going to infinity such that
		\[
			\lim_{n\to \infty} \wh{q} (y_n) = q.
		\]
	\end{proposition}

	\begin{proof}
		Assume that there exist $q_1 < q_2$ such that there are sequences $(y^1_n)_{n \geq 1}$ and $(y^2_n)_{n \geq 1}$ going to infinity for which $\lim_{n\to \infty} \wh{q}(y^i_n) = q_i$, but there does not exists a sequence $y_n$ going to infinity for which $q_1 < \lim_{n \to \infty} \wh{q}(y_n) < q_2$. We will derive a contradiction. We do so via the following claim.
		\begin{claim}
			For each $\epsilon > 0$, there exists an $N = N(G, o, \epsilon)$ such that for each neighboring vertices $x, y \notin B(o, N)$, we have
			\[
				|\wh{q}(x) - \wh{q}(y)| < \epsilon.
			\]
		\end{claim}
		To see this claim is true, we use Lemma \ref{lem: expression for CRW hitting probability} and Corollary \ref{cor: harmonic measure are local, not global} to get the existence of $N_1$ such that
		\begin{equation} \label{eq:diff q and hm}
			|\wh{q}(z) - \hm_{o, z}(z)| < \frac{\epsilon}{4}
		\end{equation}
		for all $z \notin B(o, N_1)$. Next, pick $N_2$ such that all $z \notin B(o, N_2)$ have $\Reff(o \leftrightarrow z) > \frac{4}{\epsilon}$. Let $x, y \notin B(o, N_1 \vee N_2)$ be neighbors. Due to \eqref{eq: growth of a bound} we have $a(x) - a(y) \leq 1$ and by the triangle inequality for effective resistance also $\Reff(o \leftrightarrow y) \leq \Reff(o \leftrightarrow x) + 1$. Hence, using the expression $a(x) = \hm_{x, o}(x) \Reff(o \leftrightarrow x)$ of Corollary \ref{cor: the harmonic measure for two points is well-defined}, we deduce that
		\[
			\hm_{x, o}(x) \frac{\Reff(o \leftrightarrow y) - 1}{\Reff(o \leftrightarrow y)} - \hm_{y, o}(y) \leq \frac{a(x) - a(y)}{\Reff(y \leftrightarrow o)} \leq \frac{1}{\Reff(o \leftrightarrow y)},
		\]
		which implies by choice of $N_2$ that in fact
		\begin{equation} \label{eq: hm neighbor diff}
			\hm_{x, o}(x) - \hm_{o, y}(y) < \frac{\epsilon}{2}.
		\end{equation}
		Thus, taking together equations \eqref{eq:diff q and hm} and \eqref{eq: hm neighbor diff} we obtain
		\[
			\wh{q}(x) - \wh{q}(y) \leq \epsilon.
		\]
		Since $x, y$ are arbitrary neighbors, this implies the claim when taking $N = N_1 \vee N_2$.
		
		By Corollary \ref{corollary: not one-ended inplies that potential kernels is not well-defined}, we know that the graph $G$ is one-ended as the potential kernel is assumed to be well defined. Take $\epsilon > 0$ so small that $q_2 > q_1 + 3 \epsilon$. By assumption on $q_1, q_2$, we thus have that for each $n$ large enough, there exist two neighboring vertices $x, y \notin B(o, n)$ satisfying
		\[
			\wh{q}(y) > q_2 - \epsilon > q_1 + 2\epsilon > \wh{q}(x) + \epsilon,
		\]
		so that $\wh{q}(y) > \wh{q}(x) + \epsilon$, a contradiction.
	\end{proof}

	\subsubsection{Harnack inequality for conditioned walk}

	Notice that the conditioned random walk viewed as a Doob $h$-transform may be viewed as a random walk on the original graph $G$ but with new conductances by $$\wh{c}(x, y) = a(x)a(y)$$ 
for each edge $\{x, y\} \in e(G)$. Indeed the symmetry of this function is obvious, as is non-negativity, and since $a$ is harmonic for the original graph Laplacian $\Delta$,
	\[
		\pi(x) := \sum_{y \sim x} \wh{c}(x, y) = \sum_{y \sim x} a(y) a(x) = a(x)^2,
	\]
we get that the random walk associated with these conductances coincides indeed with our Doob $h$-transform description of the conditioned walk. 
	
	We can thus consider the network $(G, \wh{c} \:)$, which is transient by Proposition \ref{prop: properties of the CRW}. It will be useful to consider the graph Laplacian $\wh{\Delta}$, associated with these conductances, defined by setting 
	\[
		(\wh{\Delta} h)(x) = \sum_{y \sim x} \wh{c}(x, y)(h(y) - h(x)).
	\]
for a function $h$ defined on the vertices of $G$, although $h$ does not need to be defined at $o$. 
	We will say that a function $h: v(G)\setminus \{o\} \to \RR$ is harmonic (w.r.t. the network $(G, \wh{c} \:)$) whenever $\wh{\Delta } h \equiv 0$. This is of course equivalent to
	\[
		h(x) = \E_x[h(\wh{X}_1)]
	\]
	for each $x \in v(G) \setminus \{o\}$.
	
	It might be of little surprise that the anchored Harnack inequality (Theorem \ref{theorem: harnack with a pole}) implies (in fact, it is equivalent but this will not be needed) to an elliptic Harnack inequality on the graph $G$ with conductance function $\wh{c}$, at least when viewed from the root (i.e., for exhaustion sequences centered on the root $o$).
	
	\begin{proposition} \label{prop: Harnack CRW}
		There exists a $C > 1$ 
such that the following holds. Suppose the graph $G$ satisfies the standing assumptions. Let $\hat h: v(G) \setminus \{o\} \to \RR_+$ be harmonic with respect to $(G, \wh{c} \:)$. Then for each $R \geq 1$,
		\[
			\max_{x \in \partial \Lambda_a(R)} \hat h(x) \leq C \min_{x \in \partial \Lambda_a(R)} \hat h(x).
		\]
		Equivalently, the max and the min could (by the maximum principle) be taken over $\Lambda_a(R)$ instead of $\partial \Lambda_a(R)$.
	\end{proposition}

	\begin{proof}
		Since the graph follows the standing assumptions it satisfies the anchored Harnack inequality of Theorem \ref{theorem: harnack with a pole}. Furthermore, $\hat h(x)$ is $\wh{\Delta}$-harmonic if and only if
		 $$
		 	h(x) = \begin{cases}
		 	a(x) \hat h(x) & \text{ if } x \neq o\\
		 	0 & \text{ if } x = o
		 	\end{cases}
  		$$
		is harmonic for $\Delta$ away from $o$. Since on $|a(x) - R | \le 1$ for $x \in \partial \Lambda_a(R)$, we obtain the result immediately.
	\end{proof}

	As a corollary we obtain the Liouville property for $\hat X$: $(G, \hat c)$ does not carry any non-constant, bounded harmonic functions. This implies in turn that the invariant $\sigma$-algebra $\c{I}$ of the CRW is trivial.
	\begin{corollary} \label{cor: CRWliouville}
		The network $(G, \wh{c} \:)$ satisfies the Liouville property, that is: any function $h:v(G) \setminus \{o\} \to \RR$ that is harmonic and bounded must be constant.
	\end{corollary}
	\begin{proof}
		Let $h$ be a bounded, harmonic function with respect to $(G, \wh{c} \:)$. Define the function
		\[
			\hat {h} = h - \inf_{x \in v(G)} h(x),
		\]
		which is non-negative and harmonic. Moreover, for each $\epsilon > 0$, there exists an $x_{\epsilon}$ such that $\hat{h}(x_{\epsilon}) \leq \epsilon$. Take $R_{\epsilon}$ so large that $x_{\epsilon} \in \Lambda_a(R_{\epsilon})$. By the Harnack inequality (Proposition \ref{prop: Harnack CRW}) we deduce that for all $x \in \Lambda_a(R_{\epsilon})$,
		\[
			0 \leq \hat{h}(x) \leq C\hat{h}(x_{\epsilon}) \leq C\epsilon.
		\]
		Since $\epsilon$ is arbitrary, and $C$ does not depend on $R_{\epsilon}$ nor $\epsilon$, this shows the desired result.
	\end{proof}

	\subsubsection{Recurrence of sets}
	We will say that a set $A \subset v(G)$ is recurrent for the chain $\wh{X}$ whenever there exist $x \in v(G)$ such that
	\[
		\PP_x(\wh{X}_n \in A \io) = 1,
	\]
	where $\io$ is short-hand for `infinitely often'. Since $(G, \wh{c} \:)$ satisfies the Liouville property, such probabilities are $0$ or $1$, hence the definition of $A$ being recurrent is independent of the choice of $x$. If a set is not recurrent, it is called transient. Since $\wh{X}$ is transient, any finite set $A$ is transient too. Notice, by the way, that the definition above is equivalent to saying that $A$ is recurrent whenever $\PP_x(\wh{T}_A < \infty) = 1$ for all $x \in v(G)$.
	
	We capture next some results, relating recurrence and transience of sets to the harmonic measure from infinity. Recall Definition \ref{def: delta-good} of $\delta$-good points: $x$ is $\delta$-good whenever $\hm_{x, o}(x) \geq \delta$.
	
	\begin{lemma} \label{lemma: A inf delta-good then recurrent}
		If $A$ has infinitely many $\delta$-good points for some $\delta > 0$, then $A$ is recurrent for $\wh{X}$.
	\end{lemma}
	\begin{proof}
		This follows from a Borel-Cantelli argument. Indeed, fix $x \in v(G)$. Let $\delta$ be as in the assumption. Take $(g_{i})_{i = 1}^\infty$ a sequence of $\delta$-good points in $A$, with $a(g_i) >i $ (which we can clearly find as $\Lambda_a(i)$ is finite whereas $A$ has infinitely many good points).

		We will define two sequences $(R_i)_{i \geq 1}$ and $(M_i)_{i \geq 1}$. Set $M_0 = 0$ and $R_0 = 0$. 
		Suppose we have defined $R_i, M_{i - 1}$ already. Set $a_i = a(g_{R_i})$ and $\Lambda_i = \Lambda_a (a_i)$, and note that by definition $g_{R_i} \in \Lambda_i$. Take $M_i$ so large (and greater than $R_i$) that
		\begin{equation} \label{eq:condscale1}
			\PP_z(\wh{X} \text{ ever hits } \Lambda_i) \leq \frac{\delta }{4}, \text{ uniformly over } z \in \Lambda_a (M_i)^c
		\end{equation}
This is possible since $\Lambda_i $ is finite and $\hat X$ is transient by Proposition \ref{prop: properties of the CRW} and more precisely the hitting probabilities of a finite set converge to zero (see \eqref{eq:hitsmall}). Next, let $R_{i + 1}$  be so large (and greater than $M_i$) that
		\begin{equation}\label{eq:condscale2}
			\frac{\PP_y(T_x < T_o)}{  \hm_{o, x}(x)} \geq 1/2,
		\end{equation}
		for $y  = g_{R_{i+1}}$ and all $x \in \Lambda_a(M_i)$. This is possible because $\Lambda_a(M_i)$ is finite and hitting probabilities converge to harmonic measure from infinity, by Corollary \ref{cor: the harmonic measure for two points is well-defined}. We can also require without loss of generality that $g_{R_{i+1}} \in \Lambda_a(M_i)^c$.

Suppose that $x \in \Lambda_a (M_{i-1})$ is arbitrary. We first claim that from $x$ it is reasonably likely that the conditioned walk $\wh{X}$ will hit $ y = g_{R_{i}}$. Indeed, note that by Lemma \ref{lem: expression for CRW hitting probability}, and since $y$ is $\delta$-good and \eqref{eq:condscale2} holds,
\begin{align*}
\P_x( \wh{T}_{y} < \infty ) &= \hm_{y,o} (y) \frac{\P_y (T_x< T_o)}{\hm_{x,o} (x)}
 \ge \delta /2
\end{align*}
On the other hand, conditionally on hitting $y = g_{R_i}$, the conditioned walk $\wh{X}$ is very likely to do so before exiting $\Lambda_{a} (M_i)$ (let us call $\tau_i$ this time). Indeed, by the strong Markov property at $\tau_i$ and \eqref{eq:condscale1},
$$
\P_x ( \wh{T}_y > \tau_i, \wh{T}_y < \infty) \le \sup_{z \in \Lambda_a(M_i)^c} \P_z ( \wh{X} \text{ ever hits } \Lambda_i) \le \delta /4.
$$
Therefore,
$$
\P_x (\wh{T}_y < \tau_i)  \ge \frac{\delta}{2} - \frac{\delta}{4} =  \frac{\delta }{4}.
$$
Let $E_i$ be the above event, i.e., $E_i = \{ \wh{T}_{g_{R_i}} < \tau_i\}$. Since $x \in \Lambda_a(M_{i-1})$ in the above lower bound is arbitrary, it follows from the strong Markov property at time $\tau_{i-1}$ that $\P(E_i | \cF_{\tau_{i-1}} ) \ge \delta /4$, where $(\cF_n)_{n \ge 0}$ is the filtration of the conditional walk. By Borel--Cantelli we conclude immediately that $E_i$ occurs infinitely often a.s. (for the conditioned walk), which concludes the proof.
	\end{proof}

	\begin{wrong-old}
	The next lemma connects `transience' of infinite sets to harmonic measure from infinity. The lemma hereafter shows the existence of such sets for graphs that are not `uniformly $\delta$-good'. Combining this with the previous lemma, we deduce that there are infinite transient sets for the CRW if and only if the rooted graph $(G, o)$ is not uniformly $\delta$-good for each $\delta > 0$.
	\begin{lemma}
		Suppose $A \subset v(G)$ is an infinite and satisfies $\PP_x(\wh{T}_A < \infty) < 1$ for some $x \sim o$. Write $A_N = A \cap \Beff(N)$. Then
		\[
		\lim_{N \to \infty} \hm_{A_N}(o) > 0.
		\]
	\end{lemma}
	\begin{proof}
		Existence of the limit is not too hard: the map $N \mapsto \hm_{A_N}(o)$ is bounded from below and non-increasing. Positivity is the only real question.
		
		Fix $x \sim o$ and $\epsilon > 0$. Take $N_0 = N_0(x, \epsilon)$ so large that for all $y \not \in \Beff(N_0)$ we have that
		\[
		\frac{\PP_y(T_x < T_o)}{\hm_{x, o}(x)} \in \left( 1 - \epsilon, 1 + \epsilon \right).
		\]
		Next, assume without loss of generality that $A \cap \Beff(N_0) = \emptyset$. Glue together all vertices in $A_N$. The potential kernel is still well defined by the Gluing Theorem \ref{theorem: gluing expression} and by Lemma \ref{lem: expression for CRW hitting probability} we have
		\[
		\PP_x(\wh{T}_{A_N} < \infty) = \frac{(1 - \hm_{A_N \cup \{o\}}(o)) \PP_{A_N}(T_x < T_o)}{\hm_{x, o}(x)}.
		\]
		By choice of $N_0$, and because the left-hand side stays positive uniformly in $N$, by transience of $A$, we get the result.
	\end{proof}
	
	\begin{lemma}
		If for all $\delta > 0$, the set $\{x \in v(G): x \text{ is not } \delta\text{-good}\}$ is non-empty, then there exist an infinite set $A$ which satisfies $\PP_x(\wh{T}_A < \infty) < 1$ for some $x$ and is in particular transient.
	\end{lemma}
	\begin{proof}
		We will be slightly informal. Fix $x \sim o \in v(G)$ and take $N_0$ so large that $\PP_y(T_x < T_o)$ is close to $\hm_{x, o}(x)$ for all $y \in v(G) \setminus \Beff(N_0)$. By assumption, we can take a sequence of points $x_i$, such that $x_i$ is not $3^{-i}$-good and such that $x_i \not \in \Beff(N_0)$. Set $A = \{x_i: i \geq 1\}$. By a simple union bound, we get
		\[
		\PP_x(\wh{T}_A < \infty) = \PP_x\left( \bigcup_{i=1}^\infty \{T_{x_i} < \infty\} \right) \leq \sum_{i = 1}^\infty \PP_x(\wh{T}_{x_i} < \infty) < 1.
		\]
	\end{proof}
	
	With the same proof, we get that if the sets $\{x: x \text{ is not } \delta\text{-good}\}$ are non-empty for all $\delta > 0$, then for each $\epsilon > 0$, there exist infinite sets $A \subset v(G)$ with $o \in A$ for which $\lim_{N \to \infty} \hm_{A_N}(o) \geq 1 - \epsilon$. We would like to warn the reader, on the other hand, that
	\[
		\liminf_{w \to \infty} \PP_w(T_o < T_A)	 = 0
	\]
	for instance by taking a sequence $w$ that always stays inside (or near) $A$.
	\end{wrong-old}
	\subsubsection{Infinite intersection of two conditioned walks}
	We finish this section by showing that two independent conditioned random walks have traces that intersect infinitely often (for simplicity here the CRW's are conditioned to not hit the same root $o$). We manage to prove this under two (different) additional assumptions. We start by adding the assumption that $(G, o)$ is random and reversible.
	
	\begin{proposition} \label{prop: REVERSIBLE CRW traces intersect io}
		Suppose that $(G, o)$ is a reversible random graph, such that a.s. it is recurrent and a.s. the potential kernel is well defined. Let $\wh{X}$, $\wh{Y}$ be two independent CRW's  started from $x, y \in v(G)$ respectively, avoiding $o$. Then a.s.
		\[
			\PP(|\{\wh{X}_n: n \in \NN\} \cap \{\wh{Y}_n: y \in \NN\} | = \infty) =1.
		\]
	\end{proposition}
	\begin{proof}
		Suppose that $(G, o)$ has infinitely many $\frac{1}{3}$-good vertices, and call the set of such vertices $A := A(G, o)$. Since there are various sources of randomness here, it is useful to recall that $\P$
the underlying probability measure $\PP$ is always conditional on the rooted graph $(G, o)$.
Then by Lemma \ref{lemma: A inf delta-good then recurrent}, we know that
		\[
			\PP(|\{\wh{X}_n : n \in \NN\} \cap A| = \infty) = 1.
		\]
		Now, consider the set $B = \{\wh{X}_n: n \in \NN\} \cap A$. By definition, every point in $B$ is $\frac{1}{3}$-good. Since $\wh{Y}$ is independent of $\wh{X}$ (when conditioned on $(G, o)$), we can use Lemma \ref{lemma: A inf delta-good then recurrent} again to see that on an event of $\PP$-probability $1$,
		\[
			\PP(|\{\wh{Y}_n: n \in \NN\} \cap B| = \infty \mid \wh{X}) = 1
		\]
		Taking expectation w.r.t. $\wh{X}$ we deduce that the traces of $\wh{X}$ and $\wh{Y}$ intersect infinitely often  $\PP$-almost surely, conditioned on $(G, o)$ having infinitely many $\frac{1}{3}$-good vertices. However, Lemma \ref{lemma: reversible graphs are delta-good} implies that, under our assumptions on $(G, o)$, this happens with $\bf{P}$-probability one, showing the desired result.
	\end{proof}

	A consequence of the infinite intersection property is that the (random) network $(G, \wh{c} \:)$ is a.s. Liouville. Therefore we get a new proof of the already obtained (in Corollary \ref{cor: CRWliouville}) Liouville property for the conditioned walk,
but this time without using the Harnack inequality. On the other hand, \cite{BenCurienGeorga2012} proved that for planar graphs, the Liouville property is in fact equivalent to the infinite intersection property and this results extends without any additional arguments to the case of planar networks.

By  Proposition \ref{prop: Harnack CRW} and Corollary  \ref{cor: CRWliouville} we thus also obtain as a corollary of  \cite{BenCurienGeorga2012} the infinite intersection property for planar networks such that the potential kernel tends to infinity.
	
	\begin{proposition} \label{prop: PLANAR CRW traces intersect i.o.}
		Suppose $G$ is a (not necessarily reversible) planar graph satisfying the standing assumptions. Let $\wh{X}$ and $\wh{Y}$ be two independent CRW's avoiding $o$, started from $x, y \in v(G)$ respectively. Then
		\[
			\PP(|\{\wh{X}_n: n \in \NN\} \cap \{\wh{Y}_n: n \in \NN\}| = \infty) = 1.
		\]
	\end{proposition}

	\begin{remark}
		It will be useful for us to recall that the infinite intersection property implies that one walk intersects the loop-erasure of the other:
		\[
			\PP(|\{\LE(\wh{X})_n: n \in \NN\} \cap \{\wh{Y}_n: n \in \NN\}| = \infty) = 1,
		\]
		where $\LE(\wh{X})$ is the Loop Erasure of $\wh{X}$ and $\wh{X}, \wh{Y}$ are two CRW's that don't hit the root $o$, started from $x, y$ respectively. See \cite{Lyons_2003} for this result.
	\end{remark}

	\section{(a) implies (d): One-endedness of the uniform spanning tree} \label{sec: UST one ended}
	In this section we show that the uniform spanning tree is one ended, provided that the underlying graph satisfies the standing assumptions and is either planar or unimodular. In particular, since on unimodular graphs $a(x) \to \infty$ along any sequence $x \to \infty$ (see Proposition \ref{prop: PK level-sets are exhaustion}) we prove that (a) implies (d) in Theorem \ref{T:main_equiv}.
	
	\begin{theorem} \label{T: UST one-ended}
		Suppose that $(G,o)$ is a reversible, recurrent graph for which the potential kernel is a.s. well defined and such that $a(x) \to \infty$ along any sequence $x \to \infty$. Then the uniform spanning tree is one-ended almost surely.
	\end{theorem}

Before proving this theorem, we start with a few preparatory lemmas. We will write $\c{T}$ to denote the uniform spanning tree and begin by recalling the following ``path reversal'' for the simple random walk, a standard result. In what follows, fix the vertex $o \in v(G)$, but it plays no particular role (in the random setting) other than to simplify the notation.
	
	\begin{lemma}[Path reversal] \label{lemma: path reversal of RW}
		Let $o, u \in v(G)$. For any subset of paths $\c{P}$
		\[
			\PP_u((X_n: n \leq T_o) \in \c{P} \mid T_o < T_u^+) = \PP_o((X_n: n \leq T_u) \in \c{P}' \mid T_u < T_o^+),
		\]
		where a path $\varphi \in \c{P}'$ if and only if the reversal of the path is in $\c{P}$.
	\end{lemma}
	
See Exercise (2.1d) in \cite{LyonsPeresProbNetworks}. The next result says that the random walk started from $o$ and stopped when hitting $u$, conditioned to hit $u$ before returning to $o$ looks locally like a conditioned random walk when $u$ is far away. This is an extension of Lemma \ref{lemma: cRW does what it should} and its proof is similar.
	\begin{lemma} \label{L: RW cond to hit u is CRW}
		For each $M \in \NN$ and $\epsilon > 0$, there exists an $L$ such that for all $u \notin \Lambda_a(L)$ and uniformly over all paths $\varphi$ going from $o$ to $\partial \Lambda_a(M)$,
		\[
			\PP_o((X_0, \ldots, X_{T_M}) = \varphi \mid T_u < T_o^+) = \PP_o((\wh{X}_0, \ldots, \wh{X}_{T_M}) = \varphi) \pm \epsilon.
		\]
	\end{lemma}

	\begin{proof}
		Fix $M \in \NN$ and $\epsilon > 0$. Let $\varphi$ be some path $o$ to $\Lambda_a(M)$ not returning to $o$. Denote by $\varphi_{end} \in \partial \Lambda_a(M)$ the endpoint of such a path. By the Markov property for the simple walk
		\begin{align*}
			\PP_o((X_o, \ldots, X_{T_M}) = \varphi, T_u < T_o^+) &= \PP_o((X_{o}, \ldots, X_{T_M}) = \varphi)\PP_{\varphi_{end}}(T_u < T_o).
		\end{align*}
		Now, take $L$ so large that uniformly over $x \in \Lambda_a(M)$ with $x \neq o$,
		\[
			\frac{\PP_{x}(T_u < T_o)}{\deg(o)\PP_o(T_u < T_o^+)} = a(x) \pm \epsilon
		\]
		By definition, we have that
		\[
			\PP_o(X_1 = \varphi_1) = \frac{1}{\deg(o)},
		\]
		yet $\PP_o(\wh{X}_1 = \varphi_1) = a(\varphi_1)$. Therefore, and by definition of the $h$-transform,
		\[
			\PP_o((X_o, \ldots, X_{T_M}) = \varphi, T_u < T_o^+) = \PP_o((\wh{X}_o, \ldots, \wh{X}_{T_M}))\frac{1}{\deg(o)a(\varphi_{end})}\PP_{\varphi_{end}}(T_u < T_o),
		\]
		so that after dividing both sides through $\PP_o(T_u < T_o^+)$, we have
		\[
			\PP_o((X_o, \ldots, X_{T_M}) = \varphi \mid T_u < T_o^+) = \PP_o((\wh{X}_o, \ldots, \wh{X}_{T_M}) = \varphi) \pm \epsilon
		\]
		as desired.
	\end{proof}

	We will say that the graph satisfies an infinite intersection property for the CRW whenever
	\begin{equation} \label{E: CRW IP} \tag{cIP}
		\PP(|\{\wh{X}_n: n \in \NN\} \cap \{\LE(\wh{Y})_n: n \in \NN\}| = \infty) = 1.
	\end{equation}
	
	Next, under the assumption \eqref{E: CRW IP} it holds that as $u \to \infty$, a simple random walk started at $u$ is very unlikely to hit $\LE(\wh{Y})$ in $o$. This is the key property which gives one-endedness of the UST.
	
	\begin{lemma} \label{L: UST limsup X hits LE at o}
		Suppose \eqref{E: CRW IP} holds, then
		\[
			\limsup_{M \to \infty} \sup_{u \notin \Lambda_a(M)} \PP_u(X_{T_{\LE(\wh{Y})}} = o) = 0,
		\]
		where $X$ is a simple random walk started at $u$ and $\wh{Y}$ an independent conditioned walk started at $o$.
	\end{lemma}

	\begin{proof}
		Let $A$ be any simple path from $o$ to infinity in $G$. Then
		\begin{equation}\label{eq:excdom}
			\PP_u(X_{T_{A}} = o) \leq \PP_u(\{X_n: n \leq T_o\} \cap A = \{o\} \mid T_o < T_u^+).
		\end{equation}
To see this, it is useful to recall that the successive excursions (or loops) from $u$ to $u$ forms a sequence $(Z_1, Z_2, \ldots)$ of i.i.d. paths (with a.s. finite length). Let $N$ be the index of the first excursion which touches $o$. Then the law of $Z_N$, up to its hitting time of $o$, is that of $\P_u( \cdot | T_o < T_u^+)$. Furthermore, on the event $ \{ X_{T_{A}} = o\}$ it is necessarily the case that:
\begin{itemize}
  \item $Z_1, \ldots, Z_{N-1}$ avoid $A$. 
  \item $Z_N$ touches $A$ for the first time in $o$. 
\end{itemize}
When we ignore the first point above, we therefore obtain the upper-bound \eqref{eq:excdom}.

By Lemma \ref{lemma: path reversal of RW}, the right hand side is equal to 
		\[
			\PP_o(\{X_n: n \leq T_u\} \cap A = \{o\} \mid T_u < T_o^+). 
		\]
		Therefore, it suffices to show that this converges uniformly to zero over $u \in \Lambda_a(M)^c$, as $M \to \infty$.
		
		Let $\wh{X}$ be a CRW, started at $o$. Fix $\epsilon > 0$ and let $M$ be some integer to be fixed later. Take $L = L(M, \epsilon)$ large enough so that
		\begin{equation} \label{E: CRW X_n vs whX_n hits A}
			\PP_o(\{X_n: n \leq T_M\} \cap A = \{o\} \mid T_u < T_o^+) \leq \PP_o(\{\wh{X}_n: n \leq T_M\} \cap A = \{o\}) + \frac{\epsilon}{2},
		\end{equation}
		for all $u \notin \Lambda_a(L)$, which is possible by Lemma \ref{L: RW cond to hit u is CRW} (note that $L$ depends only on $\epsilon$ and $M$, in particular does not depend on the choice of $A$). Next, take $M$ so large that
		\begin{equation} \label{E CRW avoid LECRW inside B(M)}
			\PP(\{\wh{X}_n: n \leq T_M\} \cap \{\LE(\wh{Y})_n: n \in \NN\} \neq \{o\}) \geq 1 - \frac{\epsilon}{2},
		\end{equation}
		where $\wh{Y}$ is an independent CRW. This is possible by the intersection property \eqref{E: CRW IP} and by monotonicity. Hence for $u \notin \Lambda_a(L)$, combining \eqref{E: CRW X_n vs whX_n hits A} and \eqref{E CRW avoid LECRW inside B(M)}, conditioning on  $\LE(\wh{Y})$,
		\[
			\PP_u(X_{T_{\LE(\wh{Y})}} = o) \leq  \epsilon
		\]
	As $\epsilon$ was arbitrary, this shows the result.
	\end{proof}

	\subsubsection*{Wilson's algorithm rooted at infinity}
	Recall Wilson's algorithm for recurrent graphs: let $I = (v_0, v_1, \ldots)$ be any enumeration of the vertices $v(G)$. Fix $E_0 = \{v_0\}$ and define inductively $E_{i + 1}$ given $E_i$, to be $E_i$ together with the loop erasure of an (independent) simple random walk started at $v_{i + 1}$ and stopped when hitting $E_{i}$. Set $E = E(I) = \cup_{i \geq 0} E_i$. Then Wilson's algorithm tells us that the spanning tree $E$ is in fact a uniform spanning tree (i.e., its law is the weak limit of uniform spanning trees on exhaustions) and in particular, its law does not depend on $I$, see Wilson \cite{WilsonAlgorithm} for finite graphs and e.g. \cite{LyonsPeresProbNetworks} for infinite recurrent graphs.
	
	Since the conditioned random walk is well defined, we can also start differently: namely take again some enumeration $I = (v_0, \ldots)$ of $v(G)$. Define $F_0 = \LE(\wh{X})$, started at $v_0$ say and let $F_{i + 1}$ be $F_i$ together with the loop erasure of a simple random walk started at $v_{i + 1}$ and stopped when hitting $F_{i}$. Define $F = F(I) = \cup_{i \geq 0} F_i$. It is not hard to see that again, $F$ is a spanning tree of $G$ (the idea is that the loops formed by the walk coming back to the origin are erased anyway, so one might as well consider the conditioned walk). This is called ``Wilson's algorithm rooted at infinity''.
A similar idea was first introduced for transient graphs in \cite{BLPS} and later defined for $\ZZ^2$.
	
	\begin{lemma}[Wilson's algorithm rooted at infinity] \label{L: UST Wilson's algorithm}
		The spanning tree $F$ is a uniform spanning tree.
	\end{lemma}
	\begin{proof}
		Begin with $o$ and let $(z_n)_{n \geq 0}$ be some sequence of vertices going to infinity in $G$. Apply Wilson's algorithm with the orderings $I_n := (o, z_n, v_2, \ldots ) \equiv v(G)$, then the law of the first branch $E_1$ equals $\LE(X^{z_n \to o})$ by construction, where $X^{z_n \to o}$ is (the trace of) a random walk started at $z_n$ and stopped when hitting $o$. This law converges to $\LE(\wh{X})$ as $i \to \infty$ due to first the path-reversal (Lemma \ref{lemma: path reversal of RW}) and them Lemma \ref{L: RW cond to hit u is CRW}. Since Wilson's algorithm is independent of the ordering of $v(G)$, the result follows.
	\end{proof}

\paragraph{Orienting the UST.} When the UST is one-ended, it is always possible to unambiguously assign a consistent orientation to the edges (from each vertex there is a unique forward edge) such that the edges are oriented towards the unique end of the tree. Although we do not of course know \emph{a priori} that the UST is one-ended, it will be important for us to observe that the tree inherits such a consistent orientation from Wilson's algorithm rooted at infinity. Furthermore, this orientation does not depend on the ordering used in the algorithm. To see this, consider an exhaustion $G_n$ of the graph. Perform Wilson's algorithm (with initial boundary given by the boundary of $G_n$) and some given sequence of vertices. When adding the  branch containing the vertex $x$ to the tree by performing a loop-erased walk starting from $x$, orient these edges uniquely from $x$ to the boundary. 

We point out that it is not entirely clear \emph{a priori} that this orientation converges, the limit of the orientation does not on the exhaustion (indeed on $\ZZ$ the oriented tree converges but the orientation depends on the exhaustion, though the UST itself doesn't), and the law of the oriented doesn't depend on the sequence of vertices. But this follows readily from the fact that the loops at $x$ from a random walk starting from $x$ are all erased, so that the branch containing $x$ is obtained by loop-erasing a random walk conditioned to hit the boundary before returning to $x$, a process which has a limit as $n \to \infty$, is transient, and does not depend on the exhaustion used. The resulting orientation does not in fact depend on the ordering of vertices, since we can also see that this orientation is identical to the one where all edges are oriented towards $\partial G_n$, and the law of the tree itself does not depend on the ordering, as discussed before.

Once the oriented spanning tree has a limit, it is clear that the orientation of edges does not depend on the sequence of vertices, since this is true for all finite $n \ge 1$. Let $\vec{\mathcal{T}}$ denote the oriented unform spanning tree obtained in this fashion. Note that if $x,y$ are two vertices on a bi-infinite path of $\vec{\cT}$, then it makes sense to ask if $y$ is in the past of $x$ or vice-versa: exactly one of these alternatives must hold.

\medskip We are now ready to start with the proof of Theorem \ref{T: UST one-ended}.

	\begin{proof}[Proof of Theorem \ref{T: UST one-ended}]
		Notice that if $G$ is a graph satisfying the standing assumptions (recurrent, the potential kernel $a(x)$ is well defined and $a(x) \to \infty$ as $x \to \infty$) and is moreover planar or random and unimodular then (almost surely), $G$ satisfies the intersection property for CRW \eqref{E: CRW IP} due to Propositions \ref{prop: PLANAR CRW traces intersect i.o.} and \ref{prop: REVERSIBLE CRW traces intersect io} respectively.

Suppose $(G,o)$ is reversible, and satisfies the standing assumptions a.s. For a vertex $x$ of $G$, consider the event $\c{A}_2(x)$ that there are two \emph{disjoint and simple} paths from $x$ to infinity in the UST $\c{T}$, in other words there is a bi-infinite path going through $x$. Note that it is sufficient to prove
		\[
			\PP(\c{A}_2(x)) = 0
		\]
		for each $x \in v(G)$ a.s., where we remind the reader that here $\P$ is conditional given the graph (i.e., it is an average over the spanning tree $\cT$).
Indeed, for the tree $\c{T}$ to be more than one-ended, there must at least be some simple path in $\c{T}$ which goes to infinity in both directions. By biaising and unbiaising by the degree of the root to get a unimodular graph, it is sufficient to prove that $\P(\cA_2(o)) = 0$ a.s. Therefore it is sufficient to prove $\bf P (\cA_2(o)) = 0$, where we remind the reader that $\bf P$ is averaged also over the graph. Suppose for contradiction that $\bf P(\cA_2(o)) \ge \epsilon>0$. The idea will be to say that if this is the case then it is possible for both $\cA_2(o)$ and $\cA_2(x)$ to hold simultaneously, for many other vertices -- including vertices far away from $o$. However, $\cT$ is connected (since $G$ is recurrent) and by Theorem 6.2 and Proposition 7.1 in \cite{AldousLyonsUnimod2007}, $\cT$ is at most two-ended. Therefore the bi-infinite paths going through $x$ and $o$ must coincide: essentially, the bi-infinite path containing $o$ must be almost space-filling. 

Suppose $x$ is in the past of $o$ (which we can assume without loss of generality by reversibility). Using Wilson's algorithm rooted at infinity to sample first the path from $o$ and then that from $x$, the event $\cA_2(o) \cap \cA_2(x)$ requires a very unlikely behaviour: namely, a random walk starting from $x$ must hit the loop-erasure of the conditioned walk starting from $o$ exactly at $o$. This is precisely what Lemma \ref{L: UST limsup X hits LE at o} shows is unlikely, because of the infinite intersection properties.

Let us now give the details. Given $G$, we sample $k$ independent random walks $(X^1, \ldots, X^k)$ from $o$, independently of $\cT$, where $k = k(\epsilon)$ will be chosen below. Observe that by stationarity of $(G, o)$, we have for every $n\ge 0$,
$$
\bf P (\cA_2(X^i_n)) = \bf P (\cA_2(o)) \ge \epsilon.
$$
First we show that we can choose $k$ such that for every $n$, there is $i$ and $j$ such that  $\cA_2(X^i_n) \cap \cA_2(X^j_n)$ holds with $\bf P$- probability at least $\epsilon/2$. Indeed fix $n\ge 0$ arbitrarily for now, write $E_i = \cA_2(X^i_n)$. Then by the Bonferroni inequalities,
$$
\bf P ( \bigcup_{i=1}^k E_i) \ge \sum_{i=1}^k \bf P(E_i) - \sum_{1\le i \neq j\le k} \bf P (E_i \cap E_j)
$$
so that 
$$
\sum_{1\le i \neq j\le k} \bf P (E_i \cap E_j) \ge k \epsilon - \bf P ( \bigcup_{i=1}^k E_i) \ge k \epsilon - 1.
$$
Choose $k = \lceil 2/\epsilon\rceil$, then we deduce that for some $1\le i< j \le k$,
$$
\bf P ( E_i \cap E_j) \ge {k \choose 2}^{-1}.
$$
By stationarity (rerooting at the endpoint of the $i$th walk), and the Markov property of the walk, this implies
\begin{equation}
\bf P ( \cA_2(o) \cap \cA_2 (X_{2n}) ) \ge {k \choose 2}^{-1}.
\end{equation}
When $\cA_2(o) \cap \cA_2 (X_{2n})$ occurs, both $o$ and $X_{2n}$ are on some bi-infinite path, the two paths must coincide. By symmetry (i.e., reversibility) and invariance of the oriented tree $\vec{\cT}$ with respect to the ordering of vertices,
\begin{equation}
\label{eq:bothpath}
\bf P ( \cA_2(o) \cap \cA_2 (X_{2n}) ; X_{2n} \in \mathbf{Past}(o)) \ge \delta := (1/2) {k \choose 2}^{-1}.
\end{equation}
Let $\wh{Y}$ denote a conditioned walk starting from $o$ and let $\LE( \wh{Y})$ denote its loop-erasure, and let $Z$ be a random walk starting from a different vertex $x$. Now, pick $M$ large enough that for any $x \in \Lambda_a(M)^c$
\begin{equation}\label{eq:LEhit}
 \PP_x(Z_{T_{\LE(\wh{Y})}} = o)  \le \delta/3,
\end{equation}
which we may by Lemma \ref{L: UST limsup X hits LE at o}. Even though $M$ is random (depending only on the graph), observe that as $n \to \infty$, 
$$
\P( X_{2n} \in \Lambda_a(M)) \to 1
$$
since $G$ is a.s. null recurrent (as is any recurrent infinite graph). Therefore by dominated convergence, 
$$
\bf P (X_{2n} \in \Lambda_a(M)) \to 1.
$$
It follows using \eqref{eq:bothpath} that we may choose $n$ large enough that 
\begin{equation}\label{eq:bothpathfaraway}
\bf P ( \cA_2(o) \cap \cA_2 (X_{2n})  \cap \{ X_{2n } \in \mathbf{Past}(o)\} \cap\{ X_{2n} \notin \Lambda_a(M)\}) \ge 2 \delta /3.
\end{equation}
To conclude, we pick $n$ as above, and		
use Wilson's algorithm rooted at infinity (Lemma \ref{L: UST Wilson's algorithm}) by first sampling the path from $o$ (which is nothing else by $\LE (\wh{Y})$ and then sampling the path in $\vec{\cT}$ from $x = X_{2n}$, by loop-erasing a random walk $Z$ from this point, stopped at the time $T$ where it hits $\LE (\hat Y)$. As mentioned above, When $\cA_2(x) $ and $\cA_2(o)$ occur and $x$ is in the past of $o$, since $\cT$ is at most two-ended (by \cite{AldousLyonsUnimod2007}), it must be that $Z_T = o$. (If we do not specify that $x \in \mathbf{Past} (o)$ there might otherwise also be the possibility that $x$ itself was directly on the loop-erasure of the conditioned walk). Hence, using \eqref{eq:bothpathfaraway} and \eqref{eq:LEhit},
\begin{align*}
2\delta/3 &\le  \bf P ( \cA_2(o) \cap \cA_2 (X_{2n})  \cap\{X_{2n} \in \mathbf{Past} (o) \} \cap \{ X_{2n} \notin \Lambda_a(M)\}) \\
& \le \bf E (1_{\{Z_T = o\}} 1_{\{X_{2n} \notin \Lambda_a(M)\}} ) \\ 
&\le \bf E (  \P_{X_{2n}} (Z_T = o) 1_{ X_{2n} \notin \Lambda_a(M)} ) \le \delta /3,
\end{align*}
after conditioning on $X_{2n}$. This is a contradiction, and concludes the proof of Theorem \ref{T: UST one-ended} (and hence also that (a) implies (d) in Theorem \ref{T:main_equiv}).
\end{proof}

Furthermore, (d) is already known by \cite[Theorem 14.2]{BLPS} to imply (b), which we have already shown is equivalent to (a). This finishes the proof of Theorem \ref{T:main_equiv}.

	\begin{wrong-old}
	\subsubsection*{Wilson's algorithm rooted at infinity}
	Let $\mu$ denote the law of a uniform spanning tree of a recurrent graph $G$. Assume that the potential kernel is well-defined, so that the random walk conditioned to never return to its starting point is well-defined. Define Wilson's algorithm as follows.
	
	Take some ordering of $v(G)$, say $\{o, v_1, v_2, \ldots\}$ for the fixed vertex $o \in v(G)$. Let $E_0$ denote the trace of $\LE(\wh{X})$, where $\wh{X}$ is a CRW and $\LE(\cdot)$ denotes the loop erasure of this path. Given $E_i$, take $E_{i + 1}$ the loop erasure of an independent random walk started at $v_{i + 1}$ and stopped when hitting $E_{i}$, together with $E_i$. Notice that, if $v_{i + 1} \in E_i$, we get $E_{i + 1} = E_i$. Define $E = \bigcup_{i=1}^\infty E_i$, as a collection of vertices and edges. It is not hard to see that $E$ is a spanning tree of $G$.
	
	\begin{lemma}
		$E$ is a uniform spanning tree.
	\end{lemma}
	\begin{proof}
		Begin with $o$ and let $(z_i)_{i \geq 0}$ be some sequence of vertices going to infinity in $G$. Apply Wilson's algorithm with the orderings $I_i := (o, z_i, v_2, \ldots ) \equiv v(G)$, then the law of the first branch equals $\LE(X^{z_i \to o})$ by construction. This law converges to $\LE(\wh{X})$ as $i \to \infty$. Since Wilson's algorithm is independent of the ordering of $v(G)$, the result follows.
	\end{proof}
	
	Wilson's algorithm rooted at infinity offers more than just a uniform spanning tree: it gives an \emph{oriented} tree where, for each $x \in v(G)$, there exists exactly one vertex pointing outwards. We say such a tree is \textbf{rooted at infinity}. Write $\overrightarrow{T}$ for this oriented uniform spanning tree. We define the ``past'' of $x$ in the (rooted) spanning tree $\overrightarrow{T}$ to be the subtree of $\overrightarrow{T}$, made by the vertices such that there exists an oriented path in $\overrightarrow{T}$ to $x$ and denote this by $P(x)$. Similarly, we define the future $F(x)$ of $x$ to be the subtree $\overrightarrow{T} \setminus P(x)$ (and we include $x$ in both $F(x)$ and $P(x)$ for concreteness).
	
	\subsubsection*{The future contains almost everything.}
	Suppose $G$ satisfies the standing assumptions, and the ``intersection property''
	\begin{equation} 
		\PP(|\{\wh{X}_n: n \in \NN\} \cap \{\LE(\wh{Y})_n: n \in \NN\}| = \infty) = 1
	\end{equation}
	This holds for instance when $G$ satisfies the standing assumptions and is planar or unimodular. It is somewhat tempting to assume that this also implies that two loop erasures intersect infinitely often, yet this is an open question. However, due to Wislon's algorithm rooted at infinity, it does help to prove the following proposition.
	
	\begin{proposition} \label{P: UST past is almost empty}
		If $G$ satisfies the standing assumptions and \eqref{E: CRW IP}, then
		\[
			\limsup_{n \to \infty} \E \left[ \frac{|P(o) \cap B(n)|}{|B(n)|} \right] = 0.
		\]
	\end{proposition}
	
	We begin with a preliminary result, which is an extension of Lemma \ref{lemma: cRW does what it should}.
	
	\begin{lemma} \label{L: CRW cond to hit u}
		Let $\wh{X}$ be a CRW. For each $M \geq 1$, uniformly over all paths $\varphi$ going from $o$ to $M$ and staying inside $B(M)$, and $u \in \partial B(R)$
		\[
			\PP((\wh{X}_0, \ldots, \wh{X}_{T_M}) = \varphi) = \PP((X_0, \ldots, X_{T_M}) = \varphi \mid T_u < T_0^+)(1 + o_R(1)).
		\]
	\end{lemma}
	\begin{proof}
		Since $\wh{X}$ is a Markov process, the law of $\wh{X}_{0}, \ldots, \wh{X}_{T_M}$ depends on $\wh{X}_{T_M}, \ldots, \wh{X}_{T_u}$ only through $\wh{X}_{T_M}$. However, by Proposition \ref{prop: conditioned exit measure mixes} it follows that the exit measure
		\[
		\wh{\mu}_R(x, b) := \PP_x(\wh{X}_{T_R} = b)
		\]
		satisfies $\mu_R(x, b) = \mu_R(y, b)(1 + o_R(1))$ uniformly in $b \in \partial B(R)$ and $x, y \in B(M)$. This implies that $\wh{X}_0, \ldots, \wh{X}_{T_M}$ is asymptotically independent of where the walk exists $B(R)$.
	\end{proof}
	
	Next, we prove that the probability when $u \to \infty$, for the random walk started at $u$ to hit the loop erasure of $\wh{Y}$ at $o$ is asymptotically $0$.
	
	\begin{lemma}
		Let $G$ satisfy the standing assumptions and the intersection property for $\wh{X}$. Let $X$ denote a simple random walk started from $u$ and $\wh{Y}$ an independent CRW started at $o$. Then
		\[
			\limsup_{M \to \infty}\sup_{u \notin B(M)} \PP_u(X_{T_{\LE(\wh{Y})}} = o) = 0.
		\]
	\end{lemma}
	\begin{proof}
		Notice that, by independence of $\wh{Y}$ and $X$, it is thus enough to prove
		\[
			\PP_u(\{X_n: n \leq T_o\} \cap \{\LE(\wh{Y})_n: n \in \NN\} = \{o\} \mid T_o < T_u^+)
		\]
		goes to $0$ as $u \to \infty$ along any sequence. Indeed, any loop from $u$ that returns to $u$ before touching $o$ might also hit $\LE(\wh{Y})$ and therefore, conditioning on $o$ being hit before returning to $u$ increases (or does not decrease) the probability. By Lemma \ref{lemma: path reversal of RW}, this is equivalent to
		\[
			\PP_o(\{X_n: n \leq T_u\} \cap \{\LE(\wh{Y})_n: n \in \NN\} = \{o\} \mid T_u < T_o^+) \to 0
		\]
		uniformly in $u \in \partial B(M)$ as $M \to \infty$.
		
		Let $\wh{X}$ be an independent CRW started at $o$. Fix $\epsilon > 0$ and let $M_1$ be so large that
		\[
		\PP(|\{\wh{X}_n: n \leq \wh{T}^{\wh{X}}_{M_1}\} \cap \{\LE(\wh{Y})_n: n \in \NN\}| > 1) \geq 1 - \epsilon.
		\]
		This is possible by the intersection property \eqref{E: CRW IP} and because $\wh{X}$ will avoid $o$ eventually. Let $M_2 \geq M_1$ be so large that for all paths $\mathcal{P}$ that stay inside $B(M_1)$, we have
		\[
			\PP_o(\{X_n: n \leq T_u\} \in \mathcal{P} \mid T_u < T_o^+) \geq \PP_o(\{\wh{X}_n: n \leq T_u\} \in \mathcal{P}) - \epsilon,
		\]
		for all $u \notin B(M_2)$, which is possible by Lemma \ref{L: CRW cond to hit u}. The two last inequalities joined together give that for all $u \notin B(M_2)$,
		\[
			\PP_o(\{X_n: n \leq T_u\} \cap \{\LE(\wh{Y})_n: n \in \NN\} = \emptyset \mid T_u < T_o^+) < 2\epsilon.
		\]
		As $\epsilon$ was arbitrary, we obtain the result.
	\end{proof}

	\begin{proof}[Proof of Proposition \ref{P: UST past is almost empty}]
		Fix $\epsilon > 0$ and then $M$ so large that for all $u \notin B(M)$
		\[
			\PP_u(X_{T_{\LE(\wh{Y})}} = o) \leq \frac{\epsilon}{2},
		\]
		possible by Lemma \ref{L: UST limsup X hits LE at o}. Next, let $N$ be so large that
		\[
			\frac{|B(M)|}{|B(N)|} < \frac{\epsilon}{2}.
		\]
		It follows from Wilson's algorithm (as it is invariant under the ordering of the vertices) that
		\[
			\E \left[ \frac{|P(o) \cap B(n)|}{|B(n)|} \right] \leq \frac{1}{|B(n)|} \sum_{u \in B(n)} \PP_u(X_{T_{\LE(\wh{Y})}} = o) < \epsilon,
		\]
		for all $n \geq N$ showing the desired result.
	\end{proof}
	
	\subsubsection*{Reversing the UST rooted at infinity}
	The goal of this section will be to show that if $A_2(o)$ occurs, then the past cannot be ``almost empty'' in the sense of Proposition \ref{P: UST past is almost empty}. To do so, we will assume again that the underlying graph satisfies the standing assumptions so that we can make sense of the CRW. Then, we will try to ``reverse'' the orientation of some bi-infinite ray in $\overrightarrow{T}$ to obtain a new spanning tree - which we prove has the same distribution as $\overrightarrow{T}$. This implies that the future cannot be ``bigger'' than the past, which will finish the proof of Theorem \ref{T: UST one-ended}.
	
	If $\overrightarrow{T}$ is a spanning tree rooted at infinity (for each vertex, there is a unique oriented ray to infinity), we define a random walk on the infinite rays as follows. Take $\overrightarrow{T}$ and switch the orientation of each edge that is on an infinite ray of $o$, (meaning that there is a \emph{simple} path from $o$ to infinity on the unoriented tree $T$ to infinity passing by the edge). Denote this new oriented graph $\overrightarrow{S}$
	
	Now, let $W$ be a simple random walk on $\overrightarrow{S}$ started at $o$, which can only follow oriented edges in the natural way. The trace of $W$ is a ray from $o$ to infinity in the past of $o$ on $\overrightarrow{T}$, if such a ray exists. Otherwise the trace of $W$ is a finite path.
	
	Define next the oriented bi-infinite ray $\mathrm{RAY}(o)$ by setting it equal to the unique ray (a simple path) from $o$ to infinity in the future of $o$ together with the trace of $W$ as a subgraph of $\overrightarrow{T}$. Thus, $\mathrm{RAY}(o)$ is a by-infinite path passing through $o$ with a fixed orientation induced by $\overrightarrow{T}$ whenever $A_2(o)$ occurs.
	
	This allows to define a spanning tree $\overleftarrow{T}$ as follows. If $A_2(o)$ does not occur, set $\overleftarrow{T} = \overrightarrow{T}$. If $A_2(o)$ does occur, let $\overleftarrow{T}$ be the oriented tree obtained from $\overrightarrow{T}$ by switching the orientation of $\mathrm{RAY}(o)$. Clearly, this gives an oriented spanning tree which is ``rooted at infinity''. We will prove that $\overrightarrow{T}$ has the same distribution as $\overleftarrow{T}$, so that the latter is in fact a uniform spanning tree rooted at infinity.
	
	In what follows, we will write $T$ for the uniform spanning tree of $G$ without orientation and $T(x \to y)$ for the (oriented) path $x \to y$ in $T$. The notation $X^{x \to y}$ stands for the law of (the trace of) a simple random walk started at $x$, stopped when hitting $y$. Lastly, we introduce the notation $\c{B}_2(M)$ for the collection of \emph{simple} paths that connect $\partial B(M)$ with $\partial B(M)$, stay otherwise inside $B(M)$ and pass by $o$.
	
	The following lemma is an immediate result from Wilson's algorithm.
	
	\begin{lemma} \label{L: T(x, y) equals LE(x, y)}
		Let $x, y$ be any two vertices, then for any path $\varphi$ connecting $x$ to $y$ and passing by $o$,
		\[
			\PP(T(x \to y) = \varphi \mid o \in T(x \to y)) = \PP(\LE(X^{x \to y}) = \varphi \mid o \in \LE(X^{x \to y})).
		\]
	\end{lemma}
	\begin{proof}
		Fix $\varphi$ as above, then by Wilson's algorithm, we know that
		\[
			\PP(T(x \to y) = \varphi) = \PP(\LE(X^{x \to y}) = \varphi).
		\]
		This shows the result since the sum over all allowed $\varphi$ provides $\PP(o \in T(x \to y)) = \PP(o \in \LE(X^{x \to y}))$.
	\end{proof}

	Next, we provide a technical lemma, giving a notion of asymptotic independence for the CRW, which is essentially a consequence of Proposition \ref{prop: conditioned exit measure mixes}.
	
	\begin{lemma} \label{L: CRW asymp indep}
		For $M, L \in \NN$ define $\c{F}_{M} = \sigma(\wh{X}_0, \ldots, \wh{X}_{\wh{T}_M})$ and $\c{F}_{L, \infty} = \sigma(\wh{X}_{\wh{T}_L}, \ldots)$. For each $\epsilon > 0$ and all $M \in \NN$, there exists an $L_0$ such that for all $L \geq L_0$ and uniformly in $A \in \c{F}_M, B \in \c{F}_{L, \infty}$ it holds that
		\[
			\PP(A \cap B) = \PP(A)\PP(B)(1 \pm \epsilon).
		\]
	\end{lemma}

	\begin{proof}
		Let $M \in \NN$, $L \geq L_0$ - where we fix $L_0$ later - and $\epsilon > 0$. Since $\wh{X}$ is a Markov chain, it is enough to prove the result for $A \in \sigma(\wh{X}_{\wh{T}_M})$ and $B \in \sigma(\wh{X}_{\wh{T}_L})$. Let $(\alpha_x)_{x \in \partial \Lambda_a(M)}$ and $(\beta_y)_{y \in \partial \Lambda_a(L)}$ be two sequences of real numbers such that
		\[
			\id_{A} = \sum_{x \in \partial \Lambda_a(M)} \alpha_x \id_{\wh{X}_{\wh{T}_M} = x} \qquad \text{and} \qquad \id_{B} = \sum_{y \in \partial \Lambda_a(L)} \beta_y \id_{\wh{X}_{\wh{T}_L} = y}.
		\]
		Notice that Proposition \ref{prop: conditioned exit measure mixes} shows that for $L_0$ large enough
		\[
			\wh{\mu}(x, y) = \wh{\mu}(z, y)(1 \pm \epsilon),
		\]
		uniformly in $x, z \in \partial \Lambda_a(M)$ and $y \in \partial \Lambda_a(L)$. In particular, it holds that
		\[
			\wh{\mu}(x, y) = (1 \pm \epsilon) \sum_{z \in \partial \Lambda_a(M)} \PP(\wh{X}_{\wh{T}_M} = z)\wh{\mu}(z, y) = (1 \pm \epsilon) \PP(\wh{X}_{\wh{T}_L} = y),
		\]
		uniformly in $x$ and $y$. Therefore,
		\begin{align*}
			\PP(A \cap B) &= \sum_{x \in \partial \Lambda_a(M), y \in \partial \Lambda_a(L)} \alpha_x \beta_y \PP(\wh{X}_{\wh{T}_M} = x, \wh{X}_{\wh{T}_L} = y) \\
			&= \sum_{x \in \partial \Lambda_a(M), y \in \partial \Lambda_a(L)} \alpha_x \PP(\wh{X_{\wh{T}_M}} = x)\wh{\mu}(x, y) \\
			&= (1 \pm \epsilon )\sum_{x \in \partial \Lambda_a(M)} \alpha_x \PP(\wh{X}_{\wh{T}_M} = x) \sum_{y \in \partial \Lambda_a(L)} \beta_y\PP(\wh{X}_{\wh{T}_M} = y) = (1 \pm \epsilon)\PP(A)\PP(B)
		\end{align*}
	\end{proof}
	
	Next, we prove that locally, the law $T(x \to y)$ conditioned to contain $o$ does not depend on $x, y$ when $x, y$ are far away.
	
	\begin{lemma} \label{L: T(x, y) locally independent xy}
		For each $M \in \NN, \epsilon > 0$, there exists a $K \in \NN$ such that for all $x, y. u, v \notin B(K)$ and all $\varphi \in \c{B}_2(M)$
		\[
			\PP(T(x \to y)_M = \varphi \mid o \in T(x \to y)) = \PP(T(u \to v)_M = \varphi \mid o \in T(u \to v))(1 \pm \epsilon).
		\]
	\end{lemma}

	\begin{proof} \DE{detailed sketch}
		Fix first $M \in \NN$, $\epsilon > 0$ and let $K > L > M$ to be chosen large enough later. Take $x, y \in B(K)^c$. By Lemma \ref{L: T(x, y) equals LE(x, y)}, it follows that the law of $T(x \to y)_M$ conditioned to contain $o$ is that of $\LE(X^{x \to y})_M$, conditioned to contain $o$. Then, by what is called the ``domain Markov property'', it holds that
		\[
			\PP(\LE(X^{x \to y})_M = \varphi \mid o \in \LE(X^{x \to y})) = \PP(\LE(X^{x \to o})_M = \varphi_1, \LE(\wh{S})_M = \varphi_2 \mid \wh{S} \text{ touches } y),
		\]
		where $\wh{S}$ is a random walk, conditioned to never hit $\LE(X^{x \to o})$ and $\varphi_1, \varphi_2$ is the unique decomposition of $\varphi$ into two pieces connecting $o$ to $\partial B(M)$ and say $\varphi_2$ does not contain $o$.
		
		Notice that inside $B(L)$, we have that $\LE(X^{x \to o})$ has the law of $\LE(\wh{X})$, up to a small error as $K = K(L)$ is chosen large enough. If also $L = L(M)$ is large enough, then the law of $\wh{S}$ inside $B(M)$ only depends on (up to a small error) $\LE(X^{x \to o})$ inside $B(L)$ due to Lemma \ref{L: CRW asymp indep}, which looks like $\LE(\wh{X})$ by the above. Moreover, picking possibly $K$ larger, it will not depend on $y$ up to a small error. Hence
		\[
			\PP(\LE(X^{x \to o})_M = \varphi_1, \LE(\wh{S})_M = \varphi_2 \mid \wh{S} \text{ touches } y) = \PP(\LE(\wh{X})_M = \varphi_1, \LE(\wh{S}')_M = \varphi_2)(1 \pm \epsilon)
		\]
		where now $\wh{S}'$ is a random walk conditioned to never hit $\LE(\wh{X})_L$. Since the expression on the right-hand side does not depend on $x$ or $y$, the result follows.
	\end{proof}

	The next result is an immediate corollary of this lemma, showing that the oriented paths $T(x \to y)$ and  $T(y \to x)$, both conditioned on containing $o$, are locally around $o$ the same whenever $x, y$ are far away from $o$.
	
	\begin{corollary}
		For each $M \in \NN, \epsilon > 0$, there exists a $K = K(\epsilon, M)$ such that for all $x, y \notin B(K)$ and all \emph{oriented} $\varphi \in \mathcal{B}_2(M)$
		\[
			\PP(T(x \to y) = \varphi \mid o \in T(x \to y)) = \PP(T(y \to x) = \varphi \mid o \in T(y \to x))(1 \pm \epsilon).
		\]
	\end{corollary}
	
	We finally use this to prove that $\overleftarrow{T}$ is in fact a uniform spanning tree rooted at infinity, after which we can prove Theorem \ref{T: UST one-ended}.

	\begin{proposition}[Switching the orientation] \label{L: UST switching orientation}
		The distribution of $\overleftarrow{T}$ is the same as the distribution of $\overrightarrow{T}$.
	\end{proposition}

	\begin{proof}
		Since any long path containing $o$ inside $T$ is symmetric under reversing its orientation, we obtain that the law of $\mathrm{RAY}(o)$ must also be symmetric under changing its orientation, conditionally on $A_2(o)$ occurring. In particular, the coupling $(\overrightarrow{T}, \overleftarrow{T})$ is distributed as $(\overleftarrow{T}, \overrightarrow{T})$ and therefore, $\overrightarrow{T} \overset{d}{=} \overleftarrow{T}$.
	\end{proof}
	
	\begin{proof}[Proof of Theorem \ref{T: UST one-ended}]
		Notice first that as $G$ satisfies the standing assumptions and is unimodular or planar, it satisfies the intersection property for $\wh{X}$ due to Propositions \ref{prop: REVERSIBLE CRW traces intersect io} and \ref{prop: PLANAR CRW traces intersect i.o.} respectively.
		
		Since $\overrightarrow{T}$ has the same distribution as $\overleftarrow{T}$ by Lemma \ref{L: UST switching orientation} it follows that when $A_2(o)$ occurs, the past of $o$ cannot be smaller than the future (indeed, there might be more ways to switch the orientation, any such way will make the future into the past), hence for each $n$,
		\[
			\E\left[\frac{|F(o) \cap B(n)|}{|B(n)|} \Big| A_2(o)\right] \leq \E \left[ \frac{|P(o) \cap B(n)|}{|B(n)|} \Big| A_2(o) \right].
		\]
		If $\PP(A_2(o)) > 0$, this contradicts Lemma \ref{L: UST limsup X hits LE at o}. Therefore, $\PP(A_2(o)) = 0$ as desired.
	\end{proof}
	\end{wrong-old}
	\section{Harmonic measure from infinity on mated-CRT maps} \label{section: examples}
	
	Let $\bf{P}$ denote the law of the mated-CRT map $G = G^{1}$ with parameter $\gamma \in (0,2)$ and with root $o$. We will not give a precise definition of these maps here and instead refer the reader for instance to \cite{GwynneMillerSheffieldTutteCRT} or \cite{NathanaelEwainCRT}. As already mentioned, it is straightforward to check that the theorem guaranteeing the existence of the potential kernel (Theorem \ref{T:main_existence}) applies a.s. to these maps. We now discuss a more quantitative statement concerning the harmonic measure from infinity which underlines substantial differences with the usual square lattice.

	We will write $\Beuc(x, n)$ for the ball of vertices $z \in v(G)$ such that the Euclidean distance between $z$ and $x$ (w.r.t. the natural embedding) is at most $n$.
	
	\begin{theorem}\label{theorem: CRT hm bounds}
		There exists a $\delta = \delta(\gamma) > 0$ such that the following holds. Almost surely, there exits an $N \geq 1$ such that for all $x \notin \Beuc(N)$ we have that
		\[
			\hm_{o, x}(x) \leq 1 - \delta.
		\]
		In particular,
		\[
			\bf{P}\left(\frac{1}{2} \leq \limsup_{y \to \infty} \wh{q}(y) \leq 1 - \delta\right) = 1.
		\]
	\end{theorem}
	In fact, we expect the following stronger result to hold:

	\begin{conjecture}\label{Conj}
		For some (nonrandom) $a,b>0$, almost surely
		\begin{equation}\label{Conj:hmbounds}
				a = \liminf_{y \to \infty} \wh{q}(y) \le \limsup_{y \to \infty} \wh{q}(y) =1 - b.
		\end{equation}
	\end{conjecture}
	In fact, sharp values for $a,b$ can be conjectured by considering the minimal and maximal exponents for the LQG volume of a Euclidean ball of radius $\varepsilon$ in a $\gamma$-quantum cone, which all decay polynomially as  $\varepsilon \to 0$ (see Lemma A.1 in \cite{NathanaelEwainCRT}). We also conjecture that this holds for, say, the UIPT.

	Based on this we conjecture that $\max(a,b) < 1/2$. This would show a stark contrast with the square lattice $\mathbb{Z}^2$ where we recall that $a = b = 1/2$ (see e.g. \cite{PopovRW}).
	The upper bound in \eqref{Conj:hmbounds} is of course stated in Theorem \ref{theorem: CRT hm bounds} so that the lower bound in \eqref{Conj:hmbounds} is what we are asking about.
	While we are not able to prove this, we may use the unimodularity of the law $\bf{P}$ is unimodular, to prove a slightly weaker lower bound:
	\begin{corollary}
		Let $\delta > 0$ as in the previous theorem. Then, almost surely, the asymptotic fraction of $\delta$-good points equals one or in other words, a.s.,
		\[
			\liminf_{n \to \infty} \frac{1}{|B(n)|}|\{x \in B(n): \hm_{o, x}(x) < \delta\}| = 0.
		\]
	\end{corollary}

	\begin{proof}
		Let $\tilde{\bf{P}}$ denote the law $\bf{P}$ after degree biasing. We write $(\tilde G, \tilde o)$ for the random graph with law $\tilde{\bf{P}}$.
		
		On the one hand, by reversibility of $\tilde{\bf{P}}$, we know that
		\[
			\tilde{\bf{P}}(\hm_{\tilde{o}, X_n}(X_n) > 1 - \delta) = \tilde{\bf{P}}(\hm_{\tilde{o}, X_n}(o) > 1 - \delta) = \tilde{\bf{P}}(\hm_{\tilde{o}, X_n}(X_n) < \delta).
		\]
		On the other hand, by Theorem \ref{theorem: CRT hm bounds} and the reversed Fatou's lemma, we have
		\[
			\limsup_{n \to \infty}\tilde{\bf{P}}(\hm_{\tilde{o}, X_n}(X_n) > 1 - \delta) = 0,
		\]
		thus
		\[
			\lim_{n \to \infty}\tilde{\bf{P}}(\hm_{\tilde{o}, X_n}(X_n) < \delta) = 0.
		\]
		The result now follows by contradiction: indeed, suppose that with positive probability, there is a positive asymptotic fraction of vertices $x \in B(n)$ which have $\hm_{\tilde{o}, x}(x) < \delta$, then the random walk will spend a positive fraction of time in these points, giving a contradiction.
	\end{proof}

	\subsection{Preliminaries and known results.} \label{sec: CRT prelim}
	We collect some known results about mated-CRT maps which are needed for the proof of Theorem \ref{theorem: CRT hm bounds}.
	
	\begin{lemma} \label{lemma: CRTresistanceBound}
		There exist $C = C(\gamma) < \infty$ and $\alpha = \alpha(\gamma) > 0$, such that for all $n \in \NN$,
		\[
			\bf{P}\left( \frac{1}{C} \log(n) \leq \Reff(o \leftrightarrow \partial \Beuc(o, n)) \leq C \log(n) \right) \geq 1 - \frac{1}{\log(n)^\alpha}.
		\]
	\end{lemma}
	\begin{proof}
		This is Proposition 3.1 in \cite{GwynneMiller2020}.
	\end{proof}
	
	\begin{wrong-old}
	\begin{lemma} \label{lemma: CRT RW close to curve}
		There exist $\alpha = \alpha(\gamma) > 0$, $C = C(\gamma)$ and $s_0 > 0$ such that the following holds. Let $\c{P}$ be a path of finite length in the plane. For $r > 0$, write $D(\c{P}, r) = \{x \in v(G): |\eta(x) - \c{P}| \leq r\} \subset \C$. Take $n/10 \leq \tilde{r} < r$ and let $N$ denote the minimum number of euclidean balls of radius $s_0(r - \tilde{r})$ needed to cover $D(\c{P}, r)$. Then with $\c{P}$-probability at least $1 - CN n^{-\alpha}$, we have
		\[
			\min_{x \in D(\c{P}, \tilde{r})} \min_{z \in D(\c{P}, \tilde{r})} \PP_x(X \text{ hits } \Beuc(z, \tilde{r}) \text{ before exiting } D(\c{P}, r)) > 2^{-N}.
		\]
	\end{lemma}

	\begin{proof}
		An immediate consequence of Proposition ... in \cite{GwynneMillerSheffield-CRTHarmonics}.
	\end{proof}
	\end{wrong-old}
	
	\begin{lemma} \label{lemma: CRTharnack}
		There exists a $C = C(\gamma) < \infty$ and $\alpha = \alpha(\gamma) > 0$ such that with $\bf{P}$-probability at least $1 - n^{-\alpha}$, for all $x \in \Beuc(3n)$ and all $s \in [1/3, 1]$
		\[
			\max_{z \in \partial \Beuc(x, sn)} h(z) \leq C \min_{x \in \Beuc(x, sn)} h(z)
		\]
		whenever $h: \Beuc(x, 3n) \cup \partial \Beuc(x, 3n) \to \RR_+$ is harmonic outside of possibly $x$ and $\partial \Beuc(x, 3n)$.
	\end{lemma}
	\begin{proof}
		This is the content of Proposition 3.8 \cite{NathanaelEwainCRT}.
	\end{proof}

	\begin{lemma} \label{lemma: CRTeffres lowerbound}
		There exist $C = C(\gamma) < \infty$ and $\alpha = \alpha(\gamma)$ such that with $\bf{P}$-probability at least $1 - n^{-\alpha}$, for all $x \in \Beuc(3n)$,
		\[
			\Reff(x \leftrightarrow \Beuc(x, n)) \geq \frac{1}{C} \log(n).
		\]
	\end{lemma}
	\begin{proof}
		This follows from Lemma 4.2 in \cite{NathanaelEwainCRT}.
	\end{proof}

	\begin{proposition} \label{prop: PK on CRT bounds root}
		Let $(G, o)$ have the law of the mated-CRT map with parameter $\gamma$. There exist constants $C = C(\gamma)$ and $\alpha  = \alpha(\gamma) > 0$ such that
		\[
			\bf{P}\Big(\frac{1}{C} \log(n) \leq a(x, o) \leq C \log(n) \text{ for all } x \in \Beuc(o, 2n) \setminus \Beuc(o, n) \Big) \geq 1 - \frac{1}{\log(n)^\alpha}.
		\]
	\end{proposition}
	\begin{proof}[Proof of Proposition \ref{prop: PK on CRT bounds root}]
		By Lemma \ref{lem: green function finite set and PK} we know that for each $n \in \NN$
		\[
			\Reff(o \leftrightarrow \Beuc(o, 2n)) = \E_o[a(X_{T_{2n}}, o)]
		\]
		(where we recall that $\E_o$ is the expectation solely on the random walk). Now, fix $n$ and let $\c{E}_n$ be the intersection of both events in Lemmas \ref{lemma: CRTresistanceBound} and \ref{lemma: CRTharnack}, which are properties of the graph only. Note that  $\cE_n$ holds with high probability over the mated-CRT maps (possibly by suitably changing the values of the constants).

		Then, as $x \mapsto a(x, o)$ is harmonic outside $o$, conditional on $\c{E}_{2n}$, we know that whenever $x \in \Beuc(o, 2n) \setminus \Beuc(o, n)$ ,
		\[
			\frac{1}{C} a(X_{T_{2n}}, o) \leq a(x, o) \leq C a(X_{T_{2n}}, o),
		\]
		so that taking (random walk) expectations,
		\[
			\frac{1}{C^2}\log(n) \leq a(x, o) \leq C^2\log(n).
		\]
		This is the desired result.
	\end{proof}

	\subsection{Proof of Theorem \ref{theorem: CRT hm bounds}.}
	
	\begin{wrong-old}
	 Define next for all values $l \in \RR$ the half-line $L(l) = \{x + il: x \in \RR\}$. Let $\Beuc^+(l)$ denote the ``positive'' part of the ball $\Beuc(l)$ (the vertices in $\Beuc(l)$ for which $\eta$ is in the upper half-plane: $\Beuc^+(l) = \{x \in \Beuc(l): \eta(x) \in \mathbb{H}_+\}$. Define $\Beuc^-$ similarly. The following lemma says that the random walk can go quite far ``down'' before ever hitting $(\Beuc^+(\ph(n))^c$.
	
	\begin{lemma} \label{lemma: CRT RW does goes down}
		There exists a $p = p_0 > 0$ and a $\tilde{C} < \infty$, such that the following holds. With $\bf{P}$-probability at least $1 - \tilde{C}n^{-2}$ we have that
		\[
			\PP_o(X \text{ hits } L(-3\ph(n)) \text{ before } (\Beuc^+(\ph(n - 1)))^c) \geq p_0.
		\]
	\end{lemma}

	\begin{proof}		
		We will rely on Lemma \ref{lemma: CRT RW close to curve}, which tells us that the random walk will stay close to a deterministic path $\c{P}$ in the plane with high probability.
		
		Fix $n > 2$ an integer. Let $z_i = (0, -2^{i - 1}\ph(n - 1))$, with the convention that $z_0 = 0$. Consider the path $\c{P}_i$ connecting $z_{i}$ with $z_{i + 1}$ in a straight line. Take $\tilde{r}_i = 2^{i - 1}\ph(n - 1) / 4$ and $r_i = 2^{i - 1}\ph(n - 1)$. Then we can cover the set $B_{\tilde{r}_i}(\c{P}_i)$ with say $C_1$ balls of radius $s_0r_i$ (where $C_1$ does not depend on $i$ or $n$). Letting $\c{E}_i$ be the event that, uniformly in $x \in \Beuc(z_i, \tilde{r})$,
		\[
			\PP_x \big(X \text{ hits } \Beuc(z_{i + 1}, \tilde{r}_i) \text{ before hitting } (\Beuc^+(\ph(n - 1)))^c\big) \geq 2^{-C_1},
		\]
		we know that $\bf{P}(\c{E}_i) \geq 1 - C_1\left(\frac{1}{\ph(n - 1)}\right)^\alpha$ uniformly in $i$ and $n$. Moreover, there exists a constant $C_2 \in \NN$ depending only on $\alpha$, such that
		\[
			\log\left(\frac{3\ph(n)}{(\ph(n - 1))}\right) \leq C_2
		\]
		uniformly in $n$. Thus, if the events $\c{E}_i$ occur for each $i = 0, \ldots, C_2$, then actually
		\[
			\PP_o \big(X \text{ hits } L(-3\ph(n)) \text{ before } (\Beuc^+(\ph(n - 1)))^c \big) \geq 2^{-C},
		\]
		with $C = (C_2 + 1)C_1$ since clearly, $\Beuc(z_{i + 1}, \tilde{r}_i) \subset \Beuc(z_{i + 1}, \tilde{r}_{i + 1})$ by definition.
		
		On the other hand, by a simple union bound, we have
		\[
			\bf{P}\left(\bigcap_{i = 1}^{C_2} \c{E}_i\right) \geq 1 - C_2C_1\left(\frac{1}{\ph(n - 1)}\right)^\alpha \geq 1 - \tilde{C}\frac{1}{n^2},
		\]
		showing the desired result.
	\end{proof}
	\end{wrong-old}
	Take throughout the proof the constants $C, \alpha$ such that Lemmas \ref{lemma: CRTeffres lowerbound}, \ref{lemma: CRTharnack} and \ref{lemma: CRTresistanceBound} and Proposition \ref{prop: PK on CRT bounds root} hold simultaneously with the same constants.
	\begin{proof}
		The second statement follows immediately from the first statement, from the identity
		\[
			\limsup_{y \to \infty} \wh{q}(y) = \limsup_{y \to \infty} \hm_{y, o}(y),
		\]
		in Corollary \ref{cor: hm and q are equivalent}, and from the fact that for each $\epsilon > 0$, there are infinitely many $(\frac{1}{2}-\epsilon)$-good vertices by Lemma \ref{lemma: reversible graphs are delta-good}. We are thus left to prove the first statement.
		
		To that end, fix $N_0$ so large that for all $n \geq N_0$,
		\[
			\frac{n^{2/\alpha}}{(n - 1)^{2/\alpha}} \leq 3.
		\]
		Define next for $m \geq 1$ the event
		\begin{equation}
			E_m \text{ the event that } a(x, o) \leq C\log(m) \text{ for all } x \in \Beuc(m).
		\end{equation}
		By Proposition \ref{prop: PK on CRT bounds root}, we know that $\bf{P}(E_m^c) \leq \log(m)^{-\alpha}$ and therefore,
		\[
			\sum_{n = 1}^\infty \bf{P}(E_{e^{n^{2/ \alpha }}}^c) < \infty.
		\]
		By Borel-Cantelli, this implies that there is some (random) $N_1 = N_1(G, o) < \infty$ such that $E_{e^{n^{2 / \alpha}}}$ occurs for all $n \geq N_1$. Suppose without loss of generality that $N_1 \geq N_0$ almost surely. In this case, it follows that
		\begin{equation} \label{E: a(x) bound all x}
			a(x, o) \leq C \log|x| \qquad \text{for all} \quad x \notin \Beuc(o, N_1).
		\end{equation}
		Next, define the events
		\begin{equation*}
			\begin{gathered}
				H_m \text{ the event that for all } x \in \Beuc(3m) \setminus \Beuc(m) \text{ and for all } \\
				h:v(G) \to \RR_+ \text{ harmonic outside of } x,\\
				\max_{z \in \partial \Beuc(x, |x|)} h(z) \leq C \min_{z \in \partial \Beuc(x, |x|)} h(z)
			\end{gathered}
		\end{equation*}
		and
		\begin{equation*}
			R_m \text{ the event that for all } x \in \Beuc(3m), \Reff(x \leftrightarrow \partial \Beuc(x, m)) \geq \frac{1}{C} \log(m).
		\end{equation*}
		By Lemmas \ref{lemma: CRTharnack} and \ref{lemma: CRTeffres lowerbound} respectively, it holds that $\bf{P}(H_m^c) \leq m^{-\alpha}$ and $\bf{P}(R_m) \leq m^{-\alpha}$. Therefore, using again a Borel-Cantelli argument, there exists some (random) $N_2 \geq N_1 \geq N_0$ such that almost surely, for all $n \geq N_2$ the events $H_{n^{2/\alpha}}$ and $R_{n^{2/\alpha}}$ occur. In particular, we know that almost surely,
		\begin{equation} \label{E: CRT R(x-o) lower bound}
			\Reff(x \leftrightarrow \Beuc(x, |x|)) \geq \frac{1}{C} \log|x| \qquad \text{for all} \quad x \notin \Beuc(o, N_2)
		\end{equation}
		and almost surely
		\begin{equation} \label{E: CRT Harnack at x}
			\begin{gathered}
				\text{For all } x \notin \Beuc(o, N_2), \text{ for all } h: v(G) \to \RR_+ \text{ harmonic outside of } x \\
				\max_{z \in \partial \Beuc(x, |x|)} h(z) \leq C \min_{z \in \partial \Beuc(x, |x|)} h(z).
			\end{gathered}
		\end{equation}
		
		Take $x \notin \Beuc(o, N_2)$. Assume without loss of generality that $\hm_{o, x}(x) \leq \frac{1}{2}$, as otherwise we are done. Then
		\begin{equation} \label{E: CRT R(o-x) upper bound}
			\Reff(o \leftrightarrow x) \leq 2\hm_{o, x}(x)\Reff(o \leftrightarrow x) = 2a(x, o) \leq 2C\log|x|,
		\end{equation}
		where we used Corollary \ref{cor: the harmonic measure for two points is well-defined} in the equality and \eqref{E: a(x) bound all x} in the last inequality.
		
		Furthermore, as $z \mapsto a(z, x)$ is harmonic outside of $x$, applying \eqref{E: CRT Harnack at x} first and then \eqref{E: CRT R(x-o) lower bound} gives
		\[
			a(o, x) \geq \frac{1}{C} \E_x[a(X_{T_{\Beuc(x, |x|)}}, x)] = \frac{1}{C} \Reff(x \leftrightarrow \Beuc(x, |x|)) \geq \frac{1}{C^2} \log|x|.
		\]
		Combining the last equation with \eqref{E: CRT R(o-x) upper bound}, we find
		\[
			\hm_{o, x}(o) = \frac{a(o, x)}{\Reff(o \leftrightarrow x)} \geq \frac{1}{2C^3},
		\]
		which shows the final result.
	\end{proof}

	\begin{wrong-old}
		\DE{Some ideas on ``symmetry''.
	We establish next a general result concerning unimodular environments, which shows that it is enough to find a non-trivial upper bound on the map $x \mapsto \hm_{o, x}(x)$ in order to find non-trivial upper \emph{and} lower bounds. Since the mated-CRT map is unimodular, we can apply this proposition.
	\begin{proposition}
		Let $(G, o)$ be unimodular random environment for which the harmonic measure from infinity is a.s. well defined. If there exists an $n \in \NN \cup \{\infty\}$ and a $\delta > 0$ such that almost surely $|\{x \in B(o, n): \hm_{x, o}(x) > 1 - \delta\}| = 0$, then also almost surely $|\{x \in B(o, n): \hm_{o, x}(x) < \delta\}| = 0$.
	\end{proposition}
	\begin{proof}
		Let $\delta > 0$ as in the statement of this proposition and fix $n \in \NN$ accordingly. Define the function
		\[
			f_n(G, o, x) = \id_{\hm_{x, o}(x) > 1 - \delta} \id_{x \in B(o, n)}.
		\]
		This map is clearly positive, but also measurable as we can see that $(G, o, x) \mapsto \hm_{x, o}(x)$ is continuous (by definition of the harmonic measure from infinity), so it is a mass transport. Moreover, we have
		\[
			\sum_{x \in v(G)} f_n(G, o, x) = \sum_{x \in B_n(o)} \id_{\hm_{x, o}(x) > 1 - \delta} = |\{x \in B(o, n): \hm_{x, o}(x) > 1 - \delta\}|.
		\]
		Since $x \in B(o, n)$ precisely when $o \in B(x, n)$, we get
		\[
			f_n(G, x, o) = \id_{\hm_{x, o}(o) > 1 - \delta}\id_{x \in B(o, n)},
		\]
		from which it follows that
		\[
			\sum_{x \in v(G)} f_n(G, o, x) = |\{x \in B(o, n): \hm_{x, o}(o) > 1 - \delta\}| = |\{x \in B(o, n): \hm_{x, o}(x) < \delta\}|.
		\]
		By the mass transport principle, and by assumption on the number of vertices which have $\hm_{x, o}(x) > 1 - \delta$, we deduce
		\[
			0 = \bf{E}\left[ \sum_{x \in v(G)}f_n(G, o, x) \right] = \bf{E} \left[ \sum_{x \in v(G)} f_n(G, x, o) \right].
		\]
		This implies the result since the expected value of a non-negative function can only equal zero if the function is a.s. equal to zero.
	\end{proof}
	}
	\end{wrong-old}

\begin{wrong-old}	
	\appendix
	\section{Adaption of Theorem \ref{theorem: Benjamini et al - no linear harmonics}.}

	Let $(G, o)$ denote a stationary environment, where we will denote $\mathbf{P}, \mathbf{E}$ the probability respectively expectation w.r.t. the total environment. Conditionally on a realization $(G, o)$, define the entropy of the random walk started at $o$ through
	\[
		H_n(G, o) = \sum_{x \in v(G)} \phi(\PP_o(X_n = x)),  \qquad \phi(t) = -t \log(t)
	\]
	with the convention that $\phi(0) = 0$. Denote the mean entropy through
	\[
		\mathbf{H}_n = \mathbf{E}[H_n(G, o)].
	\]
	We also need to introduce a way to measure the distance between two probability measures $\mu, \nu$ on the same set $\Omega$. In \cite{benjamini2015}, the following way to measure distances between two measures was introduced:
	\[
		\Delta(\mu, \nu) = \left[\sum_{\omega \in \Omega} \frac{(\mu(x) - \nu(x))^2}{\mu(x) + \nu(x)}\right]^{\frac{1}{2}}.
	\]
	The total variation distance then satisfies $\|\mu - \nu\|_{TV} \leq \sqrt{2}\Delta(\nu, \mu)$. This implies that for any $f: \Omega \to \RR$ one has
	\begin{equation} \label{eq: bounding integrals w.r.t. distance in measures}
		|\nu(f) - \mu(f)| \leq \Delta(\mu, \nu)(\mu(f^2) + \nu(f^2))^{\frac{1}{2}}
	\end{equation}
	by Cauchy-Schwartz.
	
	We will follow the notation in \cite{benjamini2015}. For a measurable event $C$, let $\mathscr{L}(X_n | C)$ denote the law of the random walk after $n$ steps, conditioned on $C$. We will write
	\begin{align*}
		\Delta_n(x, y):&= \Delta \big( \mathscr{L}(X_n | X_0 = x), \mathscr{L}(X_{n - 1} | X_0 = y)\big) \\
		&= 	\Delta \big( \mathscr{L}(X_n | X_0 = x), \mathscr{L}(X_{n} | X_1 = y)  \big).
	\end{align*}
	We recall the following result of \cite{benjamini2015}, their Theorem 7.
	\begin{theorem} \label{theorem 7 in Benjamini et al}
		Let $(G, o)$ be a stationary environment. For every $n \geq 1$, we have
		\[
			\mathbf{E}\big[ \Delta_n(o, X_1)^2 \big] \leq 2( \mathbf{H}_n - \mathbf{H}_{n - 1}).
		\]
	\end{theorem}

	\begin{proof}[Proof of Theorem \ref{theorem: Benjamini et al - no linear harmonics}. ]
		We stay very close to the proof in \cite{benjamini2015}. Suppose that $(G, o)$ is strictly subdiffusive, that is, there is some $\beta > 2$ for which a.s.
		\[
			\E[d(o, X_n)^\beta] \leq Cn.
		\]
		We need to prove that $h(o) = h(X_1)$ a.s., whenever $h$ is a (sub)linear harmonic function, in which case stationarity implies that $h(X_n) = h(X_{n + 1})$ for all $n$, and hence (as the underlying graph is connected), then $h$ is constant a.s.
		
		Let $h$ be harmonic w.r.t. $(G, o)$, then
		\[
			h(o) = \E_o[h(X_n)] \quad \text{and} \quad h(X_1) = \E_{X_1}[h(X_{n - 1})].
		\]
		Thus, by \eqref{eq: bounding integrals w.r.t. distance in measures} we get
		\begin{align} \label{subeq: proof of benj theorem 1}
			|h(o) - h(X_1)| &= \big| \E_o[h(X_n)] - \E_{X_1}[h(X_{n-1})] \big| \nonumber \\
			&\leq \Delta_n(o, X_1)\sqrt{\E_o[h^2(X_n)] + \E_{X_1}[h^2(X_{n - 1})]}.
		\end{align}
		In particular, taking the average over $X_1$, we deduce by Cauchy-Schwartz
		\[
			\E_o[|h(o) - h(X_1)|] \leq \sqrt{2\E_o[\Delta_n(o, X_1)^2]\E_o[h^2(X_n)]},
		\]
		where it is used that $\E_{o}[\E_{X_1}[h^2(X_{n - 1})]] = \E_o[h^2(X_n)]$. Assume that $h$ is at most linear. Then, for $\beta > 2$ as assumed and each $\epsilon > 0$, there exists a constant $K$ such that for all $x \in v(G)$ we have
		\begin{equation} \label{subeq: proof of benj theorem 2}
			h^2(x) \leq \epsilon d(o, x)^{\beta} + K.
		\end{equation}
		Since the chain is also assumed to have annealed polynomial growth, the entropy satisfies:
		\[
			\mathbf{H}_n \leq \mathbf{E}[\log|B(o, n)|] \leq \log \mathbf{E}[|B(o, n)|] \leq \log[Cn^d]
		\]
		for some $d$. Thus $\mathbf{H}_n - \mathbf{H}_{n-1} \leq c / n$ for infinitely many $n$. Using Theorem \ref{theorem 7 in Benjamini et al} we the subdiffusive behavior, we get
		\[
			\mathbf{E}[n\Delta_n(o, X_1)^2] + \mathbf[n^{-1}d(o, X_n)^{\beta}] \leq \mathbf{C}
		\]
		for infinitely many $n$. Using Fatou's Lemma, we deduce that for almost each environment, there must exist a (random) $c_1 < \infty$ such that
		\begin{equation} \label{subeq: proof of benj theorem 3}
			\E_o[n\Delta_n(o, X_1)^2] + \E_o[n^{-1}d(X_n, o)^{\beta}] < c_1
		\end{equation}
		for infinitely many $n$, where we have used the (strict) subdiffusivity assumption. Putting \eqref{subeq: proof of benj theorem 2} and \eqref{subeq: proof of benj theorem 3} into \eqref{subeq: proof of benj theorem 1}, together with the trivial bound $\sqrt{2ab} \leq |a + b|$, we get for almost every environment and $h$ harmonic and (sub)linear on it that
		\[
			\E_o[|h(o) - h(X_1)|] \leq c_2 \epsilon^{\frac{1}{2}}.
		\]
		Letting $\epsilon \downarrow 0$ we deduce that $h(o) = h(X_1)$ a.s., showing the desired result.

	\end{proof}
\end{wrong-old}

	\bibliography{bibliography/biblio}
\end{document}